%% file: HyperSLE.tex
\pdfoutput=1
\documentclass[11pt]{article}
\usepackage{authblk}
\usepackage[toc,page]{appendix}
\usepackage{color}
\usepackage{helvet}         
\usepackage{courier}        
\usepackage{type1cm}        
\usepackage{makeidx}         
\usepackage{graphicx}        
\usepackage{multicol}        
\usepackage[bottom]{footmisc}
\usepackage{amsmath}
\usepackage{amssymb}
\usepackage{bbold}
\usepackage{amsthm}
\usepackage{subcaption}
\usepackage{sidecap}
\usepackage{comment}
\usepackage[font=small]{caption}
\newtheorem{theorem}{Theorem}
\newtheorem{corollary}[theorem]{Corollary}

\newtheorem{lemma}[theorem]{Lemma}
\newtheorem*{conjecture}{Conjecture}
\newtheorem{proposition}[theorem]{Proposition}

\newtheorem{definition}[theorem]{Definition}
\usepackage[top=2cm, bottom=2cm, left=2cm, right=2cm]{geometry}
\numberwithin{theorem}{section}
\numberwithin{figure}{section}
\numberwithin{equation}{section}

\DeclareMathOperator{\dist}{dist}
\DeclareMathOperator{\SLE}{SLE}
\DeclareMathOperator{\hSLE}{hSLE}

\DeclareMathOperator{\GFF}{GFF}

\DeclareMathOperator{\free}{free}
\DeclareMathOperator{\hF}{{}_2F_1}
\DeclareMathOperator{\simple}{simple}
\DeclareMathOperator{\LP}{LP}
\DeclareMathOperator{\removal}{/}

\begin{document}

\title{Hypergeometric SLE: \\Conformal Markov Characterization and Applications}

\author{Hao Wu\thanks{H.W. is supported by the startup funding no. 042-53331001017 of Tsinghua
University.
Part of this work was done while H.W. was at Geneva University under the support of the NCCR/SwissMAP, the ERC AG COMPASP, the Swiss NSF. 
Email: hao.wu.proba@gmail.com}}
\affil{Yau Mathematical Sciences Center, Tsinghua University, China}

\date{}
\maketitle
\vspace{-1cm}
\begin{center}
\begin{minipage}{0.8\textwidth}
\abstract{This article pertains to the classification of pairs of simple random curves with conformal Markov property and symmetry. We give the complete classification of such curves: conformal Markov property and symmetry single out a two-parameter family of random curves---Hypergeometric SLE---denoted by $\hSLE_{\kappa}(\nu)$ for $\kappa\in (0,4]$ and $\nu<\kappa-6$. The proof relies crucially on Dub\'{e}dat's commutation relation \cite{DubedatCommutationSLE} and a uniqueness result proved in \cite{MillerSheffieldIG2}.
The classification indicates that hypergeometric SLE is the only possible scaling limit of the interfaces in critical lattice models (conjectured or proved to be conformal invariant) in topological rectangles with alternating boundary conditions. 

We also prove various properties of $\hSLE$: continuity, reversibility, target-independence, and conditional law characterization. As by-products, we give two applications of these properties. The first one is about the critical Ising interfaces. We prove the convergence of the Ising interface in rectangles with alternating boundary conditions. This result was first proved by Izyurov\cite{IzyurovObservableFree}, but our proof is new which is based on the properties of $\hSLE$. The second application is the existence of the so-called pure partition functions of multiple SLEs. Such existence was proved for 
$\kappa\in (0,8)\setminus \mathbb{Q}$ in \cite{KytolaPeltolaPurePartitionFunctions}, and it was later proved for $\kappa\in (0,4]$ in \cite{PeltolaWuGlobalMultipleSLEs}. We give a new proof of the existence for $\kappa\in (0,6]$ using the properties of $\hSLE$. 
\smallbreak
}
\end{minipage}
\end{center}

\newpage
\tableofcontents
\newpage
\newcommand{\eps}{\epsilon}
\newcommand{\ov}{\overline}
\newcommand{\U}{\mathbb{U}}
\newcommand{\T}{\mathbb{T}}
\newcommand{\HH}{\mathbb{H}}
\newcommand{\LA}{\mathcal{A}}
\newcommand{\LB}{\mathcal{B}}
\newcommand{\LC}{\mathcal{C}}
\newcommand{\LD}{\mathcal{D}}
\newcommand{\LF}{\mathcal{F}}
\newcommand{\LK}{\mathcal{K}}
\newcommand{\LE}{\mathcal{E}}
\newcommand{\LG}{\mathcal{G}}
\newcommand{\LL}{\mathcal{L}}
\newcommand{\LM}{\mathcal{M}}
\newcommand{\LQ}{\mathcal{Q}}
\newcommand{\CP}{\mathcal{P}}
\newcommand{\LT}{\mathcal{T}}
\newcommand{\LS}{\mathcal{S}}
\newcommand{\LU}{\mathcal{U}}
\newcommand{\LV}{\mathcal{V}}
\newcommand{\PartF}{\mathcal{Z}}
\newcommand{\LH}{\mathcal{H}}
\newcommand{\R}{\mathbb{R}}
\newcommand{\C}{\mathbb{C}}
\newcommand{\N}{\mathbb{N}}
\newcommand{\Z}{\mathbb{Z}}
\newcommand{\E}{\mathbb{E}}
\newcommand{\PP}{\mathbb{P}}
\newcommand{\QQ}{\mathbb{Q}}
\newcommand{\A}{\mathbb{A}}
\newcommand{\one}{\mathbb{1}}
\newcommand{\bn}{\mathbf{n}}
\newcommand{\MR}{MR}
\newcommand{\cond}{\,|\,}
\newcommand{\la}{\langle}
\newcommand{\ra}{\rangle}
\newcommand{\tree}{\Upsilon}
\global\long\def\chamber{\mathfrak{X}}

\section{Introduction}
\input{tex/intro}

\section{Preliminaries}
\label{sec::pre}
\subsection{Space of Curves}
\input{tex/pre_space_curves}
\subsection{Loewner Chain and $\SLE$}
\input{tex/pre_loewner_sle}
\subsection{Convergence of Curves}
\label{subsec::pre_cvg_curves}
\input{tex/pre_cvg_curves}
\subsection{Conformal Markov Characterization of $\SLE_{\kappa}(\rho)$}
\input{tex/pre_cmp_sle}
\subsection{SLE with Multiple Force Points}
\input{tex/pre_sle_multiple_forcepoints}
\section{Hypergeometric SLE: Continuity and Reversibility}
\label{sec::hypersle}
\subsection{Definition of $\hSLE$}
\input{tex/hypersle_def}

\subsection{Continuity of $\hSLE$: $\nu>(-4)\vee(\kappa/2-6)$}
\label{subsec::hypersle_continuity}
\input{tex/hypersle_continuity}

\subsection{Reversibility of $\hSLE$}
\input{tex/hypersle_reversibility}
\subsection{Continuity of $\hSLE$: $\nu\le (-4)\vee(\kappa/2-6)$}
\label{subsec::hypersle_continuity_lownu}
\input{tex/hypersle_continuity_lownu}

\subsection{Relation between Different $\hSLE$'s}
\input{tex/hypersle_relation_different}

\section{Hypergeometric SLE: Conformal Markov Characterization}
\label{sec::slepairs}
\subsection{Definitions and Statements}
\input{tex/slepair_statements}

\subsection{Proof of Proposition~\ref{prop::slepair_rev}}
\input{tex/slepair_proof_rev}
\subsection{Commutation Relation}
\input{tex/slepair_commutation}
\subsection{Proof of Proposition~\ref{prop::slepair_cmp_sym}}
\input{tex/slepair_proof_sym}

\subsection{Proof of Theorem~\ref{thm::SLEpair_CMP_SYM} and Corollary \ref{cor::SLEpair_CMP_REV}}
\input{tex/slepair_cmp}

\section{Convergence of Ising Interfaces to Hypergeometric SLE}
\label{sec::ising}
\input{tex/ising}
\subsection{Proof of Proposition~\ref{prop::ising_hypersle}}
\input{tex/ising_proof}

\section{Pure Partition Functions of Multiple SLEs}
\label{sec::purepartitionfunctions}
\input{tex/purepartitionfunctions_pre}

\subsection{Proof of Proposition~\ref{prop::purepartition_existence}}

\input{tex/purepartitionfunctions_proof}

\subsection{Proof of Lemma~\ref{lem::purepartition_cascade_SYM}}
\label{subsec::proof_SYM}
\input{tex/purepartitionfunctions_SYM}

\appendix
\section{Appendix: Hypergeometric Functions}
\label{sec::appendix}
\input{tex/appendix_hypergeometric_functions}


{\small
\newcommand{\etalchar}[1]{$^{#1}$}

}

\end{document}

%% file: tex/intro.tex
Conformal invariance and critical phenomena in two-dimensional lattice models play the central role in mathematical physics in the last few decades. We take Ising model as an example (see details in Section~\ref{sec::ising}). Suppose $\Omega$ is a simply connected domain and $x,y$ are distinct boundary points. When one considers the critical Ising model in $\Omega\cap\Z^2$ with Dobrushin boundary conditions: $\oplus$ along the boundary arc $(xy)$ and $\ominus$ along the boundary arc $(yx)$, an interface from $x$ to $y$ appears naturally which separates $\oplus$-spin from $\ominus$-spin. The scaling limit of the interface, if exists, is believed to satisfy conformal invariance and domain Markov property. We call the combination of the two as \textit{conformal Markov property}. 
Thus, to understand the scaling limit of interfaces in critical lattice model, one needs to understand random curves with conformal Markov property. 

In \cite{SchrammFirstSLE}, O. Schramm introduced $\SLE$ which is a random growth process in simply connected domain starting from one boundary point to another boundary point. This is a one-parameter family of random curves, denoted by $\SLE_{\kappa}$ with $\kappa\ge 0$. This family is the only one with conformal Markov property, and is conjectured to be the scaling limits of interface in critical models. Since its introduction, this conjecture has been rigorously proved for several models: percolation~\cite{SmirnovPercolationConformalInvariance, CamiaNewmanPercolation},
loop-erased random walk and uniform spanning tree 
\cite{LawlerSchrammWernerLERWUST}, 
level lines of the discrete Gaussian free 
field~\cite{SchrammSheffieldDiscreteGFF, SchrammSheffieldContinuumGFF}, and
the critical Ising and FK-Ising models~\cite{ChelkakSmirnovIsing, CDCHKSConvergenceIsingSLE}. 

$\SLE$ process corresponds to the scaling limit of interface in critical model with Dobrushin boundary condition. It is natural to consider critical model with more complicate boundary conditions. In this article, we focus on the alternating boundary condition in topological rectangles (\textit{quads} for short). We take Ising model as an example again. Suppose $\Omega$ is a simply connected domain and $x^R, y^R, y^L, x^L$ are four distinct boundary points in counterclockwise order. Consider critical Ising model in $\Omega\cap\Z^2$ with alternating boundary condition: $\oplus$ along the boundary arcs $(x^Ry^R)$ and $(y^Lx^L)$, and $\ominus$ along the arcs $(x^Lx^R)$ and $(y^Ry^L)$. With this boundary condition, a pair of interfaces appears naturally. This pair of interfaces connects between the four points $x^R, y^R, y^L, x^L$ and the two interfaces can not cross, see Figure~\ref{fig::ising_pair}. The scaling limit of the pair of interfaces, if exists, should satisfy conformal Markov property (see Definition~\ref{def::CMP_slepair}). 
This article concerns probability measures on pairs of simple curves with conformal Markov property, and they should describe scaling limits of pairs of interfaces in critical lattice model with alternating boundary condition in quads. 

\begin{figure}[ht!]
\begin{center}
\includegraphics[width=0.7\textwidth]{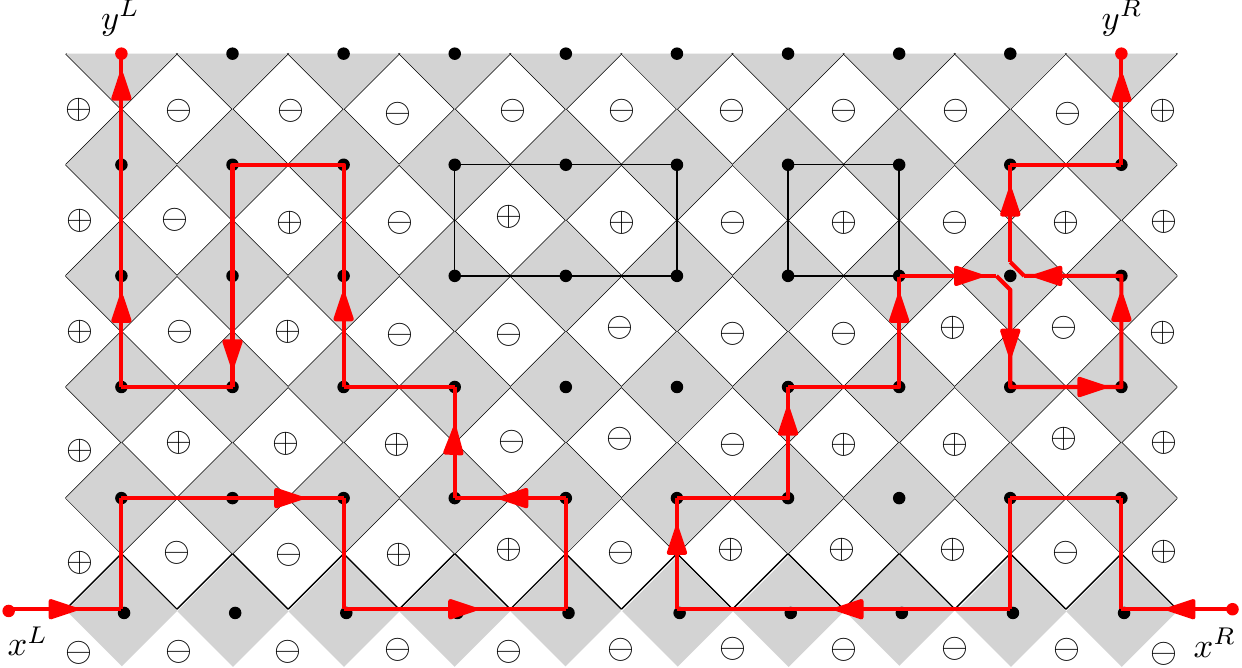}
\end{center}
\caption{\label{fig::ising_pair} The Ising interface with alternating boundary condition. }
\end{figure}

In the case of Dobrushin boundary condition, there are two boundary points, and conformal Markov property determines the one-parameter family of random curves $\SLE_{\kappa}$. However, in the case of alternating boundary conditions in quads, there are four boundary points, and conformal Markov property is not sufficient to naturally single out random processes. We go back to the critical Ising model. As described before, there is a pair of interfaces when the boundary condition is alternating. The scaling limit of such pair should satisfy conformal Markov property; at the same time, it is clear that the pair of curves also satisfy a particular symmetry (see Definition~\ref{def::sym}). To understand the scaling limit of such pair, it is then natural to require the symmetry as well as conformal Markov property.

It turns out that the combination of conformal Markov property and symmetry determines a two-parameter family of pairs of curves. These curves are \textit{hypergeometric SLE}s. 

\subsection{Hypergeometric SLE}

Hypergeometric SLEs is a two-parameter family of random curves in quad. The two parameters are $\kappa\in (0,8)$ and $\nu\in\R$, and we denote it by $\hSLE_{\kappa}(\nu)$. It comes naturally as the scaling limit of interfaces in critical planar models with alternating boundary condition in quad, and it is conformally invariant.  
We will give definition of this process in Section~\ref{sec::hypersle}, and the main theorem of Section~\ref{sec::hypersle} is continuity and reversibility of hypergeometric SLEs. 

\begin{theorem} \label{thm::hyperSLE_reversibility}
Fix $\kappa\in (0,8), \nu>(-4)\vee(\kappa/2-6)$, and $x_1<x_2<x_3<x_4$. Let $\eta$ be the $\hSLE_{\kappa}(\nu)$ in $\HH$ from $x_1$ to $x_4$ with marked points $(x_2, x_3)$. The process $\eta$ is almost surely generated by a continuous transient curve. Moreover, the process $\eta$ enjoys reversibility for $\nu\ge\kappa/2-4$: the time reversal of $\eta$ is the $\hSLE_{\kappa}(\nu)$ in $\HH$ from $x_4$ to $x_1$ with marked points $(x_3, x_2)$.  
\end{theorem}

Here we briefly summarize the relation between $\hSLE$ and $\SLE_{\kappa}$ (or  $\SLE_{\kappa}(\rho)$) process. Fix $x_1=0<x_2<x_3<x_4=\infty$. 
Suppose $\eta$ is $\hSLE_{\kappa}(\nu)$ in $\HH$ from $0$ to $\infty$ with marked points $(x_2, x_3)$. 
\begin{itemize}
\item When $\nu=-2$, the law of $\eta$ equals $\SLE_{\kappa}$.
\item When $\kappa=4$, the law of $\eta$ equals $\SLE_4(\nu+2, -\nu-2)$ with force points $(x_2, x_3)$.
\item When $x_3\to \infty$, the law of $\eta$ converges weakly to the law of $\SLE_{\kappa}(\nu+2)$ with force point $x_2$. See Lemma~\ref{lem::hSLE_degenerate}. 
\item When $\kappa\in (4,8)$ and $\nu=\kappa-6$, the law of $\eta$ equals the law of $\SLE_{\kappa}$ conditioned to avoid the interval $(x_2, x_3)$. See Proposition~\ref{prop::hSLE_conditioned}. 
\end{itemize} 
From these relations, we see that $\hSLE_{\kappa}(\nu)$ is a generalization of $\SLE_{\kappa}(\rho)$ process. In general, the driving function of $\hSLE$ has a drift term which involves a hypergeometric function. When $\kappa=4$, the hypergeometric term becomes zero, and the process coincides with $\SLE_4(\nu+2, -\nu-2)$ process.

\subsection{Conformal Markov Characterization}
We denote by $\LQ$ the collection of all quads, and for each quad $q=(\Omega; x^R, y^R, y^L, x^L)$, we denote by $X_0(\Omega; x^R, y^R, y^L, x^L)$ the collection of pairs of disjoint simple curves $(\eta^L;\eta^R)$ such that $\eta^R$ connects $x^R$ and $y^R$ and $\eta^L$ connects $x^L$ and $y^L$. 
The following definitions concern confomral Markov property and symmetry for pairs of simple curves. See also Figure~\ref{fig::def_CMP} for an illustration. 

\begin{definition}\label{def::CMP_slepair}
Suppose $(\PP_q, q\in\LQ)$ is a family of probability measures on pairs of disjoint simple curves $(\eta^L; \eta^R)\in X_0(\Omega; x^R, y^R, y^L, x^L)$. We say that $(\PP_q, q\in\LQ)$ satisfies conformal Markov property (CMP) if it satisfies the following two properties.
\begin{itemize}
\item Conformal invariance. Suppose $q=(\Omega; x^R, y^R, y^L, x^L), \tilde{q}=(\tilde{\Omega}; \tilde{x}^R, \tilde{y}^R, \tilde{y}^L, \tilde{x}^L)\in\LQ$, and $\psi: \Omega\to \tilde{\Omega}$ is a conformal map with $\psi(x^R)=\tilde{x}^R, \psi(y^R)=\tilde{y}^R, \psi(y^L)=\tilde{y}^L, \psi(x^L)=\tilde{x}^L$. Then for $(\eta^L; \eta^R)\sim \PP_q$, we have $(\psi(\eta^L); \psi(\eta^R))\sim \PP_{\tilde{q}}$. 
\item Domain Markov property. Suppose $(\eta^L;\eta^R)\sim\PP_q$, then for every $\eta^L$-stopping time $\tau^L$ and $\eta^R$-stopping time $\tau^R$, the conditional law of $(\eta^L|_{t\ge \tau^L}; \eta^R|_{t\ge \tau^R})$ given $\eta^L[0,\tau^L]$ and $\eta^R[0,\tau^R]$ is the same as $\PP_{q_{\tau^L, \tau^R}}$ where \[q_{\tau^L, \tau^R}=(\Omega\setminus (\eta^L[0,\tau^L]\cup\eta^R[0,\tau^R]); \eta^R(\tau^R), y^R, y^L, \eta^L(\tau^L)).\] 
\end{itemize}
\end{definition}

\begin{figure}[ht!]
\begin{center}
\includegraphics[width=0.8\textwidth]{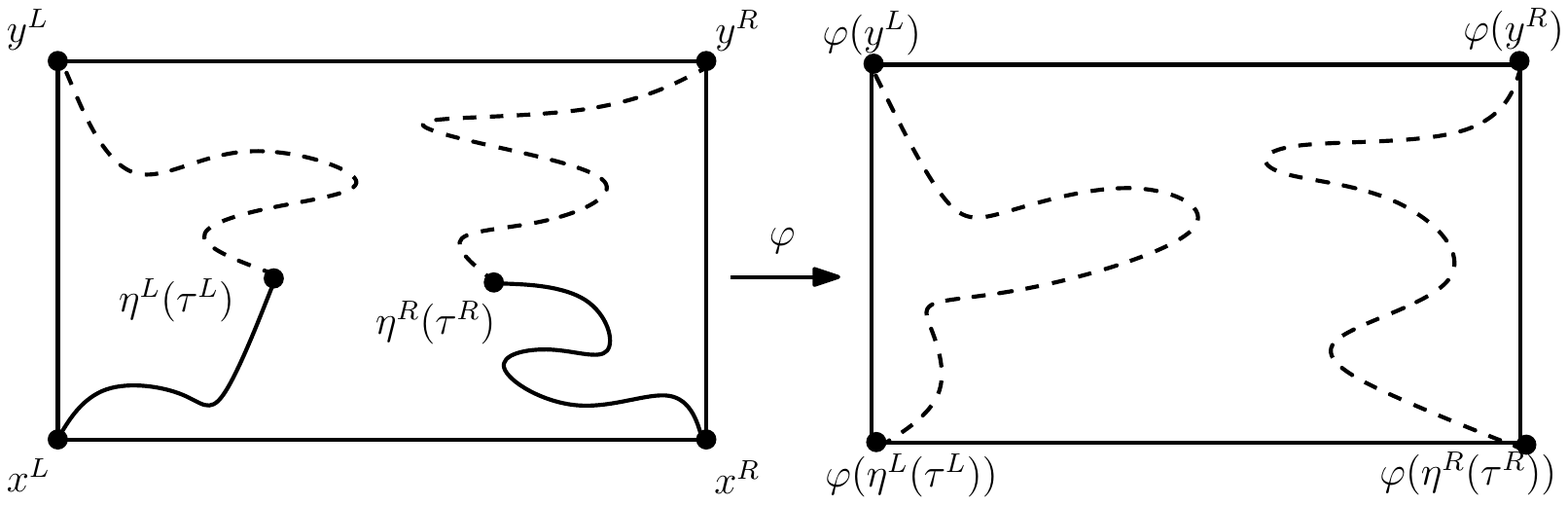}
\end{center}
\caption{\label{fig::def_CMP} Suppose the pair $(\eta^L;\eta^R)$ satisfies CMP. For any $\eta^L$-stopping time $\tau^L$ and any $\eta^R$-stopping time $\tau^R$, let $\varphi$ be a conformal map from $\Omega\setminus(\eta^L[0,\tau^L]\cap\eta^R[0,\tau^R])$ onto a quad $\tilde{q}=(\tilde{\Omega}; \tilde{x}^R, \tilde{y}^R, \tilde{y}^L, \tilde{x}^L)$ such that $\varphi(\eta^R(\tau^R))=\tilde{x}^R, \varphi(y^R)=\tilde{y}^R, \varphi(y^L)=\tilde{y}^L, \varphi(\eta^L(\tau^L))=\tilde{x}^L$. Then the conditional law of $(\varphi(\eta^L); \varphi(\eta^R))$ given $\eta^L[0,\tau^L]\cup\eta^R[0,\tau^R]$ is the same as $\PP_{\tilde{q}}$. }
\end{figure}

\begin{definition}\label{def::sym}
Suppose $(\PP_q, q\in\LQ)$ is a family of probability measures on pairs of disjoint simple curves $(\eta^L; \eta^R)\in X_0(\Omega; x^R, y^R, y^L, x^L)$. We say that $(\PP_q, q\in\LQ)$ satisfies symmetry if for all $q=(\Omega; x^R, y^R, y^L, x^L)\in\LQ$ the following is true. Suppose $(\eta^L;\eta^R)\sim\PP_q$, and $\psi: \Omega\to \Omega$ is the anti-conformal map which swaps $x^L, y^L$ and $x^R, y^R$, then $(\psi(\eta^R); \psi(\eta^L))\sim\PP_q$. 
\end{definition}

It turns out that the combination of CMP and the symmetry determines a two-parameter family of pairs of curves---$\hSLE_{\kappa}(\nu)$. In Theorem~\ref{thm::SLEpair_CMP_SYM}, we consider pairs of random curves with CMP and the symmetry, and we also require ``Condition C1". This is a technical requirement concerning certain regularity of the curves and its definition is in Section~\ref{subsec::pre_cvg_curves}.

\begin{theorem}\label{thm::SLEpair_CMP_SYM}
Suppose $(\PP_q, q\in\LQ)$ satisfies CMP in Definition~\ref{def::CMP_slepair}, the symmetry in Definition~\ref{def::sym} and Condition C1. Then there exist $\kappa\in (0,4]$ and $\nu<\kappa-6$ such that, for $q=(\Omega; x^R, y^R, y^L, x^L)\in\LQ$ and $(\eta^L;\eta^R)\sim\PP_q$, the marginal law of $\eta^R$ (up to the first hitting time of $[y^Ry^L]$) equals $\hSLE_{\kappa}(\nu)$ in $\Omega$ from $x^R$ to $x^L$ with marked points $(y^R, y^L)$ conditioned to hit $[y^Ry^L]$ (up to the first hitting time of $[y^Ry^L]$). 
\end{theorem}

Definition~\ref{def::sym} describes the left-right symmetry. It is also natural to think about top-bottom symmetry, which we call reversibility in Definition~\ref{def::rev} to distinguish from the symmetry in Definition~\ref{def::sym}. The combination of CMP, the symmetry and the reversibility singles out a one-parameter family of pairs of curves, see Corollary~\ref{cor::SLEpair_CMP_REV}.

\begin{definition}\label{def::rev}
Suppose $(\PP_q, q\in\LQ)$ is a family of probability measures on pairs of disjoint simple curves $(\eta^L; \eta^R)\in X_0(\Omega; x^R, y^R, y^L, x^L)$. We say that $(\PP_q, q\in\LQ)$ satisfies reversibility if for all $q=(\Omega; x^R, y^R, y^L, x^L)\in\LQ$ the following is true. Suppose $(\eta^L;\eta^R)\sim\PP_q$, and $\psi: \Omega\to \Omega$ is the anti-conformal map which swaps $x^L, x^R$ and $y^L, y^R$, then $(\psi(\eta^R); \psi(\eta^L))\sim\PP_q$.
\end{definition}

\begin{corollary}\label{cor::SLEpair_CMP_REV}
Suppose $(\PP_q, q\in\LQ)$ satisfies CMP in Definition~\ref{def::CMP_slepair}, the symmetry in Definition~\ref{def::sym}, the reversibility in Definition~\ref{def::rev} and Condition C1. Then there exists $\kappa\in (0,4]$ such that, for any $q=(\Omega; x^R, y^R, y^L, x^L)\in\LQ$ and $(\eta^L;\eta^R)\sim\PP_q$, the marginal law of $\eta^R$ equals $\hSLE_{\kappa}$ in $\Omega$ from $x^R$ to $y^R$ with marked points $(x^L, y^L)$.
\end{corollary}

\subsection{Convergence of Critical Planar Ising Interfaces}

Let us go back to the critical Ising model. We take Ising model as an example to explain the interest in pairs of random curves and the
motivation for the definition of conformal Markov property and symmetries. We find that the combination of conformal Markov property and symmetries singles out hypergeometric SLEs. In this section, we point out that hypergeometric SLE DOES correspond to the scaling limit of critical Ising model with alternating boundary conditions. 

\begin{proposition}\label{prop::ising_hypersle}
Let discrete quads $(\Omega_{\delta}; x^R_{\delta}, y^R_{\delta}, y^L_{\delta}, x^L_{\delta})$ on $\delta\Z^2$ approximate some quad $q=(\Omega; x^R, y^R, y^L, x^L)$ as $\delta\to 0$. Consider the critical Ising in $\Omega_{\delta}$ with the following boundary condition: 
\[\ominus\text{ along }(x^L_{\delta}x^R_{\delta}), \quad \oplus\text{ along }(x^R_{\delta}y^R_{\delta})\cup (y^L_{\delta}x^L_{\delta}), \quad \xi\in \{\oplus, \free\} \text{ along } (y^R_{\delta}y^L_{\delta}).\]
Denote by $\LC_v^{\ominus}(q)$ the event that the quad is vertically crossed by $\ominus$ and by $\LC_h^{\oplus}(q)$ the event that the quad is horizontally crossed by $\oplus$. 
\begin{itemize}
\item Suppose $\xi=\oplus$. On the event $\LC_v^{\ominus}(q)$, let $\eta_{\delta}$ be the interface connecting $x^R_{\delta}$ and $y^R_{\delta}$. Then the law of $\eta_{\delta}$ converges weakly to $\hSLE_{3}$ in $\Omega$ from $x^R$ to $y^R$ with marked points $(x^L, y^L)$.
\item Suppose $\xi=\free$. On the event $\LC_v^{\ominus}(q)$, let $\eta_{\delta}$ be the interface connecting $x^R_{\delta}$ and $y^R_{\delta}$. Then the law of $\eta_{\delta}$ (up to the first hitting time of $[y^R_{\delta}y^L_{\delta}]$) converges weakly to $\hSLE_3(-7/2)$ in $\Omega$ from $x^R$ to $x^L$ conditioned to hit $[y^Ry^L]$ (up to the first hitting time of $[y^Ry^L]$).
\item Suppose $\xi=\free$. On the event $\LC_h^{\oplus}(q)$, let $\eta_{\delta}$ be the interface connecting $x^R_{\delta}$ and $x^L_{\delta}$. Then the law of $\eta_{\delta}$ converges weakly to $\hSLE_3(-3/2)$ in $\Omega$ from $x^R$ to $x^L$ with marked points $(y^R, y^L)$.
\end{itemize}
\end{proposition}

The conclusions in Proposition~\ref{prop::ising_hypersle} are not new. They were proved by K.~Izyurov \cite{IzyurovObservableFree}, and we will give a new proof. There are three features on the method developed in Section~\ref{sec::ising}. 
\begin{itemize}
\item First, no need to construct new observable. Constructing holomorphic observable is the usual way to prove the convergence of interfaces in the critical lattice model (as in \cite{IzyurovObservableFree}); however, with our method, there is no need to construct new observable. The only input we need is the convergence of the interface with Dobrushin boundary condition.
\item Second, the result is ``global". There are many works on multiple SLEs trying to study the scaling limit of interfaces in critical lattice model with alternating boundary conditions, see \cite{DubedatCommutationSLE, BauerBernardKytolaMultipleSLE, KytolaPeltolaPurePartitionFunctions, IzyurovObservableFree}, and their works study the local growth of these interfaces. Whereas, our result is ``global": we prove the convergence of the entire interface. 
\item Third, easy to generalize. Our method can be generalized to more complicated boundary conditions, and the method also works for other critical lattice models including loop-erased random walk, FK-Ising model and percolation, see \cite{BeffaraPeltolaWuUniqueness}.
\end{itemize}

\subsection{Pure Partition Functions of Multiple SLEs}
In Theorem~\ref{thm::SLEpair_CMP_SYM}, we have shown the classification of the interfaces in critical lattice model in quad with alternating boundary conditions. It is natural to consider the interfaces in general polygon. We call $(\Omega; x_1, \ldots, x_{2N})$ a polygon if $\Omega\subsetneq\C$ is simply connected and $x_1, \ldots, x_{2N}$ are $2N$ boundary points in counterclockwise order. We take Ising model as an example again. Suppose $(\Omega^{\delta}; x_1^{\delta}, \ldots, x_{2N}^{\delta})$ are discrete domains on $\delta\Z$ that approximate some polygon $(\Omega; x_1, \ldots, x_{2N})$. Consider the critical Ising model in $\Omega^{\delta}$ with alternating boundary conditions: 
\[\oplus \text{ on }(x_{2j-1}^{\delta}, x_{2j}^{\delta}), \quad \text{for }j\in\{1, \ldots, N\};\quad \ominus \text{ on }(x_{2j}^{\delta}, x_{2j+1}^{\delta}),\quad \text{for }j\in\{0, 1, \ldots, N\},\]
with the convention that $x_{0}=x_{2N}$ and $x_{2N+1}=x_1$. Then $N$ interfaces $(\eta_1^{\delta}, \ldots, \eta_N^{\delta})$ arise in the model and they connect the $2N$ boundary points $x_1^{\delta}, \ldots, x_{2N}^{\delta}$, forming a planar connectivity. We describe the connectivities by planar pair partitions $\alpha=\{\{a_1, b_1\}, \ldots, \{a_N, b_N\}\}$ where $\{a_1, b_1, \ldots, a_N, b_N\}=\{1,2,\ldots, 2N\}$. We call such $\alpha$ link patterns and we denote the set of them by $\LP_N$. We denote $\LP=\sqcup_{N\ge 0}\LP_N$.  Given a link pattern $\alpha\in\LP_N$ and $\{a,b\}\in\alpha$, we denote by $\alpha\removal\{a,b\}$ the link pattern in $\LP_{N-1}$ obtained by removing $\{a,b\}$ from $\alpha$ and then relabelling the remaining indices so that they are the first $2(N-1)$ integers. 

It turns out that the scaling limit of the collection $(\eta_1^{\delta}, \ldots, \eta_N^{\delta})$ are the Loewner chains associated to the pure partition functions whose definition is given below. See \cite{BauerBernardKytolaMultipleSLE, DubedatCommutationSLE, KytolaPeltolaPurePartitionFunctions, PeltolaWuGlobalMultipleSLEs, BeffaraPeltolaWuUniqueness} for the background. Fix $\kappa\in (0,8)$, multiple SLE \textit{pure partition functions} is a collection of positive smooth functions 
\[\PartF_{\alpha}: \chamber_{2N}:=\{(x_1, \ldots, x_{2N}): x_1<\cdots<x_{2N}\}\to \R_+,\quad \alpha\in\LP_N\]
satisfying the following three properties:
\begin{itemize}
\item PDE system (PDE): 
\begin{align}\label{eqn::purepartition_PDE}
\left[ \frac{\kappa}{2}\partial^2_i + \sum_{j\neq i}\left(\frac{2}{x_{j}-x_{i}}\partial_j - 
\frac{2h}{(x_{j}-x_{i})^{2}}\right) \right]
\PartF(x_1,\ldots,x_{2N}) =  0, \quad \text{for all } i \in \{1,\ldots,2N\} .
\end{align}
\item Conformal covariance (COV): for all M\"obius maps 
$\varphi$ of $\HH$ 
such that $\varphi(x_{1}) < \cdots < \varphi(x_{2N})$, 
\begin{align}\label{eqn::purepartition_COV}
\PartF(x_{1},\ldots,x_{2N}) = 
\prod_{i=1}^{2N} \varphi'(x_{i})^{h} 
\times \PartF(\varphi(x_{1}),\ldots,\varphi(x_{2N})), \quad \text{where } h = \frac{6-\kappa}{2\kappa}.
\end{align}
\item Asymptotics (ASY): 
for all $\alpha \in \LP_N$ and for all $j \in \{1, \ldots, 2N-1 \}$ and $\xi \in (x_{j-1}, x_{j+2})$, 
\begin{align}\label{eqn::purepartition_ASY}
\lim_{x_j , x_{j+1} \to \xi} 
\frac{\PartF_\alpha(x_1 , \ldots , x_{2N})}{(x_{j+1} - x_j)^{-2h}} 
=\begin{cases}
0 \quad &
    \text{if } \{j,j+1\} \notin \alpha \\
\PartF_{\hat{\alpha}}(x_{1},\ldots,x_{j-1},x_{j+2},\ldots,x_{2N}) &
    \text{if } \{j,j+1\} \in \alpha
\end{cases}
\end{align}
where
$\hat{\alpha} = \alpha \removal\{j,j+1\} \in \LP_{N-1}$.
\end{itemize}

Although the scaling limit of the interfaces in the critical lattice model in polygon leads to the pure partition functions, it is far from clear why such functions exist, and we will discuss the existence of such functions in the following theorem.

\begin{theorem}\label{thm::purepartition}
Let $\kappa \in (0,6]$. There exists a unique collection $\{\PartF_{\alpha}: \alpha\in \LP\}$ of smooth
functions $\PartF_\alpha : \chamber_{2N} \to \R_+$, for $\alpha \in \LP_N$,
satisfying the normalization
$\PartF_\emptyset = 1$ and PDE~\eqref{eqn::purepartition_PDE}, COV~\eqref{eqn::purepartition_COV}, ASY~\eqref{eqn::purepartition_ASY}
and, for all $\alpha =\{ \{a_1, b_1\}, \ldots, \{a_N, b_N\} \} \in \LP_N$, the power law bound
\begin{align}\label{eqn::purepartition_PLB}
0<\PartF_{\alpha}(x_1, \ldots, x_{2N}) 
\le \prod_{j=1}^N |x_{b_j}-x_{a_j}|^{-2h} .
\end{align}
\end{theorem}
The uniqueness is a deep result and it was proved in \cite{FloresKlebanPDE1} for all $\kappa\in (0,8)$. 
The existence part was proved for $\kappa\in (0,8)\setminus \mathbb{Q}$ in \cite{KytolaPeltolaPurePartitionFunctions} using Coulomb gas techniques. The difficulty with the Coulomb gas techniques is that the authors could not show the positivity of the constructed functions, neither the optimal bound~\eqref{eqn::purepartition_PLB}. The existence was later proved for $\kappa\in (0,4]$ in \cite{PeltolaWuGlobalMultipleSLEs} using the construction of global multiple SLEs. In this method, the positivity and the power law bound~\eqref{eqn::purepartition_PLB} is clear from the construction, but showing the smoothness requires certain work. 
In this paper, we will give a new proof of the existence for $\kappa\in (0,6]$ using properties of hypergeometric SLE. We will construct the pure partition by cascade relation and then show that they satisfies all the requirements. The main obstacle in this construction is checking the PDE, and this is obtained using properties of hypergeometric SLEs.  

\smallbreak
\noindent\textbf{Outline and relation to previous works.}
We will give preliminaries on SLEs in Section~\ref{sec::pre}. 
Hypergeometric SLEs were previously introduced by D.~Zhan \cite{ZhanReversibility} and W.~Qian \cite{QianConformalRestrictionTrichordal} with different motivations and different definitions. 
We will introduce hypergeometric SLE in Section~\ref{sec::hypersle} with our motivation, and the definition is different from theirs. 
We will prove Theorem~\ref{thm::hyperSLE_reversibility} in Section~\ref{sec::hypersle} and many other interesting properties of $\hSLE$. 
We prove Theorem~\ref{thm::SLEpair_CMP_SYM} in Section~\ref{sec::slepairs}. 
We introduce Ising model in Section~\ref{sec::ising} and prove Proposition~\ref{prop::ising_hypersle}. 
We complete the proof of Theorem~\ref{thm::purepartition} in Section~\ref{sec::purepartitionfunctions}. 
\smallbreak
\noindent\textbf{Acknowledgment.}
The author thanks D. Chelkak, K.Izyurov, and S. Smirnov for helpful discussion on critical Ising interfaces. 
The author thanks V. Beffara and E. Peltola for helpful discussion on multiple SLEs.

%% file: tex/pre_space_curves.tex
A planar curve is a continuous mapping from $[0,1]$ to $\C$ modulo reparameterization. Let $X$ be the set of planar curves. The metric $d$ on $X$ is defined by 
\[d(\eta_1,\eta_2)=\inf_{\varphi_1,\varphi_2}\sup_{t\in[0,1]}|\eta_1(\varphi_1(t))-\eta_2(\varphi_2(t))|,\]
where the inf is over increasing homeomorphisms $\varphi_1,\varphi_2:[0,1]\to [0,1]$. The metric space $(X, d)$ is complete and separable. A simple curve is a continuous injective mapping from $[0,1]$ to $\C$ modulo reparameterization. 
Let $X_{\simple}$ be the subspace of simple curves and denote by $X_0$ its closure. The curves in $X_0$ may have multiple points but they do not have self-crossings. 

We call $(\Omega; x,y)$ a \textit{Dobrushin domain} if $\Omega$ is a non-empty simply connected proper subset of $\C$ and $x,y$ are two distinct boundary points. Denote by $(xy)$ the arc of $\partial\Omega$ from $x$ to $y$ counterclockwise. We say that a sequence of Dobrushin domains $(\Omega_{\delta}; x_{\delta}, y_{\delta})$ converges to a Dobrushin domain $(\Omega; x, y)$ in the \textit{Carath\'eodory sense} if $f_{\delta}\to f$ uniformly on any compact subset of $\HH$ where $f_{\delta}$ (resp. $f$) is the unique conformal map from $\HH$ to $\Omega_{\delta}$ (resp. $\Omega$) satisfying $f_{\delta}(0)=x_{\delta}, f_{\delta}(\infty)=y_{\delta}$ and $f_{\delta}'(\infty)=1$ (resp. $f(0)=x, f(\infty)=y, f'(\infty)=1$). 

Given a Dobrushin domain $(\Omega; x,y)$, let $X_{\simple}(\Omega; x,y)$ be the space of simple curves $\eta$ such that \[\eta(0)=x, \quad\eta(1)=y,\quad\eta(0,1)\subset\Omega.\]
Denote by $X_0(\Omega; x,y)$ the closure of $X_{\simple}(\Omega; x,y)$. 

We call $(\Omega; x, w, y)$ a \textit{(topological) triangle} if $\Omega$ is a non-empty simply connected proper subset of $\C$ and $x, w, y$ are three boundary points along $\partial\Omega$ in counterclockwise order such that $x\neq y$. Let $\LT$ be the collection of all such triangles. We say that a sequence of triangles $(\Omega_{\delta}; x_{\delta}, w_{\delta}, y_{\delta})$ converges to a triangle $(\Omega; x, w, y)$ in the \textit{Carath\'eodory sense} if $f_{\delta}\to f$ uniformly on any compact subset of $\HH$ where $f_{\delta}$ (resp. $f$) is the unique conformal map from $\HH$ to $\Omega_{\delta}$ (resp. $\Omega$) satisfying $f_{\delta}(0)=x_{\delta}, f_{\delta}(1)=w_{\delta}$, and $f_{\delta}(\infty)=y_{\delta}$.(resp. $f(0)=x, f(1)=w, f(\infty)=y$).

We call $(\Omega; a, b, c, d)$ a \textit{quad (or topological rectangle)} if $\Omega$ a non-empty simply connected proper subset of $\C$ and $a, b, c, d$ are four boundary points in counterclockwise order such that $a\neq d$. Given a quad $(\Omega; a, b, c, d)$, we denote by $d_{\Omega}((ab), (cd))$ the extremal distance between $(ab)$ and $(cd)$ in $\Omega$. We say a sequence of quads $(\Omega_{\delta}; a_{\delta}, b_{\delta}, c_{\delta}, d_{\delta})$ converges to a quad $(\Omega; a, b, c, d)$ in the \textit{Carath\'{e}odory sense} if $f_{\delta}\to f$ uniformly on any compact subset of $\HH$ and  $\lim_{\delta}f_{\delta}^{-1}(b_{\delta})=f^{-1}(b)$ and $\lim_{\delta}f_{\delta}^{-1}(c_{\delta})=f^{-1}(c)$ 
where $f_{\delta}$ (resp. $f$) is the unique conformal map from $\HH$ to $\Omega_{\delta}$ (resp. $\Omega$) satisfying $f_{\delta}(0)=a_{\delta}, f_{\delta}(\infty)=d_{\delta}$ and $f_{\delta}'(\infty)=1$ (resp. $f(0)=a, f(\infty)=d, f'(\infty)=1$). 

Given a quad $(\Omega; x^L, x^R, y^R, y^L)$, let $X_{\simple}(\Omega; x^L, x^R, y^R, y^L)$ be the collection of pairs of simple curves $(\eta^L;\eta^R)$ such that $\eta^L\in X_{\simple}(\Omega; x^L, y^L)$ and $\eta^R\in X_{\simple}(\Omega; x^R, y^R)$ and that $\eta^L\cap\eta^R=\emptyset$. The definition of $X_0(\Omega; x^L, x^R, y^R, y^L)$ is a little bit complicate. 
Given a quad $(\Omega; x^L, x^R, y^R, y^L)$ and $\eps>0$, let $X_0^{\eps}(\Omega; x^L, x^R, y^R, y^L)$ be the set of pairs of curves $(\eta^L; \eta^R)$ such that 
\begin{itemize}
\item $\eta^L\in X_0(\Omega; x^L, y^L)$ and $\eta^R\in X_0(\Omega; x^R, y^R)$;
\item $d_{\Omega^L}(\eta^L, (x^Ry^R))\ge \eps$ where $\Omega^L$ the connected component of $\Omega\setminus \eta^L$ with $(x^Ry^R)$ on the boundary, 
and $\eta^R$ is contained in the closure of $\Omega^L$ ;
\item $d_{\Omega^R}(\eta^R, (y^Lx^L))\ge \eps$ where $\Omega^R$ the connected component of  $\Omega\setminus \eta^R$ with $(y^Lx^L)$ on the boundary, and $\eta^L$ is contained in the closure of $\Omega^R$.
\end{itemize}
Define the metric on $X_0^{\eps}(\Omega; x^L, x^R, y^R, y^L)$ by 
\[\LD((\eta^L_1, \eta^R_1), (\eta^L_2,\eta^R_2))=\max\{d(\eta^L_1,\eta^L_2), d(\eta^R_1,\eta^R_2)\}. \]
One can check $\LD$ is a metric and the space $X_0^{\eps}(\Omega; x^L, x^R, y^R, y^L)$ with $\LD$ is complete and separable. 
Finally, set 
\[X_0(\Omega; x^L, x^R, y^R, y^L)=\bigcup_{\eps>0}X_0^{\eps}(\Omega; x^L, x^R, y^R, y^L).\]
Note that $X_0(\Omega; x^L, x^R, y^R, y^L)$ is no longer complete. 
\smallbreak
Suppose $E$ is a metric space and $\LB_E$ is the Borel $\sigma$-field. Let $\CP$ be the space of probability measures on $(E, \LB_E)$. The Prohorov metric $d_{\CP}$ on $\CP$ is defined by  
\[d_{\CP}(\PP_1, \PP_2)=\inf\left\{\eps>0: \PP_1[A]\le \PP_2[A^{\eps}]+\eps, \PP_2[A]\le \PP_1[A^{\eps}]+\eps, \forall A\in\LB_E\right\}.\]
When $E$ is complete and separable, the space $\CP$ is complete and separable (\cite[Theorem 6.8]{BillingsleyConvergenceProbabilityMeasures}); moreover, a sequence $\PP_n$ in $\CP$ converges weakly to $\PP$ if and only if $d_{\CP}(\PP_n, \PP)\to 0$. 

Let $\Sigma$ be a family of probability measures on $(E,\LB_E)$. We call $\Sigma$ \textit{relatively compact} if every sequence of elements in $\Sigma$ contains a weakly convergent subsequence. We call $\Sigma$ \textit{tight} if, for every $\eps>0$, there exists a compact set $K_{\eps}$ such that $\PP[K_{\eps}]\ge 1-\eps$ for all $\PP\in\Sigma$. By Prohorov's Theorem (\cite[Theorem 5.2]{BillingsleyConvergenceProbabilityMeasures}), when $E$ is complete and separable, relative compactness is equivalent to tightness. 

%% file: tex/pre_loewner_sle.tex
We call a compact subset $K$ of $\overline{\HH}$ an \textit{$\HH$-hull} if $\HH\setminus K$ is simply connected. Riemann's Mapping Theorem asserts that there exists a unique conformal map $g_K$ from $\HH\setminus K$ onto $\HH$ such that
$\lim_{z\to\infty}|g_K(z)-z|=0$.
We call such $g_K$ the conformal map from $\HH\setminus K$ onto $\HH$ normalized at $\infty$ and we call $a(K):=\lim_{z\to \infty} z(g_t(z)-z)$ the \textit{half-plane capacity} of $K$. 

\textit{Loewner chain} is a collection of $\HH$-hulls $(K_{t}, t\ge 0)$ associated with the family of conformal maps $(g_{t}, t\ge 0)$ obtained by solving the Loewner equation: for each $z\in\mathbb{H}$,
\[\partial_{t}{g}_{t}(z)=\frac{2}{g_{t}(z)-W_{t}}, \quad g_{0}(z)=z,\]
where $(W_t, t\ge 0)$ is a one-dimensional continuous function which we call the driving function. Let $T_z$ be the \textit{swallowing time} of $z$ defined as $\sup\{t\ge 0: \min_{s\in[0,t]}|g_{s}(z)-W_{s}|>0\}$.
Let $K_{t}:=\overline{\{z\in\mathbb{H}: T_{z}\le t\}}$. Then $g_{t}$ is the unique conformal map from $H_{t}:=\mathbb{H}\backslash K_{t}$ onto $\mathbb{H}$ normalized at $\infty$. Since the half-plane capacity of $K_t$ is $2t$ for all $t\ge 0$, we say that the process $(K_t, t\ge 0)$ is parameterized by the half-plane capacity. We say that $(K_t, t\ge 0)$ can be generated by the continuous curve $(\eta(t), t\ge 0)$ if for any $t$, the unbounded connected component of $\HH\setminus\eta[0,t]$ coincides with $H_t=\HH\setminus K_t$. 

Indeed, continuous simple curve under mild constraints does solve the Loewner equation with continuous driving function. 
Suppose $T\in (0,\infty]$ and $\eta: [0,T)\to \overline{\HH}$ is a continuous simple curve with $\eta(0)=0$. Assume $\eta$ satisfies the following: for every $t\in (0,T)$,
\begin{itemize}
\item $\eta(t, T)$ is contained in the closure of the unbounded connected component of $\HH\setminus\eta[0,t]$ and
\item $\eta^{-1}(\eta[0,t]\cup\R)$ has empty interior in $(t, T)$.
\end{itemize}
For each $t>0$, let $g_t$ be the conformal map which maps the unbounded connected component of $\HH\setminus\eta[0,t]$ onto $\HH$ normalized at $\infty$. After reparameterization, $(g_t, t\ge 0)$ solves the above Loewner equation with continuous driving function \cite[Section 4.1]{LawlerConformallyInvariantProcesses}. 

Here we discuss the evolution of a point $y\in\R$ under $g_t$. We assume $y\ge 0$. There are two possibilities: if $y$ is not swallowed by $K_t$, then we define $Y_t=g_t(y)$; if $y$ is swallowed by $K_t$, then we define $Y_t$ to be the image of the rightmost of point of $K_t\cap\R$ under $g_t$. Suppose that $(K_t, t\ge 0)$ is generated by a continuous path $(\eta(t), t\ge 0)$ and that the Lebesgue measure of $\eta[0,\infty]\cap\R$ is zero. Then the process $Y_t$ is uniquely characterized by the following equation: 
\[Y_t=y+\int_0^t \frac{2ds}{Y_s-W_s},\quad Y_t\ge W_t,\quad \forall t\ge 0.\] 
In this paper, we may write $g_t(y)$ for the process $Y_t$. 

\smallbreak
\textit{Schramm Loewner Evolution} $\SLE_{\kappa}$ is the random Loewner chain $(K_{t}, t\ge 0)$ driven by $W_t=\sqrt{\kappa}B_t$ where $(B_t, t\ge 0)$ is a standard one-dimensional Brownian motion.
In \cite{RohdeSchrammSLEBasicProperty}, the authors prove that $(K_{t}, t\ge 0)$ is almost surely generated by a continuous transient curve, i.e. there almost surely exists a continuous curve $\eta$ such that for each $t\ge 0$, $H_{t}$ is the unbounded connected component of $\mathbb{H}\backslash\eta[0,t]$ and that $\lim_{t\to\infty}|\eta(t)|=\infty$. There are phase transitions at $\kappa=4$ and $\kappa=8$: $\SLE_{\kappa}$ are simple curves when $\kappa\in (0,4]$; they have self-touching when $\kappa\in (4,8)$; and they are space-filling when $\kappa\ge 8$. 

It is clear that $\SLE_{\kappa}$ is scaling invariant, thus we can define $\SLE_{\kappa}$ in any Dobrushin domain $(\Omega; x, y)$ via conformal image: let $\phi$ be a conformal map from $\HH$ onto $\Omega$ that sends $0$ to $x$ and $\infty$ to $y$, then define $\phi(\eta)$ to be $\SLE_{\kappa}$ in $\Omega$ from $x$ to $y$. 
For $\kappa\in (0,8)$, the curves $\SLE_{\kappa}$ enjoys \textit{reversibility}: let $\eta$ be an $\SLE_{\kappa}$ in $\Omega$ from $x$ to $y$, then the time-reversal of $\eta$ has the same law as $\SLE_{\kappa}$ in $\Omega$ from $y$ to $x$. The reversibility for $\kappa\in (0,4]$ was proved in \cite{ZhanReversibility}, and it was proved for $\kappa\in (4,8)$ in \cite{MillerSheffieldIG3}. 

%% file: tex/pre_cvg_curves.tex
In this section, we first recall the main result of \cite{KemppainenSmirnovRandomCurves} and then show a similar result for pairs of curves.  Suppose $(Q; a, b, c, d)$ is a quad. 
We say that a curve $\eta$ \textit{crosses} $Q$ if there exists a subinterval $[s,t]$ such that $\eta(s,t)\subset Q$ and $\eta[s,t]$ intersects both $(ab)$ and $(cd)$. 
Fix a Dobrushin domain $(\Omega; x, y)$, 
for any curve $\eta$ in $X_0(\Omega; x,y)$ and any time $\tau$,  define $\Omega_{\tau}$ to be the connected component of $\Omega\setminus\eta[0,\tau]$ with $y$ on the boundary. 
Consider a quad $(Q; a, b, c, d)$ in $\Omega_{\tau}$ such that $(bc)$ and $(da)$ are contained in $\partial \Omega_{\tau}$. We say that $Q$ is \textit{avoidable} if it does not disconnect $\eta(\tau)$ from $y$ in $\Omega_{\tau}$.    

\begin{definition}
A family $\Sigma$ of probability measures on curves in $X_{\simple}(\Omega; x, y)$ is said to satisfy \textbf{Condition C2} if, for any $\eps>0$, there exists a constant $c(\eps)>0$ such that for any $\PP\in\Sigma$, any stopping time $\tau$, and any avoidable quad $(Q; a, b, c, d)$ in $\Omega_{\tau}$ such that $d_{Q}((ab), (cd))\ge c(\eps)$, we have 
\[\PP[\eta[\tau,1]\text{ crosses }Q\cond \eta[0,\tau]]\le 1-\eps.\]
If the above property holds for $\tau=0$, we say that the family satisfies \textbf{Condition C1}. 
\end{definition}
It is clear that the combination of Condition~C1 and
CMP implies Condition~C2. 
\begin{theorem}\label{thm::cvg_curves_chordal}
\cite[Corollary 1.7, Proposition 2.6]{KemppainenSmirnovRandomCurves}. 
Fix a Dobrushin domain $(\Omega; x, y)$. 
Suppose that $\{\eta_n\}_{n\in\N}$ is a sequence of curves in $X_{\simple}(\Omega; x, y)$ satisfying Condition C2. Denote by  $(W_n(t), t\ge 0)$ the driving process of $\eta_n$. Then 
\begin{itemize}
\item the family $\{W_n\}_{n\in\N}$ is tight in the metrisable space of continuous functions on $[0,\infty)$ with the topology of uniform convergence on compact subsets of $[0,\infty)$;
\item the family $\{\eta_n\}_{n\in\N}$ is tight in the space of curves $X$;
\item the family $\{\eta_n\}_{n\in\N}$, when each curve is parameterized by the half-plane capacity, is tight in the metrisable space of continuous functions on $[0,\infty)$ with the topology of uniform convergence on compact subsets of $[0,\infty)$.
\end{itemize}
 Moreover, if the sequence converges in any of the topologies above it also converges in the two other topologies and
 the limits agree in the sense that the limiting random curve is driven by the limiting driving function.
\end{theorem}

Next, we will explain a similar result for pairs of curves. Fix a quad $(\Omega; x^L, x^R, y^R, y^L)$. 

\begin{definition}
A family $\Sigma$ of probability measures on pairs of curves in $X_{\simple}(\Omega; x^L, x^R, y^R, y^L)$ is said to satisfy \textbf{Condition C2} if, for any $\eps>0$, there exists a constant $c(\eps)>0$ such that for any $\PP\in\Sigma$, the following holds.
Given any $\eta^L$-stopping time $\tau^L$ and any $\eta^R$-stopping time $\tau^R$,  and any avoidable quad $(Q^R; a^R, b^R, c^R, d^R)$ for $\eta^R$ in $\Omega\setminus (\eta^L[0,\tau^L]\cup\eta^R[0,\tau^R])$ such that $d_{Q^R}((a^Rb^R), (c^Rd^R))\ge c(\eps)$, and any avoidable quad $(Q^L; a^L, b^L, c^L, d^L)$ for $\eta^L$ in $\Omega\setminus (\eta^L[0,\tau^L]\cup\eta^R[0,\tau^R])$ such that $d_{Q^L}((a^Lb^L), (c^Ld^L))\ge c(\eps)$, we have 
\[\PP\left[\eta^R[\tau^R, 1]\text{ crosses }Q^R\cond \eta^L[0,\tau^L], \eta^R[0,\tau^R]\right]\le 1-\eps,\]
\[\PP\left[\eta^L[\tau^L, 1]\text{ crosses }Q^L\cond \eta^L[0,\tau^L], \eta^R[0,\tau^R]\right]\le 1-\eps.\]
If the above property holds for $\tau^L=\tau^R=0$, we say that the family satisfies \textbf{Condition C1}. 
\end{definition}

\begin{theorem}\label{thm::cvg_pairs}
Suppose that $\{(\eta^L_n;\eta^R_n)\}_{n\in\N}$ is a sequence of pairs of curves in $X_{\simple}(\Omega; x^L, x^R, y^R, y^L)$ and denote their laws by $\{\PP_n\}_{n\in\N}$. Let $\Omega_n^L$ be the connected component of $\Omega\setminus\eta_n^L$ with $(x^Ry^R)$ on the boundary and $\Omega_n^R$ be the connected component of $\Omega\setminus\eta_n^R$ with $(y^Lx^L)$ on the boundary. 
Define, for each $n$, 
\[\LD^L_n=d_{\Omega_n^L}(\eta^L_n, (x^Ry^R)),\quad \LD^R_n=d_{\Omega_n^R}(\eta^R_n, (y^Lx^L)).\]
Assume that the family $\{(\eta^L_n; \eta^R_n)\}_{n\in\N}$ satisfies Condition C2 and that the sequence of random variables $\{(\LD^L_n; \LD^R_n)\}_{n\in\N}$ is tight in the following sense: for any $u>0$, there exists $\eps>0$ such that 
\[\PP_n\left[\LD^{L}_n\ge\eps, \LD^R_n\ge\eps\right]\ge 1-u,\quad \forall n.\]
Then the sequence $\{(\eta^L_n;\eta^R_n)\}_{n\in\N}$ is relatively compact in $X_0(\Omega; x^L, x^R, y^R, y^L)$. 
\end{theorem}
\begin{proof}
By Theorem~\ref{thm::cvg_curves_chordal}, we know that there is subsequence $n_k\to\infty$ such that $\eta^L_{n_k}$ (resp. $\eta^R_{n_k}$) converges weakly in all three topologies in Theorem~\ref{thm::cvg_curves_chordal}. By Skorohod Represnetation Theorem, we could couple all $(\eta^L_{n_k};\eta^R_{n_k})$ in a common space so that $\eta^L_{n_k}\to\eta^L$ and $\eta^R_{n_k}\to\eta^R$ almost surely. 
For $\eps>0$, define 
\[K_{\eps}=\left\{(\eta^L; \eta^R)\in X_{\simple}(\Omega; x^L, x^R, y^R, y^L): d_{\Omega^L}(\eta^L, (x^Ry^R))\ge\eps, d_{\Omega^R}(\eta^R, (y^Lx^L))\ge\eps\right\}.\]
From the assumption, we know that, for any $u>0$, there exists $\eps>0$ such that $\inf_n\PP_n[K_{\eps}]\ge 1-u$. 
Therefore, with probability at least $1-u$, the sequence $(\eta^L_{n_k}; \eta^R_{n_k})$ converges to $(\eta^L; \eta^R)$ in $X^{\eps}_0(\Omega; x^L, x^R, y^R, y^L)\subset X_0(\Omega; x^L, x^R, y^R, y^L)$. This is true for any $u>0$, thus we have $(\eta^L_{n_k};\eta^R_{n_k})$ converges to $(\eta^L;\eta^R)$ in $X_0(\Omega; x^L, x^R, y^R, y^L)$ almost surely. 
\end{proof}

%% file: tex/pre_cmp_sle.tex
\textit{$\SLE_{\kappa}(\rho)$ processes} are variants of $\SLE_{\kappa}$ where one keeps track of one extra point on the boundary. $\SLE_{\kappa}(\rho)$ process with force point $w\in\R$ is the Loewner evolution driven by $W_t$ which is the solution to the system of integrated SDEs: 
\begin{equation*}
W_t=\sqrt{\kappa}B_t+\int_0^t\frac{\rho ds}{W_s-V_s},\quad V_t=w+\int_0^t\frac{2ds}{V_s-W_s},
\end{equation*}
where $B_t$ is one-dimensional Brownian motion. 
For $\rho\in\R$, the process is well-defined up to the first time that $w$ is swallowed.  
When $\rho>-2$, the process is well-defined for all time and it is generated by continuous transient curve. Assume $w\ge 0$, when $\rho\ge\kappa/2-2$, the curve never hits the interval $[w,\infty)$; when $\rho <\kappa/2-2$, the curve hits the interval $[w,\infty)$ at finite time; and when $\rho\le \kappa/2-4$, the curve accumulates at the point $w$ almost surely. Since the the process is conformal invariant, we can define $\SLE_{\kappa}(\rho)$ in any triangle via conformal image.

\begin{lemma}\label{lem::sle_c1}
Fix $\kappa\in (0,8)$ and $\rho>(-2)\vee(\kappa/2-4)$. Then $\SLE_{\kappa}(\rho)$ satisfies Condition C1. 
\end{lemma}
\begin{proof}
Suppose $\eta$ is an $\SLE_{\kappa}(\rho)$ in $\HH$ from 0 to $\infty$ with force point $w\in\R$. Then there exists a function $p(\delta)\to 0$ as $\delta\to 0$ such that 
\begin{equation}\label{eqn::sle_c1_aux}
\PP[\eta\text{ hits }B(1,\delta)]\le p(\delta),
\end{equation}
and that $p$ depends only on $\kappa, \rho$ and is uniform over $w$, see for instance \cite[Lemma 6.4]{WuAlternatingArmIsing}. 

Suppose $(Q; a, b, c, d)$ is an avoidable quad for $\eta$. It is explained in \cite[Eq.(12) in the proof of Theorem 1.10]{KemppainenSmirnovRandomCurves} that $\{\eta \text{ crosses }Q\}$ implies $\{\eta\text{ hits }B(u,r)\}$ for some $u\in\R, r>0$ such that 
\[\frac{r}{|u|}=\left(\frac{\exp(\pi d_Q((ab), (cd)))}{16}-1\right)^{-1}.\] 
Combining with \eqref{eqn::sle_c1_aux}, we see that $\eta$ satisfies Condition C1. 
\end{proof}
\begin{lemma}\label{lem::sle_kapparho_targetchanging}
\cite[Theorem 3]{SchrammWilsonSLECoordinatechanges}.
Fix $\kappa>0$ and $\rho\in\R$. Fix a triangle $(\Omega; x, w, y)$, let $\eta$ be an $\SLE_{\kappa}(\rho)$ in $\Omega$ from $x$ to $y$ with force point $w$. Then $\eta$ has the same law as $\SLE_{\kappa}(\kappa-6-\rho)$ in $\Omega$ from $x$ to $w$ with force point $y$, up to the first time that the curve disconnects $w$ from $y$. 
\end{lemma}

Next, we explain the conformal Markov characterization of $\SLE_{\kappa}(\rho)$ derived in \cite{MillerSheffieldIG2}. Recall that $\LT$ is the collection of all triangles. 
\begin{definition}\label{def::CMP_threepoints}
Suppose $(\PP_c, c\in\LT)$ is a family of probability measures on continuous curves from $x$ to $y$ in $\Omega$. We say that $(\PP_c, c\in\LT)$ satisfies conformal Markov property (CMP) if it satisfies the following two properties.
\begin{itemize}
\item Conformal invariance. Suppose that $c=(\Omega; x, w, y), \tilde{c}=(\tilde{\Omega}; \tilde{x}, \tilde{w}, \tilde{y})\in\LT$, and $\psi:\Omega\to \tilde{\Omega}$ is the conformal map with $\psi(x)=\tilde{x}, \psi(w)=\tilde{w}, \psi(y)=\tilde{y}$. Then for $\eta\sim\PP_c$, we have $\psi(\eta)\sim\PP_{\tilde{c}}$.
\item Domain Markov property. Suppose $\eta\sim\PP_c$, then for every $\eta$-stopping time $\tau$, the conditional law of $(\eta|_{t\ge\tau})$ given $\eta[0,\tau]$ is the same as $\PP_{c_{\tau}}$ where $c_{\tau}=(\Omega_{\tau}; \eta(\tau), w_{\tau}, y)$. Here $\Omega_{\tau}$ is the connected component of $\Omega\setminus\eta[0,\tau]$ with $y$ on the boundary, and $w_{\tau}=w$ if $w$ is not swallowed by $\eta[0,\tau]$ and $w_{\tau}$ is the last point of $\eta[0,\tau]\cap (xy)$ if $w$ is swallowed by $\eta[0,\tau]$.  
\end{itemize}
\end{definition}
\begin{theorem}\label{thm::CMP_threepoints}
\cite[Theorem 1.4]{MillerSheffieldIG2}. Suppose $(\PP_c, c\in\LT)$ satisfies CMP in Definition~\ref{def::CMP_threepoints} and Condition C1, then there exist $\kappa\in (0,8)$ and $\rho>(-2)\vee(\kappa/2-4)$ such that, for each $c=(\Omega; x, w, y)\in\LT$, $\PP_c$ is the law of $\SLE_{\kappa}(\rho)$ in $\Omega$ from $x$ to $y$ with force point $w$.
\end{theorem}
In \cite[Theorem 1.4]{MillerSheffieldIG2}, the authors did not require Condition C1; instead, they required the assumption that, when $\partial\Omega$ is smooth, the Lebesgue measure of $\eta\cap\partial\Omega$ is zero almost surely. Note that Condition C1 implies this latter assumption, and we find Condition C1 is more natural, since it is the continuum counterpart of RSW for critical lattice model.

%% file: tex/pre_sle_multiple_forcepoints.tex
\textit{$\SLE_{\kappa}(\underline{\rho})$ processes}
are variants of $\SLE_{\kappa}$ where one keeps 
track of multiple points on the boundary. Suppose $\underline{y}=(0\le y_1<y_2<\cdots<y_n)$ and $\underline{\rho}=(\rho_1, \ldots, \rho_n)$ with $\rho_i\in\R$. 
An $\SLE_{\kappa}(\underline{\rho})$ process with force points 
$\underline{y}$ is the Loewner evolution driven by $W_t$ 
which is the solution to the following system of integrated SDEs: 
\begin{equation*}
W_t = \sqrt{\kappa} B_t +
\sum_{i=1}^n 
\int_0^t\frac{\rho_i ds}{W_s-V_s^{i}} ,\quad 
V^{i}_t = y_i + \int_0^t\frac{2ds}{V^{i}_s-W_s} , 
\quad \text{for }1\le i\le n,
\end{equation*}
where $B_t$ is an one-dimensional Brownian motion. Note that the process $V_t^{i}$ is 
the time evolution of the point $y_i$, and we may write $g_t(y_i)$ for $V_t^{i}$. 
We define the \textit{continuation threshold} 
of the $\SLE_{\kappa}(\underline{\rho})$ 
to be the infimum of the time $t$ for which 
\[ \sum_{i : V^{i}_t=W_t} \rho_i\le -2 . \]
By~\cite{MillerSheffieldIG1}, 
the $\SLE_{\kappa}(\underline{\rho})$
process is well-defined up to the continuation threshold, and it is almost surely
generated by a continuous curve up to and including the continuation threshold. The Radon-Nikodym derivative between $\SLE_{\kappa}(\underline{\rho})$ and $\SLE_{\kappa}$ is given by the following lemma. 

\begin{lemma}\label{lem::sle_kapparho_mart}
\cite{SchrammWilsonSLECoordinatechanges}.
The process $\SLE_{\kappa}(\underline{\rho})$ with force points $\underline{y}$ is the same as $\SLE_{\kappa}$ process weighted by the following local martingale, up to the first time that $y_1$ is swallowed:
\[M_t=\prod_{1\le i\le n} \left(g_t'(y_i)^{\rho_i(\rho_i+4-\kappa)/(4\kappa)}(g_t(y_i)-W_t)^{\rho_i/\kappa}\right)\times \prod_{1\le i<j\le n}(g_t(y_j)-g_t(y_i))^{\rho_i\rho_j/(2\kappa)}.\]
\end{lemma}

%% file: tex/hypersle_def.tex
Fix $\kappa\in (0,8)$ and $\nu\in\R$, and four boundary points $x_1<x_2<x_3<x_4$.  
We first define function $F$ which is a solution to the following Euler's hypergeometric differential equation
\begin{equation}\label{eqn::euler_ode}
z(1-z)F''(z)+\left(\frac{2\nu+8}{\kappa}-\frac{2\nu+2\kappa}{\kappa}z\right)F'(z)-\frac{2(\nu+2)(\kappa-4)}{\kappa^2}F(z)=0. 
\end{equation}
When $\kappa\in (0,8)$ and $\nu>(-4)\vee(\kappa/2-6)$, define $F$ to be the hypergeometric function (see Appendix~\ref{sec::appendix}):
\begin{equation}\label{eqn::hSLE_hyperF}
F(z):=\hF\left(\frac{2\nu+4}{\kappa}, 1-\frac{4}{\kappa}, \frac{2\nu+8}{\kappa}; z\right).
\end{equation}

\begin{lemma}\label{lem::hyperF_bound}
When $\kappa\in (0,8)$ and $\nu>(-4)\vee(\kappa/2-6)$, the function $F$ defined in~\eqref{eqn::hSLE_hyperF} is monotone for $z\in [0,1]$. In particular, for all $z\in [0,1]$, 
\[0<1\wedge F(1)\le F(z)\le 1\vee F(1)<\infty. \]
\end{lemma}
\begin{proof}
Denote by 
\[A=\frac{2\nu+4}{\kappa},\quad B=1-\frac{4}{\kappa},\quad C=\frac{2\nu+8}{\kappa}.\]
Since $C>0, C>A, C>B$ and $C>A+B$, we know that $F(1)\in (0, \infty)$ by~\eqref{eqn::hyperF_one}. Thus we only need to show that $F$ is monotone. 
If $AB>0$, then $F$ is increasing by Lemma~\ref{lem::hyperF_increasing}. If $AB=0$, we have $F\equiv 1$. 

In the following, we assume $AB<0$. First assume $B<0<A$, i.e. $\kappa<4$ and $\nu>-2$. There exists $n\in\{1, 2,\ldots\}$ such that $1>B+n\ge 0$. In other words, $4/\kappa>n\ge 4/\kappa-1$. 
By \cite[Eq. (15.2.2)]{AbramowitzHandbook}, we have
\[F^{(n)}(z)=\frac{(A)_n(B)_n}{(C)_n}\hF(A+n, B+n, C+n; z).\]
Since $A+n\ge 0, B+n\ge 0, C+n\ge 0$, we have $\hF(A+n, B+n, C+n; z)\ge 1$. Moreover 
\[C=A+B+\frac{8}{\kappa}-1>A+B+n-1.\]
and thus $\hF(A+j, B+j, C+j; 1)\in (0,\infty)$ for $0\le j\le n-1$. 
If $n$ is even, we have $F^{(n)}(z)\ge 0$. Thus $F^{(n-1)}(\cdot)$ is increasing. In particular,
\[F^{(n-1)}(z)\le F^{(n-1)}(1)=\frac{(A)_{n-1}(B)_{n-1}}{(C)_{n-1}}\hF(A+n-1, B+n-1, C+n-1; 1)\le 0. \]
Thus $F^{(n-2)}(\cdot)$ is decreasing and $F^{(n-2)}(z)\ge F^{(n-2)}(1)\ge 0$. In this way, we could argue that $F^{(n-j)}(\cdot)$ is decreasing for even $j$ and it is increasing for odd $j$. In particular, $F$ is decreasing. 

If $n$ is odd, we have $F^{(n)}(z)\le 0$. Thus $F^{(n-1)}(\cdot)$ is decreasing, and thus
\[F^{(n-1)}(z)\ge F^{(n-1)}(1)=\frac{(A)_{n-1}(B)_{n-1}}{(C)_{n-1}}\hF(A+n-1, B+n-1, C+n-1; 1)\ge 0.\]
Thus $F^{(n-2)}(\cdot)$ is increasing, and in particular $F^{(n-2)}(z)\ge 1$. In this way, we could argue that $F^{(n-j)}(\cdot)$ is increasing for even $j$ and it is decreasing for odd $j$. In particular, $F$ is decreasing. 

Next assume $A<0<B$, i.e. $\kappa\in (4,8)$ and $-2>\nu>\kappa/2-6$. There exists $n\in\{1,2,\ldots\}$ such that 
$1>A+n\ge 0$. In other words, $1-(2\nu+4)/\kappa>n\ge (2\nu+4)/\kappa$.
Since $A+n\ge 0, B+n\ge 0, C+n\ge 0$, we have $\hF(A+n, B+n, C+n; z)\ge 1$. Moreover, since $\nu>\kappa/2-6$, we have
\[C=A+B+\frac{8}{\kappa}-1>A+B+n-1,\]
and thus $\hF(A+j, B+j, C+j; 1)\in (0, \infty)$ for $0\le j\le n-1$. We can repeat the same argument as above, and we see that $F$ is decreasing.
\end{proof}

When $\kappa\in (0,8)$ and $\nu\le (-4)\vee(\kappa/2-6)$, define $F$ to be the following: 
\begin{equation}\label{eqn::hSLE_hyperF_reflect}
F(z):=(1-z)^{8/\kappa-1}G(1-z),\quad \text{where}\quad G(z)=\hF\left(\frac{2\nu+12-\kappa}{\kappa}, \frac{4}{\kappa}, \frac{8}{\kappa}; z\right).
\end{equation}
Note that the function $F$ defined in both~\eqref{eqn::hSLE_hyperF} and~\eqref{eqn::hSLE_hyperF_reflect} are solutions to~\eqref{eqn::euler_ode}. By a similar argument as in Lemma~\ref{lem::hyperF_bound}, the function $G(z)$ is uniformly bounded for $z\in [0,1]$ by $G(0)=1$ and $G(1)\in (0,\infty)$. 

Set
\begin{equation}\label{eqn::hSLE_constants}
h=\frac{6-\kappa}{2\kappa}, \quad a=\frac{\nu+2}{\kappa},\quad b=\frac{(\nu+2)(\nu+6-\kappa)}{4\kappa}.
\end{equation}
Define partition function
\begin{equation}\label{eqn::hSLE_partition}
\PartF_{\kappa,\nu}(x_1, x_2, x_3, x_4)=(x_4-x_1)^{-2h}(x_3-x_2)^{-2b}z^aF(z),\quad \text{where}\quad z=\frac{(x_2-x_1)(x_4-x_3)}{(x_3-x_1)(x_4-x_2)}.
\end{equation}
Since $\PartF_{\kappa,\nu}$ is conformal covariant under M\"{o}bius maps of $\HH$, we could define partition function for any quad $q=(\Omega; x_1, x_2, x_3, x_4)$ via conformal image: 
\begin{equation}\label{eqn::hSLE_partition_general}
\PartF_{\kappa, \nu}(\Omega; x_1, x_2, x_3, x_4)=\varphi'(x_1)^h\varphi'(x_2)^b\varphi'(x_3)^b\varphi'(x_4)^h\PartF_{\kappa,\nu}(\varphi(x_1), \varphi(x_2), \varphi(x_3), \varphi(x_4)),
\end{equation}
where $\varphi$ is any conformal map from $\Omega$ onto $\HH$ such that $\varphi(x_1)<\varphi(x_2)<\varphi(x_3)<\varphi(x_4)$. 

The process $\hSLE_{\kappa}(\nu)$ in $\HH$ from $x_1$ to $x_4$ with marked points $(x_2, x_3)$ is the Loewner chain driven by $W_t$ which is the solution to the following SDEs: 
\begin{align*}
dW_t=\sqrt{\kappa}dB_t+\kappa(\partial_{x_1}\log\PartF_{\kappa, \nu})(W_t, g_t(x_2), g_t(x_3), g_t(x_4))dt,\quad \partial_t g_t(x_i)=\frac{2}{g_t(x_i)-W_t},\quad\text{for }i=2,3,4. 
\end{align*}
In particular, this implies that the law of $\eta$ is the same as $\SLE_{\kappa}$ in $\HH$ from $x_1$ to $\infty$ weighted by the following local martingale:
\begin{equation}\label{eqn::hypersle_mart}
M_t=g_t'(x_2)^{b}g_t'(x_3)^bg_t'(x_4)^{h}\PartF_{\kappa, \nu}(W_t, g_t(x_2), g_t(x_3), g_t(x_4)).
\end{equation}

It is clear that the process $\hSLE$ is well-defined up to the swallowing time of $x_2$. We will show that the process is well-defined for all time when $\nu>(-4)\vee(\kappa/2-6)$ in Section~\ref{subsec::hypersle_continuity}, and that the process is well-defined up to the swallowing time of $x_2$ when $\nu\le (-4)\vee(\kappa/2-6)$ in Section~\ref{subsec::hypersle_continuity_lownu}.
In the following lemma, we explain the relation between $\hSLE_{\kappa}(\nu)$ and $\SLE_{\kappa}(\rho)$. 

\begin{lemma}\label{lem::hSLE_degenerate}
Fix $\kappa\in (0,8), \nu\in\R$ and assume $x_1<x_2<x_3<x_4$. When $x_3\to x_4$, the process $\hSLE_{\kappa}(\nu)$ in $\HH$ from $x_1$ to $x_4$ with marked points $(x_2, x_3)$ converges weakly to $\SLE_{\kappa}(\nu+2)$ in $\HH$ from $x_1$ to $x_4$ with force point $x_2$. 
\end{lemma}
\begin{proof}
Let $\eta$ be the $\hSLE_{\kappa}(\nu)$ in $\HH$. Let $\tilde{\eta}$ be the $\SLE_{\kappa}(\nu+2, \kappa-6-\nu,\kappa-6)$ in $\HH$ from $x_1$ to $\infty$ with force points $(x_2, x_3, x_4)$. We denote by $X_{j1}=g_t(x_j)-W_t$ for $j=2,3,4$ and $X_{ji}=g_t(x_j)-g_t(x_i)$ for $2\le i<j\le 4$, and $z_t=X_{21}X_{43}/(X_{31}X_{42})$. 
By Lemma~\ref{lem::sle_kapparho_mart} and \eqref{eqn::hypersle_mart}, we know that the law of $\eta$ is the same as the law of $\tilde{\eta}$ weighted by $R_t/R_0$ where 
\begin{align*}
R_t=X_{31}^{(\nu+6-\kappa)/\kappa}X_{41}^{2h}X_{32}^{b}X_{42}^{(\nu+2)h}X_{43}^{(\kappa-6-\nu)h}\left(\frac{X_{43}}{X_{31}X_{42}}\right)^aF(z_t).
\end{align*}
where $F$ is the hypergeometric function in \eqref{eqn::hSLE_hyperF} or \eqref{eqn::hSLE_hyperF_reflect}.
Suppose $\tau_r$ is the first time that $\tilde{\eta}$ exits the ball $B(0,r)$. Fix $r$, let $x_4\to \infty$ and then $x_3\to \infty$, we see that $z_0, z_{\tau_r}\to 0$. Furthermore,  $R_{\tau_r}/R_0\to 1$ and $R_{\tau_r}/R_0$ is uniformly bounded. This implies the conclusion. 
\end{proof}

In the rest of the paper, we will write $\hSLE_{\kappa}$ for $\hSLE_{\kappa}(0)$ with $\nu=0$.

%% file: tex/hypersle_continuity.tex
To derive the continuity of the process, it is more convenient to work with $\hSLE$ in $\HH$ from $0$ to $\infty$. When $\kappa\in (0,8), \nu>(-4)\vee(\kappa/2-6)$, the function $F(z)$ defined in \eqref{eqn::hSLE_hyperF} is uniformly bounded for $z\in [0,1]$ between $F(0)=1$ and $F(1)\in (0, \infty)$. 
The process $\hSLE_{\kappa}(\nu)$ in $\HH$ from 0 to $\infty$ with marked points $(x,y)$ is the random Loewner chain driven by $W$ which is the solution to the following system of SDEs:
\begin{align}\label{eqn::hyperSLE_sde}
\begin{split}
dW_t&=\sqrt{\kappa}dB_t+\frac{(\nu+2)dt}{W_t-V^x_t}+\frac{-(\nu+2)dt}{W_t-V_t^y}-\kappa\frac{F'(Z_t)}{F(Z_t)}\left(\frac{1-Z_t}{V_t^y-W_t}\right)dt,\\
dV_t^x&=\frac{2dt}{V_t^x-W_t},\quad dV_t^y=\frac{2dt}{V_t^y-W_t}, \quad \text{where}\quad Z_t=\frac{V^x_t-W_t}{V_t^y-W_t},
\end{split}
\end{align}
where $B_t$ is one-dimensional Brownian motion, and the initial values are $W_0=0$, $V_0^x=x$ and $V_0^y=y$. 

\begin{proposition}\label{prop::hyperSLE}
Fix $\kappa\in (0,8), \nu>(-4)\vee (\kappa/2-6)$ and $0<x<y$. The $\hSLE_{\kappa}(\nu)$ in $\HH$ from 0 to $\infty$ with marked points $(x,y)$ is well-defined for all time and it is almost surely generated by a continuous transient curve. Moreover, it never hits the interval $[x,y]$ when $\nu\ge \kappa/2-4$. 
\end{proposition}

Before proving Proposition~\ref{prop::hyperSLE}, let us compare $\hSLE_{\kappa}(\nu)$ with $\SLE_{\kappa}(\nu+2,\kappa-6-\nu)$ process. 
By Girsanov's Theorem, one can check that 
the law of $\hSLE_{\kappa}(\nu)$ with marked points $(x,y)$ is the same as the law of $\SLE_{\kappa}(\nu+2,\kappa-6-\nu)$ with force points $(x,y)$ weighted by $R_t/R_0$ where 
\[R_t=F(Z_t)\left(g_t(y)-W_t\right)^{4/\kappa-1},\quad \text{where}\quad Z_t=\frac{g_t(x)-W_t}{g_t(y)-W_t}.\]
Note that $0\le Z_t\le 1$ for all $t$ and $F(z)$ is bounded for $z\in [0,1]$. Let $T_y$ be the swallowing time of $y$ and define, for $n\ge 1$, 
\[S_n=\inf\{t: g_t(y)-W_t\le 1/n \text{ or }g_t(y)-W_t\ge n\}.\]
Then we see that $R_{S_n}$ is bounded. Therefore, the law of $\hSLE_{\kappa}(\nu)$ is absolutely continuous with respect to the law of $\SLE_{\kappa}(\nu+2,\kappa-6-\nu)$ up to $S_n$. Since $\SLE_{\kappa}(\nu+2,\kappa-6-\nu)$ is generated by a continuous curve up to $T_y$ and it does not hit the interval $[x,y)$ when $\nu\ge\kappa/2-4$, we know that $\hSLE_{\kappa}(\nu)$ is generated by a continuous curve up to $S_n$ and it does not hit the interval $[x,y)$ up to $S_n$ when $\nu\ge \kappa/2-4$. Let $n\to\infty$, we see that $\hSLE_{\kappa}(\nu)$ is generated by continuous curve up to $T_y=\lim_n S_n$ and it does not hit the interval $[x,y)$ up to $T_y$ when $\nu\ge\kappa/2-4$. 

Note that the absolute continuity of $\hSLE_{\kappa}(\nu)$ with respect to $\SLE_{\kappa}(\nu+2,\kappa-6-\nu)$ is not preserved as $n\to\infty$, since $R_t$ may be no longer bounded away from 0 or $\infty$ as $t\to T_y$. The following lemma discusses the behavior of $\hSLE_{\kappa}(\nu)$ as $t\to T_y$. 

\begin{lemma}\label{lem::hyperSLE_aroundTy}
When $\kappa\in (0,8)$ and $\nu>(-4)\vee(\kappa/2-6)$, 
the $\hSLE_{\kappa}(\nu)$ is well-defined and is generated by continuous curve up to and including the swallowing time of $y$, denoted by $T_y$. Moreover, the curve does not hit the interval $[x,y]$ if $\nu\ge \kappa/2-4$; and $T_y=\infty$ when $\kappa\le 4$; and the curve accumulates at a point in the interval $(y,\infty)$ as $t\to T_y<\infty$ when $\kappa\in (4,8)$. 
\end{lemma}

\begin{proof}
Since $\hSLE_{\kappa}(\nu)$ is scaling invariant, we may assume $y=1$ and $x\in (0,1)$, and denote $T_y$ by $T$. 
In this lemma, we discuss the behavior of $\hSLE_{\kappa}(\nu)$ as $t\to T$ and we will argue that the process does not accumulate at the point $1$. To this end, we perform a standard change of coordinate and parameterize the process according the capacity seen from the point $1$, see \cite[Theorem 3]{SchrammWilsonSLECoordinatechanges}.

Set $f(z)=z/(1-z)$. Clearly, $f$ is the conformal M\"{o}bius transform of $\HH$ sending the points $(0,1,\infty)$ to $(0,\infty, -1)$. Consider the image of $(K_t, 0\le t\le T)$ under $f$: $(\tilde{K}_s, 0\le s\le \tilde{S})$ where we parameterize this curve by its capacity $s(t)$ seen from $\infty$. Let $(\tilde{g}_s)$ be the corresponding family of conformal maps and $(\tilde{W}_s)$ be the driving function. Let $f_t$ be the M\"{o}bius transform of $\HH$ such that $\tilde{g}_s\circ f=f_t\circ g_t$ where $s=s(t)$. By expanding $\tilde{g}_s=f_t\circ g_t\circ f^{-1}$ around $\infty$ and comparing the coefficients in both sides, we have
\[f_t(z)=-1-\frac{g_t''(1)}{2g_t'(1)}+\frac{g_t'(1)}{g_t(1)-z}.\]
Thus, with $s=s(t)$, 
\[ \tilde{W}_s=f_t(W_t)=-1-\frac{g_t''(1)}{2g_t'(1)}+\frac{g_t'(1)}{g_t(1)-W_t},\quad d\tilde{W}_s=\frac{(\kappa-6)g_t'(1)dt}{(g_t(1)-W_t)^3}+\frac{g_t'(1)dW_t}{(g_t(1)-W_t)^2}.\]
Define 
\[\tilde{V}^x_s=f_t(V_t^x), \quad \tilde{V}^{\infty}_s=f_t(\infty),\quad \tilde{Z}_s=\frac{\tilde{V}^x_s-\tilde{W}_s}{\tilde{V}^x_s-\tilde{V}^{\infty}_s}=Z_t.\]
Plugging in the time change
\[\dot{s}(t)=f_t'(W_t)^2=\frac{g_t'(1)^2}{(g_t(1)-W_t)^4},\]
we obtain
\[d\tilde{W}_s=\sqrt{\kappa}d\tilde{B}_s+\frac{(\nu+2)ds}{\tilde{W}_s-\tilde{V}^x_s}+\frac{(\kappa-6)ds}{\tilde{W}_s-\tilde{V}^{\infty}_s}-\kappa\frac{F'(\tilde{Z}_s)}{F(\tilde{Z}_s)}\frac{ds}{\tilde{V}^x_s-\tilde{V}^{\infty}_s},\]
where $\tilde{B}_s$ is one-dimensional Brownian motion. By Girsanov's Theorem, the Radon-Nikodym derivative of the law of $\tilde{K}$ with respect to the law of $\SLE_{\kappa}(\kappa-6;\nu+2)$ with force points $(-1; \tilde{x}:=x/(1-x))$ is given by 
\[R_s=\frac{F(Z_s)}{F(x)}\left(\frac{g_s(\tilde{x})-g_s(-1)}{(1-x)^{-1}}\right)^{-(\nu+2)/\kappa},\quad \text{where}\quad Z_s=\frac{g_s(\tilde{x})-W_s}{g_s(\tilde{x})-g_s(-1)}.\]
Note that $0\le Z_s\le 1$ and $F(z)$ is bounded for $z\in [0,1]$; and that the process $g_s(\tilde{x})-g_s(-1)$ is increasing, thus $g_s(\tilde{x})-g_s(-1)\ge 1/(1-x)$. Let $S$ be the swallowing time of $-1$. Define, for $n\ge 1$,
\[S_n=\inf\{t: K_t\text{ exits }B(0,n)\}.\]
Then $R_s$ is bounded up to $S\wedge S_n$, and thus the process $\tilde{K}$ is absolutely continuous with respect to $\SLE_{\kappa}(\kappa-6;\nu+2)$ up to $S\wedge S_n$. We list some properties of $\SLE_{\kappa}(\kappa-6;\nu+2)$ with force points $(-1; \tilde{x}=x/(1-x))$ here: it is generated by continuous curve up to and including the continuation threshold; 
the curve does not hit the interval $[\tilde{x},\infty)$ when $\nu+2\ge\kappa/2-2$. 
When $\kappa\in (0,4]$, the curve almost surely accumulates at the point $-1$, since $\kappa-6\le \kappa/2-4$; when $\kappa\in (4,8)$, the curve hits the interval $(-\infty, -1)$ at finite time almost surely, since $\kappa-6\in (-2,\kappa/2-2)$.
Therefore, the process $\tilde{K}$ is generated by continuous curve up to and including $\tilde{S}$. This implies that our original $\hSLE_{\kappa}(\nu)$ process $(K_t, t\ge 0)$ is generated by continuous curve up to and including $T$; moreover, the curve accumulates at a point in $(y,\infty)\cup\{\infty\}$ as $t\to T$, and it does not hit $[x,y]$ when $\nu\ge\kappa/2-4$. 
\end{proof}

\begin{proof}[Proof of Proposition~\ref{prop::hyperSLE}]
In Lemma~\ref{lem::hyperSLE_aroundTy}, we have shown that $\hSLE_{\kappa}(\nu)$ is well-defined and is generated by continuous curve up to and including $T_y$. In particular, when $\kappa\le 4$, since $T_y=\infty$, we obtain the conclusion for this case. It remains to prove the conclusion for $\kappa\in (4,8)$. In this case, as $t\to T_y$, we have 
$V^y_t-W_t\to 0$ and $Z_t\to 1$. Note that $F(z)$ remains bounded as $z\to 1$; and that $F'(z)(1-z)\to 0$ as $z\to 1$, since (see \cite[Eq. (15.2.1) and (15.3.3)]{AbramowitzHandbook})
\[F'(z)=\left(\frac{\nu+2}{\nu+4}\right)\left(1-\frac{4}{\kappa}\right)(1-z)^{8/\kappa-2}\hF\left(\frac{4}{\kappa}, \frac{12+2\nu}{\kappa}-1, \frac{8+2\nu}{\kappa}+1; z\right)\approx (1-z)^{8/\kappa-2}, \text{ as }z\to 1.\]
Combining these, we know that the SDE \eqref{eqn::hyperSLE_sde} degenerates to $W_t=\sqrt{\kappa}B_t$ for $t\ge T_y$. Therefore, the process is the same as standard $\SLE_{\kappa}$ for $t\ge T_y$, and hence is generated by continuous transient curve. 
\end{proof}

%% file: tex/hypersle_reversibility.tex
In this section, we still work with $\hSLE$ in $\HH$ from 0 to $\infty$. In this case, the local martingale in \eqref{eqn::hypersle_mart} has a more explicit expression. 

\begin{lemma}\label{lem::hypersle_mart}
Fix $\kappa\in (0,8), \nu\in\R$, suppose $\eta$ is an $\SLE_{\kappa}$ in $\HH$ from 0 to $\infty$ and $(g_t,t\ge 0)$ is the corresponding family of conformal maps. Fix $0<x<y$ and let $T_x$ be the swallowing time of $x$. 
Define, for $t<T_x$,  
\[J_t=\frac{g_t'(x)g_t'(y)}{(g_t(y)-g_t(x))^{2}},\quad Z_t=\frac{g_t(x)-W_t}{g_t(y)-W_t}.\]
Then the following process is a local martingale:
\[M_t:=Z_t^a J_t^bF(Z_t)\one_{\{t<T_x\}},\]
where $a,b$ are defined through \eqref{eqn::hSLE_constants} and $F$ is defined through \eqref{eqn::hSLE_hyperF} or \eqref{eqn::hSLE_hyperF_reflect}. 
\end{lemma}

\begin{proposition}\label{prop::hypersle_mart}
Fix $\kappa\in (0,8), \nu\ge \kappa/2-4$ and $0<x<y$. The local martingale defined in Lemma~\ref{lem::hypersle_mart} is a uniformly integrable martingale for $\SLE_{\kappa}$; and the law of $\SLE_{\kappa}$ weighted by $M_{\infty}$ is the same as $\hSLE_{\kappa}(\nu)$ with marked points $(x,y)$. Furthermore, we have the following explicit expression of $M_{\infty}$. Suppose $\eta$ is an $\SLE_{\kappa}$ in $\HH$ from $0$ to $\infty$. Assume $\eta\cap (x,y)=\emptyset$ and denote by $D$ the connected component of $\HH\setminus\eta$ with $(xy)$ on the boundary.  Then
\begin{equation*}
M_{\infty}=\left(H_D(x,y)\right)^b\one_{\{\eta\cap (x,y)=\emptyset\}},
\end{equation*}
where $H_D(x,y)$ is the boundary Poisson kernel.\footnote{Fix a Dobrushin domain $(\Omega; x, y)$. Suppose $x,y$ lie on analytic boundary segments of $\Omega$. The boundary Poisson kernel $H_{\Omega}(x,y)$ is a conformally covariant function which, in $\HH$ with $x,y\in\R$ is given by $H_{\HH}(x,y)=|y-x|^{-2}$, and in $\Omega$ it is defined via conformal image: we set $H_{\Omega}(x,y)=\varphi'(x)\varphi'(y)H_{\varphi(\Omega)}(\varphi(x),\varphi(y))$ for any conformal map $\varphi: \Omega\to \varphi(\Omega)$.} 
\end{proposition}
\begin{proof}
In Lemma~\ref{lem::hypersle_mart}, we see that $M_t$ is a local martingale up to the swallowing time of $x$. Note that $J_t$ is decreasing in $t$, thus $J_t\le J_0$. Therefore $M_t$ is bounded as long as $J_t$ and $Z_t$ are bounded from below. 
Define, for $n\ge 1$,
\[S_n=\inf\{t: J_t\le 1/n\text{ or }Z_t\le 1/n\}.\]
Then $M_{t\wedge S_n}$ is a bounded martingale; moreover, the law of $\SLE_{\kappa}$ weighted by $M_{t\wedge S_n}$ is the law of $\hSLE_{\kappa}(\nu)$ up to $S_n$. By Proposition~\ref{prop::hyperSLE}, we know that $\hSLE_{\kappa}(\nu)$ is generated by a continuous transient curve and the curve never hits the interval $[x, y]$. Therefore, $M_t$ is actually a uniformly integrable martingale for $\SLE_{\kappa}$. 
It remains to derive the explicit expression of $M_{\infty}$. As $t\to\infty$, we find 
\[Z_t\to 1,\quad J_t\to J_{\infty}:=\frac{g'(x)g'(y)}{(g(y)-g(x))^2},\]
where $g$ is any conformal map from $D$ onto $\HH$. In fact, the quantity $J_{\infty}$ is the boundary Poisson kernel $H_D(x,y)$. Thus we have almost surely $M_{\infty}=\lim_{t\to \infty}M_t=J_{\infty}^b$. This completes the proof. 
\end{proof}
\begin{proof}[Proof of Theorem~\ref{thm::hyperSLE_reversibility}]
We have shown that $\hSLE_{\kappa}(\nu)$ is generated by continuous curve in Proposition~\ref{prop::hyperSLE}, to show Theorem~\ref{thm::hyperSLE_reversibility}, it remains to show the reversibility when $\nu\ge \kappa/2-4$. By Proposition~\ref{prop::hypersle_mart}, the Radon-Nikodym derivative of the law of $\hSLE_{\kappa}(\nu)$ with marked points $(x,y)$ with respect to the law of $\SLE_{\kappa}$ is given by $M_{\infty}/M_0$ where $M_{\infty}$ is the boundary Poison kernel to the power $b$. 
Combining the reversibility of standard $\SLE_{\kappa}$ and the conformal invariance of the boundary Poisson kernel, we have the reversibility of $\hSLE_{\kappa}(\nu)$.  
\end{proof}

In Proposition~\ref{prop::hypersle_mart}, we proved the reversibility of $\hSLE_{\kappa}(\nu)$ for $\nu\ge \kappa/2-4$. In fact, we believe the reversibility holds for all $\nu>(-4)\vee(\kappa/2-6)$.
\begin{conjecture}
Fix $\kappa\in (0,8)$ and $\nu>(-4)\vee (\kappa/2-6)$ and a quad $(\Omega; x_1, x_2, x_3, x_4)$. Let $\eta$ be an $\hSLE_{\kappa}(\nu)$ in $\Omega$ from $x_1$ to $x_4$ with marked points $(x_2, x_3)$. The time-reversal of $\eta$ has the same law as $\hSLE_{\kappa}(\nu)$ in $\Omega$ from $x_4$ to $x_1$ with marked points $(x_3, x_2)$. 
\end{conjecture}

%% file: tex/hypersle_continuity_lownu.tex
\begin{proposition}\label{prop::hsle_continuity_lownu}
Fix $\kappa\in (0,8)$ and $\nu\le (-4)\vee(\kappa/2-6)$ and $0<x<y$. The $\hSLE_{\kappa}(\nu)$ in $\HH$ from $0$ to $\infty$ with marked points $(x, y)$ is well-defined up to the swallowing time of $x$, denoted by $T$. The process is almost surely continuous up to and including $T$, and it accumulates at a point in $[x, y)$ as $t\to T$.  
\end{proposition}
\begin{proof}
Suppose $\tilde{\eta}$ is an $\SLE_{\kappa}(\nu+2, \nu+2)$ in $\HH$ from $0$ to $\infty$ with force points $(x, y)$. The law of $\eta$ is the same as the law of $\tilde{\eta}$ weighted by $R_t/R_0$ where 
\[R_t=(g_t(y)-W_t)^{-(\nu+2)(\nu+6)/(2\kappa)}(1-Z_t)^{8/\kappa-1-(\nu+2)^2/(2\kappa)}G(1-Z_t), \quad\text{where}\quad Z_t=\frac{g_t(x)-W_t}{g_t(y)-W_t}.\]
Here $G$ is defined in~\eqref{eqn::hSLE_hyperF_reflect}. For $n\ge 1$, define $S_n$ to be the min of 
\[\inf\{t: \tilde{\eta}(t)\text{ exits }B(0,n)\},\quad \text{and}\quad \inf\{t: g_t(y)-g_t(x)\ge 1/n\}.\] Then $R_{T\wedge S_n}$ is bounded. Thus $\eta$ is continuous up to $T\wedge S_n$. 

First, we assume $\kappa\in (4,8)$ and $\nu\le \kappa/2-6$. 
Since $\nu+2\le \kappa/2-4$ and $2\nu+4\le \kappa/2-4$, we know that $\tilde{\eta}$ accumulates at the point $x$ as $t\to T$ (see \cite[Lemma 15]{DubedatSLEDuality}), and it is continuous up to and including $T$. This implies that $R_{T\wedge S_n}\to R_T\in (0,\infty)$ as $n\to\infty$. Therefore, $\eta$ is continuous up to and including $T$ and it accumulates at the point $x$ as $t\to T$.   

Next, we assume $\kappa\in (0,4]$ and $\nu\le -4$. 
Since $\nu+2< \kappa/2-2$ and $2\nu+4\le \kappa/2-4$, we know that $\tilde{\eta}$ accumulates at a point in $[x, y)$ as $t\to T$, and it is continuous up to and including $T$. Therefore, $\eta$ is continuous up to and including $T$ and it accumulates at a point in $[x, y)$.  This completes the proof. 
\end{proof}

\subsubsection*{The special case: $\kappa=4$}
When $\kappa=4$ and $\nu\le -4$, suppose $\eta\sim\SLE_4(\nu+2, -\nu-2)$ in $\HH$ from $x_1$ to $x_4$ with force points $(x_2, x_3)$.  In this case, the process $\eta$ can be viewed as the level line of $\GFF$ with the following boundary data ($\lambda=\pi/2$): 
\[-\lambda\text{ on }(-\infty, x_1), \quad \lambda\text{ on }(x_1, x_2),\quad \lambda(\nu+3)\text{ on }(x_2, x_3), \quad \lambda\text{ on }(x_3, x_4),\quad -\lambda\text{ on }(x_4, \infty).\]
In particular, the process $\eta$ is continuous up to and including the continuation threshold, denoted by $T$, see \cite{SchrammSheffieldContinuumGFF, WangWuLevellinesGFFI}. When $\nu+3\le -1$, the curve $\eta$ may terminate at either $x_2$ or $x_4$. Furthermore, we can calculate the probabilties of these two events.
\begin{lemma}\label{lem::levellines_crossingproba}
Fix $\nu\le -4$ and set $\alpha=-(\nu+2)/2\ge 1$. Suppose $\eta$ is an $\SLE_4(\nu+2, -\nu-2)$ in $\HH$ from $x_1$ to $x_4$ with force points $(x_2, x_3)$. Let $T$ be its continuation threshold, then we have 
\[\PP[\eta(T)=x_2]=1-z^{\alpha},\quad\text{and}\quad \PP[\eta(T)=x_4]=z^{\alpha},\quad{where}\quad z=\frac{(x_2-x_1)(x_4-x_3)}{(x_3-x_1)(x_4-x_2)}.\]
\end{lemma} 
\begin{proof} 
Recall that the driving function of $\eta$ satisfies the following:
\[dW_t=2dB_t+\frac{-(\nu+2)dt}{g_t(x_2)-W_t}+\frac{(\nu+2)dt}{g_t(x_3)-W_t}+\frac{2dt}{g_t(x_4)-W_t}.\]
Define 
\begin{equation}\label{eqn::hSLE4_SLE4_mart}
M_t=z_t^{\alpha},\quad \text{where}\quad z_t=\frac{(g_t(x_2)-W_t)(g_t(x_4)-g_t(x_3))}{(g_t(x_3)-W_t)(g_t(x_4)-g_t(x_2))}.
\end{equation}
By It\^{o}'s Formula, one can check that $M_t$ is a local martingale for $\eta$. We see that, as $t\to T$, 
\[M_t\to 0,\quad \text{as }\eta(t)\to x_2;\quad\text{and}\quad M_t\to 1,\quad \text{as }\eta(t)\to x_4. \] 
Note that $0\le M_t\le 1$. Thus Optional Stopping Theorem implies that 
\[\PP[\eta(T)=x_4]=\E[M_T]=M_0=z^{\alpha}. \]
This gives the conclusion. 
\end{proof}
Note that the process $\hSLE_4(\nu)$ for $\nu\le -4$ defined through~\eqref{eqn::hSLE_hyperF_reflect} is nolonger the same as $\eta\sim \SLE_4(\nu+2, -\nu-2)$. In fact, it is the same as $\eta$ conditioned on $\{\eta(T)=x_2\}$. Or equivalently, $\hSLE_4(\nu)$ is the same as $\eta$ weighted by the martingale $1-M_t$ where $M_t$ is defined in~\eqref{eqn::hSLE4_SLE4_mart}. 

When $\kappa=4$, the hypergeometric $\SLE$ process degenerates: the hypergeometric term in the driving function becomes zero. The degeneracy indicates that the solution to~\eqref{eqn::euler_ode} is algebraic. This observation also holds for more general boundary conditions and we can derive the crossing probabilities---generalization of Lemma~\ref{lem::levellines_crossingproba}---for the general alternating boundary conditions, see \cite[Sections~5 and~6]{PeltolaWuGlobalMultipleSLEs}.

%% file: tex/hypersle_relation_different.tex
By conformal invariance, the $\hSLE_\kappa(\nu)$ can be defined in any quad $(\Omega; x_1, x_2, x_3, x_4)$ via conformal image. 
Denote by $\PP_{\kappa,\nu}(\Omega; x_1, x_2, x_3, x_4)$ the law of $\hSLE_{\kappa}(\nu)$ in $\Omega$ from $x_1$ to $x_4$ with marked points $(x_2, x_3)$. Proposition~\ref{prop::hSLE_conditioned} derives the relation between $\hSLE$s with different $\nu$'s. Proposition~\ref{prop::boundary_perturbation} derives the relation between $\hSLE$s in different domains. Proposition~\ref{prop::hsle_change_target} derives the relation between $\hSLE$s with different target points. 

\begin{proposition}\label{prop::hSLE_conditioned}
Fix $\kappa\in (0,8), \nu\in\R$ and quad $(\Omega; x_1, x_2, x_3, x_4)$. 
When $\nu\ge \kappa/2-4$, we have $\eta\cap (x_2x_3)=\emptyset$ almost surely. When $(-4)\vee(\kappa/2-6)<\nu<\kappa/2-4$, the event $\{\eta\cap (x_2x_3)=\emptyset\}$ has positive chance which is given by
\begin{equation}\label{eqn::hSLE_proba_avoid}
\frac{\PartF_{\kappa, \kappa-8-\nu}(\Omega; x_1, x_2, x_3, x_4)\Gamma((2\nu+8)/\kappa)\Gamma((\kappa-4-2\nu)/\kappa)}{\PartF_{\kappa, \nu}(\Omega; x_1, x_2, x_3, x_4)\Gamma((2\nu+12-\kappa)/\kappa)\Gamma((2\kappa-8-2\nu)/\kappa)}.   
\end{equation}
Moreover, for $(-4)\vee(\kappa/2-6)<\nu<\kappa/2-4$, we have
\[\PP_{\kappa,\nu}(\Omega; x_1, x_2, x_3, x_4)[\cdot\cond \eta\cap(x_2x_3)=\emptyset]=\PP_{\kappa,\kappa-8-\nu}(\Omega; x_1, x_2, x_3, x_4)[\cdot].\]
In particular, when $\kappa\in (4,8)$, the law of $\SLE_{\kappa}$ from $x_1$ to $x_4$ conditioned to avoid $(x_2x_3)$ is the same as $\hSLE_{\kappa}(\kappa-6)$ from $x_1$ to $x_4$ with marked points $(x_2, x_3)$. 
\end{proposition}

\begin{proof}
We may assume $\Omega=\HH$ and $x_1=0<x_2=x<x_3=y<x_4=\infty$. 
Let $\eta$ be an $\hSLE_{\kappa}(\nu)$ from $0$ to $\infty$ with marked points $(x, y)$. The fact that $\eta\cap (x, y)=\emptyset$ when $\nu\ge\kappa/2-4$ is proved in Proposition~\ref{prop::hyperSLE}. In the following, we assume $(-4)\vee(\kappa/2-6)<\nu<\kappa/2-4$. 

Set $\hat{\nu}=\kappa-8-\nu$ and $\hat{a}=(\hat{\nu}+2)/\kappa$, and let $\hat{\eta}$ be an $\hSLE_{\kappa}(\hat{\nu})$ from $0$ to $\infty$ with marked points $(x, y)$. The following process is a local martingale for $\eta$: 
\[M_t=z_t^{\hat{a}-a}\hat{F}(z_t)/F(z_t),\quad \text{where}\quad z_t=\frac{g_t(x)-W_t}{g_t(y)-W_t},\]
and $F$ is defined through \eqref{eqn::hSLE_hyperF} and 
\[\hat{F}(z)=\hF\left(\frac{2(\kappa-6-\nu)}{\kappa}, 1-\frac{4}{\kappa}, \frac{2(\kappa-4-\nu)}{\kappa}; z\right).\]
Moreover, the law of $\eta$ weighted by $M$ is the same as $\hat{\eta}$ up to the first time that $x$ is swallowed. Since $\hat{\nu}\ge \kappa/2-4$, we know that $\hat{\eta}$ does not hit the interval $(x, y)$. Thus the swallowing time of $x$ coincides with the swallowing time of $y$, denoted by $T_y$. Therefore, $M$ is a uniform integrable martingale for $\eta$.  As $t\to T_y$, we have $z_t\to 1$. Thus the law of $\hat{\eta}$ is the same as $\eta$ weighted by $\one_{\{\eta\cap (x, y)=\emptyset\}}$. 
In particular, we have 
\[\PP[\eta\cap (x, y)=\emptyset]=z^{\hat{a}-a}\frac{\hat{F}(z)/\hat{F}(1)}{F(z)/F(1)}, \quad\text{where}\quad z=x/y.\]
This gives \eqref{eqn::hSLE_proba_avoid}. 
\end{proof}

Next, we derive the boundary perturbation property of the $\hSLE_\kappa(\nu)$, this is a generalization of the boundary perturbation property of $\SLE_{\kappa}$ derived in \cite[Section 5]{LawlerSchrammWernerConformalRestriction}. 
Suppose $\hat{\Omega}\subset\Omega$ such that $\hat{\Omega}$ is simply connected and agrees with $\Omega$ in a neighborhood of $(x_1x_4)$. 
In the following proposition, we will derive the relation between $\PP_{\kappa,\nu}(\hat{\Omega}; x_1, x_2, x_3, x_4)$ and $\PP_{\kappa,\nu}(\Omega; x_1, x_2, x_3, x_4)$. To this end, we need to introduce Brownian loop measure. 

The \textit{Brownian loop measure} is 
a conformally invariant measure on unrooted Brownian loops 
in the plane. In the present article, we will not need the precise definition of this 
measure, so we content ourselves with referring to the literature for the definition: 
see, e.g.,~\cite{LawlerNoteBrownianLoop} or~\cite[Sections~3~and~4]{LawlerWernerBrownianLoopsoup}. 
Given a non-empty simply connected domain $\Omega\subsetneq\C$ and two disjoint subsets $V_1, V_2 \subset \Omega$, we
denote by $\mu(\Omega; V_1, V_2)$ the Brownian loop measure of loops in $\Omega$ that intersect both 
$V_1$ and $V_2$. This quantity is conformally invariant: 
$\mu(f(\Omega); f(V_1), f(V_2)) = \mu(\Omega; V_1, V_2)$ for any conformal transformation 
$f \colon \Omega \to f(\Omega)$.
In general, the Brownian loop measure is an infinite measure. By~\cite[Corollary~4.6]{LawlerNoteBrownianLoop}, we have $0 \leq \mu(\Omega; V_1, V_2) < \infty$ 
when both of $V_1, V_2$ are closed, one of them is compact, and $\dist(V_1, V_2) > 0$.

\begin{proposition}\label{prop::boundary_perturbation}
Fix $\kappa\in (0,4]$ and $\nu>-4$. Let $(\Omega; x_1, x_2, x_3, x_4)$ be a quad and assume that $\hat{\Omega}\subset\Omega$ is simply connected and it agrees with $\Omega$ in a neighbourhood of the arc $(x_1x_4)$. Then the $\hSLE_{\kappa}(\nu)$ in $\hat{\Omega}$ is absolutely continuous with respect to $\hSLE_{\kappa}(\nu)$ in $\Omega$, and the Radon-Nikodym derivative is given by 
\[\frac{d\PP_{\kappa,\nu}(\hat{\Omega}; x_1, x_2, x_3, x_4)}{d\PP_{\kappa,\nu}(\Omega; x_1, x_2, x_3, x_4)}=\frac{\PartF_{\kappa,\nu}(\Omega; x_1, x_2, x_3, x_4)}{\PartF_{\kappa,\nu}(\hat{\Omega}; x_1, x_2, x_3, x_4)}\one_{\{\eta\subset\hat{\Omega}\}}\exp(c\mu(\Omega; \eta, \Omega\setminus\hat{\Omega})),\]
where $c=(3\kappa-8)(6-\kappa)/(2\kappa)$ and $\PartF_{\kappa,\nu}$ is defined through \eqref{eqn::hSLE_partition} and \eqref{eqn::hSLE_partition_general}.
\end{proposition}
\begin{proof}
We may assume $\Omega=\HH$ and $x_1=0<x_2=x<x_3=y<x_4=\infty$. Let $\phi$ be the conformal map from $\hat{\Omega}$ onto $\HH$ such that $\phi(0)=0$ and 
$\lim_{z\to\infty}\phi(z)/z=1$. Suppose $\eta$ (resp. $\hat{\eta}$) is $\hSLE_{\kappa}(\nu)$ in $\HH$ (resp. in $\hat{\Omega}$) from 0 to $\infty$ with marked points $(x, y)$. Let $(g_t, t\ge 0)$ be the corresponding family of conformal maps, and $V_t^x, V_t^y$ are the evolution of $x, y$ respectively. Let $T$ be the first time when $\eta$ exits $\hat{\Omega}$. We will study the law of $\tilde{\eta}(t)=\phi(\eta(t))$ for $t<T$. 
Define $\tilde{g}_t$ to be the conformal map from $\HH\setminus\tilde{\eta}[0,t]$ onto $\HH$ normalized 
at $\infty$ and let $\varphi_t$ be the conformal map from $\HH\setminus g_t(K)$ onto $\HH$ such that 
$\varphi_t\circ g_t=\tilde{g}_t\circ\phi$. One can check that the following process is a local martingale for $\eta$:
\begin{align*}
M_t:=&\one_{\{t<T\}}\varphi_t'(W_t)^h\varphi_t'(V_t^x)^b\varphi_t'(V_t^y)^b\left(\frac{\varphi_t(V_t^y)-\varphi_t(V_t^x)}{V_t^y-V_t^x}\right)^{-2b}\exp(c\mu(\HH; \eta[0,t], \HH\setminus\hat{\Omega}))\\
&\times \left(\frac{\varphi_t(V_t^x)-\varphi_t(W_t)}{\varphi_t(V_t^y)-\varphi_t(W_t)}\Big{/}\frac{V_t^x-W_t}{V_t^y-W_t}\right)^{a}
\times F\left(\frac{\varphi_t(V_t^x)-\varphi_t(W_t)}{\varphi_t(V_t^y)-\varphi_t(W_t)}\right)\Big{/}F\left(\frac{V_t^x-W_t}{V_t^y-W_t}\right),
\end{align*}
where $a, b, h$ are defined through \eqref{eqn::hSLE_constants} and $F$ is defined through \eqref{eqn::hSLE_hyperF}. Moreover, the law of $\eta$ weighted by $M$ is the same as $\hat{\eta}$ up to $t<T$. Since $\kappa\le 4$, the process  $\hSLE_{\kappa}(\nu)$ in $\hat{\Omega}$ never exits $\hat{\Omega}$ and goes to $\infty$. Therefore, $M$ is a uniformly integrable martingale for $\eta$ and the law of $\eta$ weighted by $M_T/M_0$ is the same as $\hat{\eta}$ where 
\[M_T:=\lim_{t\to T}M_t=\one_{\{\eta\subset\hat{\Omega}\}}\exp(c\mu(\HH; \eta, \HH\setminus\hat{\Omega})).\]
This completes the proof. 
\end{proof}

\begin{proposition}\label{prop::hsle_change_target}
Fix $\kappa\in (0,8)$ and a quad $(\Omega; x_1, x_2, x_3, x_4)$. Let $\eta$ be an $\hSLE_{\kappa}(\kappa-8)$ in $\Omega$ from $x_1$ to $x_4$ with marked points $(x_2, x_3)$ conditioned to hit $[x_2,x_3]$. Let $\tilde{\eta}$ be an $\hSLE_{\kappa}$ in $\Omega$ from $x_1$ to $x_2$ with marked points $(x_4, x_3)$. 
Then $\tilde{\eta}$ has the same law as $\eta$ up to the swallowing time of $x_2$.
\end{proposition} 
\begin{proof}
We may assume $\Omega=\HH$ and $x_1<x_2<x_3<x_4$. 
Let $T$ (resp. $\tilde{T}$) be the swallowing time of $x_2$ for $\eta$ (resp. for $\tilde{\eta}$). 
Denote $X_{j1}=g_t(x_j)-W_t$ for $2\le j\le 4$ and $X_{ij}=g_t(x_j)-g_t(x_i)$ for $2\le i<j\le 4$. 
When $\nu=\kappa-8$, we have 
\[a=\frac{\nu+2}{\kappa}=-2h,\quad b=\frac{(\nu+2)(\nu+6-\kappa)}{4\kappa}=h. \]

First, we assume $\kappa\in (4, 8)$. In this case, we have $\kappa-8>\kappa/2-6$. 
Define 
\[F(z):=\hF\left(2-\frac{12}{\kappa}, 1-\frac{4}{\kappa}, 2-\frac{8}{\kappa}; z\right),\quad \tilde{F}(z):=\hF\left(\frac{4}{\kappa}, 1-\frac{4}{\kappa}, \frac{8}{\kappa}; z\right).\]
When $\kappa\in (4,8)$, both $F$ and $\tilde{F}$ are bounded for $z\in [0,1]$.
The law of $\eta$ is the same as $\SLE_{\kappa}$ in $\HH$ from $x_1$ to $\infty$ weighted by the following local martingale up to $T$: 
\[M_t=g_t'(x_2)^hg_t'(x_3)^hg_t'(x_4)^hX_{41}^{-2h}X_{32}^{-2h}z_t^{-2h}F(z_t),\quad\text{where}\quad z_t=\frac{X_{21}X_{43}}{X_{31}X_{42}}.\]
The law of $\tilde{\eta}$ is the same as $\SLE_{\kappa}$ in $\HH$ from $x_1$ to $\infty$ weighted by the following local martingale up to $\tilde{T}$: 
\[\tilde{M}_t=g_t'(x_2)^{h}g_t'(x_3)^hg_t'(x_4)^hX_{21}^{-2h}X_{43}^{-2h}s_t^{2/\kappa}\tilde{F}(s_t), \quad\text{where}\quad s_t=\frac{X_{32}X_{41}}{X_{31}X_{42}}.\]
Comparing these two local martingales, we see that the law of $\tilde{\eta}$ is the same as $\eta$ weighted by the following local martingale up to $T$:
\[N_t=(1-z_t)^{8/\kappa-1}\tilde{F}(1-z_t)/F(z_t).\]
As $t\to T<\infty$, we have $z_t\to 0$. Thus $N_T=\one_{\{T<\infty\}}\tilde{F}(1)/F(0)$. Therefore, the law of $\tilde{\eta}$ is the same as $\eta$ up to the swallowing time of $x_2$. 

Next, we assume $\kappa\in (0,4]$. In this case, we have $\kappa-8\le -4$. Define 
\[G(z):=\hF\left(1-\frac{4}{\kappa}, \frac{4}{\kappa}, \frac{8}{\kappa}; z\right). \]
The law of $\eta$ is the same as $\SLE_{\kappa}$ in $\HH$ from $x_1$ to $\infty$ weighted by the following local martingale up to $T$: 
\[M_t=g_t'(x_2)^hg_t'(x_3)^hg_t'(x_4)^hX_{41}^{-2h}X_{32}^{-2h}z_t^{-2h}(1-z_t)^{8/\kappa-1}G(1-z_t). \]
Therefore, the law of $\tilde{\eta}$ is the same as $\eta$ weighted by the following local martingale up to $T$: 
\[N_t=\tilde{F}(1-z_t)/G(1-z_t).\]
As $t\to T<\infty$, we have $z_t\to 0$. Thus $N_T=\one_{\{T<\infty\}}\tilde{F}(1)/G(1)$. This gives the conclusion for $\kappa\le 4$ and completes the proof.  
\end{proof}

%% file: tex/slepair_statements.tex
Fix a quad $q=(\Omega; x^R, y^R, y^L, x^L)$, consider disjoint continuous simple curves $(\eta^L;\eta^R)\in X_0(\Omega; x^R, y^R, y^L, x^L)$. Let $T^R$ be the first time that $\eta^R$ hits the closed arc $[y^Ry^L]$, and denote by $w^R$ the point $\eta^R(T^R)$, and by $\Omega^R$ the connected component of $\Omega\setminus\eta^R[0,T^R]$ with $(y^Lx^L)$ on the boundary. Note that $T^R, w^R, \Omega^R$ are deterministic functions of $\eta^R$. Let $T^L$ be the first time that $\eta^L$ hits the closed arc $[y^Ry^L]$, and denote by $w^L$ the point $\eta^L(T^L)$, and by $\Omega^L$ the connected component of $\Omega\setminus\eta^L[0,T^L]$ with $(x^Ry^R)$ on the boundary. See Figure~\ref{fig::def_forcepoint}. 

\begin{figure}[ht!]
\begin{subfigure}[b]{0.5\textwidth}
\begin{center}
\includegraphics[width=0.7\textwidth]{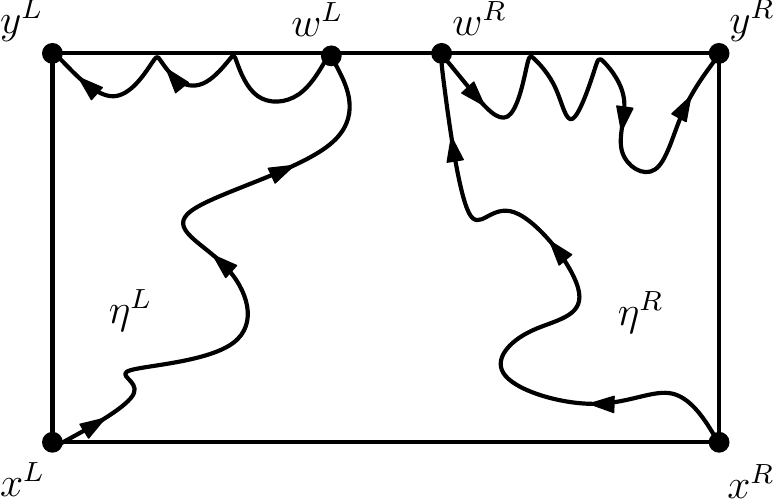}
\end{center}
\end{subfigure}
\begin{subfigure}[b]{0.5\textwidth}
\begin{center}\includegraphics[width=0.7\textwidth]{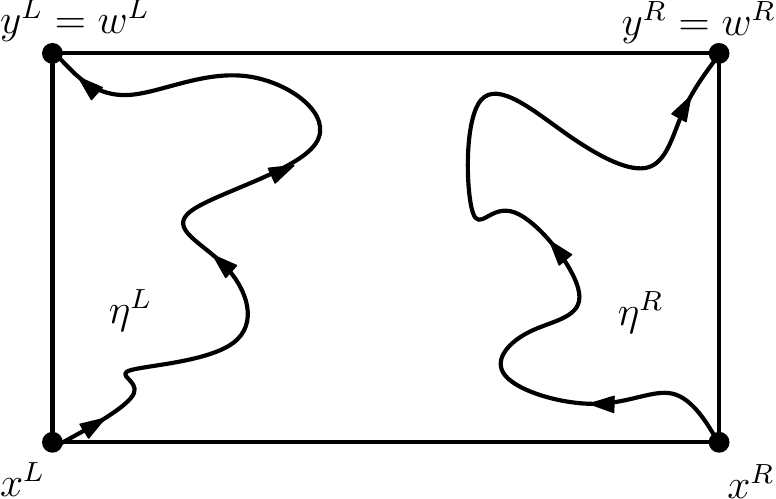}
\end{center}
\end{subfigure}
\caption{\label{fig::def_forcepoint} Let $T^R$ be the first time that $\eta^R$ hits the closed arc $[y^Ry^L]$, and denote by $w^R$ the point $\eta^R(T^R)$, and by $\Omega^R$ the connected component of $\Omega\setminus\eta^R[0,T^R]$ with $(y^Lx^L)$ on the boundary. Note that $T^R, w^R, \Omega^R$ are deterministic functions of $\eta^R$. We define $T^L, w^L, \Omega^L$ for $\eta^L$ similarly. }
\end{figure}

\begin{proposition}\label{prop::slepair_rev}
Assume the same notations as in Figure~\ref{fig::def_forcepoint}. 
Fix $\kappa\in (0,4]$ and $\rho^L>-2, \rho^R>-2$ and quad $q=(\Omega; x^R, y^R, y^L, x^L)$.
\begin{itemize}
\item (Existence and Uniqueness) 
There exists a unique probability measure on disjoint continuous simple curves $(\eta^L;\eta^R)\in X_0(\Omega; x^R, y^R, y^L, x^L)$ such that 
the conditional law of $\eta^R$ given $\eta^L$ is $\SLE_{\kappa}(\rho^R)$ in $\Omega^L$ from $x^R$ to $y^R$ with force point $x^R_+$;
and that the conditional law of $\eta^L$ given $\eta^R$ is $\SLE_{\kappa}(\rho^L)$ in $\Omega^R$ from $x^L$ to $y^L$ with force point $x^L_-$.
\item (Identification) Under this probability measure and fix $\rho^L=0$ and $\rho^R>-2$, the marginal law of $\eta^L$ is $\hSLE_{\kappa}(\rho^R)$ in $\Omega$ from $x^L$ to $y^L$ with marked points $(x^R, y^R)$.  
\end{itemize}
\end{proposition}

\begin{proposition}\label{prop::slepair_cmp_sym}
Assume the same notations as in Figure~\ref{fig::def_forcepoint}. 
Fix $\kappa\in (0,4]$ and $\rho>-2$ and quad $q=(\Omega; x^R, y^R, y^L, x^L)$.
\begin{itemize}
\item (Existence and Uniqueness) 
There exists a unique probability measure on disjoint continuous simple curves $(\eta^L;\eta^R)\in X_0(\Omega; x^R, y^R, y^L, x^L)$ such that 
the conditional law of $\eta^R$ given $\eta^L$ is $\SLE_{\kappa}(\rho)$ in $\Omega^L$ from $x^R$ to $y^R$ with force point $w^L$; 
and that the conditional law of $\eta^L$ given $\eta^R$ is $\SLE_{\kappa}(\rho)$ in $\Omega^R$ from $x^L$ to $y^L$ with force point $w^R$. We denote this probability measure by $\QQ_q(\kappa,\rho)$.
\item (Identification) 
Under this probability measure, the marginal of $\eta^R$ stopped at the first hitting time of $[y^Ry^L]$ is the same as $\hSLE_{\kappa}(\kappa-8-\rho)$ in $\Omega$ from $x^R$ to $x^L$ with marked points $(y^R, y^L)$ conditioned to hit $[y^Ry^L]$, stopped at the swallowing time of $y^R$. 
\end{itemize} 
\end{proposition}

The main result of this section is proving these two propositions. 
The uniqueness part in Propositions~\ref{prop::slepair_rev} and \ref{prop::slepair_cmp_sym} were proved in \cite[Theorem 4.1]{MillerSheffieldIG2}. The key point in the existence part of Proposition~\ref{prop::slepair_rev} is Proposition~\ref{prop::boundary_perturbation}. The key point in the existence part of Proposition~\ref{prop::slepair_cmp_sym} is ``Commutation Relation" derived by J. Dub\'{e}dat. 
Proposition~\ref{prop::slepair_cmp_sym} is the main ingredient in the proof of Theorem~\ref{thm::SLEpair_CMP_SYM}.

%% file: tex/slepair_proof_rev.tex
The uniqueness in Proposition~\ref{prop::slepair_rev} was proved in the proof of \cite[Theorem 4.1]{MillerSheffieldIG2}. 
See \cite[Section~3.2]{BeffaraPeltolaWuUniqueness} for another proof of the uniqueness. 
We only need to show the existence and the identification. 
To construct the pair $(\eta^L;\eta^R)$ in Proposition~\ref{prop::slepair_rev}, we need to introduce boundary perturbation property of $\SLE_{\kappa}(\rho)$ process. This is a particular case of Proposition~\ref{prop::boundary_perturbation}. 

\begin{lemma}\label{lem::sle_domain_mart}
\cite[Section 3]{WernerWuCLEtoSLE}. 
Fix $\kappa\in (0,4], \rho>-2$. Let $(\Omega; x, y)$ be a Doburshin domain such that $x, y$ lie on analytic segments of $\partial\Omega$. Assume that $\hat{\Omega}\subset\Omega$ is simply connected and it agrees with $\Omega$ in a neighborhood of the arc $(xy)$. Then $\SLE_{\kappa}(\rho)$ in $\hat{\Omega}$ from $x$ to $y$ with force point $x_+$ is absolutely continuous with respect to $\SLE_{\kappa}(\rho)$ in $\Omega$ from $x$ to $y$ with force point $x_+$, and the Radon-Nikodym derivative is given by 
\[\one_{\{\eta\subset\hat{\Omega}\}}(\varphi'(x)\varphi'(y))^{-b}\exp(c\mu(\Omega; \eta, \Omega\setminus\hat{\Omega})),\quad
\text{where } b=\frac{(\rho+2)(\rho+6-\kappa)}{4\kappa},\quad c=\frac{(3\kappa-8)(6-\kappa)}{2\kappa},\]
and $\mu$ is Brownian loop measure, and $\varphi$ is any conformal map from $\hat{\Omega}$ onto $\Omega$ fixing $x$ and $y$. 
\end{lemma}

\begin{proof}[Proof of Proposition~\ref{prop::slepair_rev}, Existence and Identification] 
\textit{First}, we will construct a probability measure on $(\eta^L;\eta^R)\in X_0(\Omega; x^L, x^R, y^R, y^L)$. 
Fix $\kappa\in (0,4], \rho^L>-2, \rho^R>-2$ and $0<x<y$. By conformal invariance, it is sufficient to give the construction for the quad $(\HH; 0,x,y,\infty)$. 
Denote by $\PP_L$ the law of $\SLE_{\kappa}(\rho^L)$ in $\HH$ from 0 to $\infty$ with force point $0_-$ and denote by $\PP_R$ the law of $\SLE_{\kappa}(\rho^R)$ in $\HH$ from $x$ to $y$ with force point $x_+$. Define measure $\mathcal{M}$ on $X_0(\HH; 0,x,y,\infty)$ by 
\[\mathcal{M}[d\eta^L, d\eta^R]=\one_{\{\eta^L\cap\eta^R=\emptyset\}}\exp\left(c \mu(\HH; \eta^L, \eta^R)\right)\PP_L\left[d\eta^L\right]\otimes\PP_R\left[d\eta^R\right].\] 
We argue that the total mass of $\mathcal{M}$, denoted by $|\mathcal{M}|$, is finite. Given $\eta^L\in X_0(\HH; 0,\infty)$, let $g$ be any conformal map from the connected component of $\HH\setminus \eta^L$ with $(xy)$ on the boundary onto $\HH$, then we have
\begin{align*}
|\mathcal{M}|&=\E_L\otimes\E_R\left[\one_{\{\eta^L\cap\eta^R=\emptyset\}}\exp\left(c \mu(\HH; \eta^L, \eta^R)\right)\right]\\
&=\E_L\left[\left(\frac{g'(x)g'(y)}{(g(x)-g(y))^2}\right)^b\right]\tag{By Lemma~\ref{lem::sle_domain_mart}}\\
&\le (y-x)^{-2b}. \tag{where $b=(\rho^R+2)(\rho^R+6-\kappa)/(4\kappa)$}
\end{align*}
This implies that $|\mathcal{M}|$ is positive finite. We define the probability measure $\mathcal{M}^{\sharp}$ to be $\mathcal{M}/|\mathcal{M}|$. 
\smallbreak
\textit{Second}, we show that, under $\mathcal{M}^{\sharp}$, the conditional law of $\eta^R$ given $\eta^L$ is $\SLE_{\kappa}(\rho^R)$. By the symmetry in the definition of $\mathcal{M}$, we know that the conditional law of $\eta^L$ given $\eta^R$ is $\SLE_{\kappa}(\rho^L)$. Given $\eta^L$, denote by $D$ the connected component of $\HH\setminus \eta^L$ with $(xy)$ on the boundary and let $g$ be any conformal map from $D$ onto $\HH$. Denote by $\PP_R$ the law of $\SLE_{\kappa}(\rho^R)$ in $\HH$ from $x$ to $y$ and by $\tilde{\PP}_R$ the law of $\SLE_{\kappa}(\rho^R)$ in $D$ from $x$ to $y$. By Lemma~\ref{lem::sle_domain_mart},
for any bounded continuous function $\LF$ on continuous curves, we have 
\begin{align*}
\mathcal{M}^{\sharp}\left[\LF(\eta^R)\cond \eta^L\right]
&=|\mathcal{M}|^{-1}\E_R\left[\one_{\{\eta^L\cap\eta^R=\emptyset\}}\exp\left(c\mu(\HH; \eta^L, \eta^R)\right)\LF(\eta^R)\right]\\
&=|\mathcal{M}|^{-1}\left(\frac{g'(x)g'(y)}{(g(x)-g(y))^2}\right)^b\tilde{\E}_R\left[\LF(\eta^R)\right].
\end{align*}
This implies that the conditional law of $\eta^R$ given $\eta^L$ is $\SLE_{\kappa}(\rho^R)$ in $D$. 
\smallbreak
\textit{Finally}, we show that, under $\mathcal{M}^{\sharp}$ and fixing $\rho^L=0$, the marginal law of $\eta^L$ is $\hSLE_{\kappa}(\rho^R)$. In fact, the above equation implies that the law of $\eta^L$ is the law of $\SLE_{\kappa}$ in $\HH$ from 0 to $\infty$ weighted by \[\left(\frac{g'(x)g'(y)}{(g(x)-g(y))^2}\right)^b.\]
By Proposition~\ref{prop::hypersle_mart}, we see that the law of $\eta^L$ coincides with $\hSLE_{\kappa}(\rho^R)$ as desired.
\end{proof}

%% file: tex/slepair_commutation.tex
In \cite{DubedatCommutationSLE} and \cite[Appendix A]{KytolaPeltolaPurePartitionFunctions}, the authors studied local multiple SLEs and classify them according to the so-called partition functions. Following the same idea, we will define a local $\SLE$ that describes two initial segments with two extra marked points. Fix a quad $q=(\Omega; x^R, y^R, y^L, x^L)$. We will study a local $\SLE$ in $\Omega$ that describes two initial segments $\gamma_1$ and $\gamma_4$ starting from $x_1$ and $x_4$ respectively, with two extra marked points $x_2$ and $x_3$, up to exiting some neighborhoods $U_1$ and $U_4$.  The localization neighborhoods $U_1$ and $U_4$ are assumed to be closed subsets of $\overline{\Omega}$ such that $\Omega\setminus U_j$ are simply connected for $j=1,4$ and that $U_1\cap U_4=\emptyset$ and that $\dist(\{x_2, x_3\}, U_1\cup U_4)>0$. Define 
\[\chamber_4=\{(x_1, x_2, x_3, x_4)\in\R^4: x_1<x_2<x_3<x_4\}.\]

The local $\SLE_{\kappa}$ in $\Omega$, started from $(x_1, x_4)$ and localized in $(U_1, U_4)$ with two marked points $(x_2, x_3)$, is a probability measure on two curves $(\gamma_1, \gamma_4)$ such that, for $j\in\{1,4\}$, the curve $\gamma_j: [0,1]\to U_j$ starts at $\gamma_j(0)=x_j$ and ends at $\gamma_j(1)\in\partial U_j$. The local $\SLE_{\kappa}$ is the indexed collection 
\[P=\left(P_{(q; U_1, U_4)}\right)_{q; U_1, U_4}.\] 
This collection of probability measures is required to satisfy the following three properties.
\begin{itemize}
\item Conformal invariance. Suppose that $q=(\Omega; x^R, y^R, y^L, x^L), \tilde{q}=(\tilde{\Omega}; \tilde{x}^R, \tilde{y}^R, \tilde{y}^L, \tilde{x}^L)\in\LQ$, and $\psi: \Omega\to \tilde{\Omega}$ is a conformal map with $\psi(x^R)=\tilde{x}^R, \psi(y^R)=\tilde{y}^R, \psi(y^L)=\tilde{y}^L, \psi(x^L)=\tilde{x}^L$. Then for $(\gamma_1, \gamma_4)\sim P_{(q; U_1, U_4)}$, we have $(\psi(\gamma_1), \psi(\gamma_4))\sim P_{(\tilde{q}; \psi(U_1), \psi(U_4))}$. 
\item Domain Markov property. Suppose that $\tau_1$ is a stopping time for $\gamma_1$ and $\tau_4$ is a stopping time for $\gamma_4$. The conditional law of $(\gamma_1|_{t\ge\tau_1}, \gamma_4|_{t\ge\tau_4})$, given the initial segments $\gamma_1[0,\tau_1]$ and $\gamma_4[0,\tau_4]$, is the same as $P_{(\tilde{q}; \tilde{U}_1, \tilde{U}_4)}$ where $\tilde{q}=(\tilde{\Omega}; \gamma_1(\tau_1), x_2, x_3, \gamma_4(\tau_4))$ and $\tilde{\Omega}$ is the connected component of $\Omega\setminus (\gamma_1[0,\tau_1]\cup\gamma_4[0,\tau_4])$ with $(x_2x_3)$ on the boundary, and $\tilde{U}_j=U_j\cap \tilde{\Omega}$ for $j\in\{1,4\}$. 
\item Absolute continuity of the marginals. There exist smooth functions $F_j: \chamber_4\to\R$, for $j\in\{1,4\}$, such that for the domain $\Omega=\HH$, boundary points $x_1<x_2<x_3<x_4$, and localization neighborhoods $U_1$ and $U_4$, the marginal law of $\gamma_j$ under $P_{(\HH; x_1,x_2,x_3,x_4; U_1, U_4)}$ is the Loewner chain driven by the solution to the following SDEs: 
\begin{align}\label{eqn::localsle_sde}
\begin{split}
\text{for }\gamma_1:&\quad dW_t=\sqrt{\kappa}dB_t+F_1(W_t, V_t^2, V_t^3, V_t^4)dt,\quad dV_t^i=\frac{2dt}{V_t^i-W_t},\quad \text{for }i=2,3,4;\\
\text{for }\gamma_4:&\quad d\tilde{W}_t=\sqrt{\kappa}d\tilde{B}_t+F_4(\tilde{V}^1_t, \tilde{V}_t^2, \tilde{V}_t^3, \tilde{W}_t)dt,\quad d\tilde{V}_t^i=\frac{2dt}{\tilde{V}_t^i-W_t},\quad \text{for }i=1,2,3;
\end{split}
\end{align}
where $W_0=x_1, V_0^2=x_2, V_0^3=x_3$ and $V_0^{4}=x_4$ and $\tilde{W}_0=x_4, \tilde{V}^1_0=x_1, \tilde{V}_0^2=x_2, \tilde{V}_0^3=x_3$. 
\end{itemize}

\begin{lemma}\label{lem::localsle_compatible}
Suppose both $(U_1, U_4)$ and $(V_1, V_4)$ are localization neighborhoods for quad $q=(\Omega; x_1, x_2, x_3, x_4)$ and that $V_j\subset U_j$ for $j\in\{1,4\}$. Suppose $(\gamma_1, \gamma_4)\sim P_{(q; U_1, U_4)}$ and let $\tau_j$ be $\gamma_j$'s first time to exit $V_j$ for $j\in\{1,4\}$. Then $(\gamma_1|_{[0,\tau_1]}, \gamma_4|_{[0,\tau_4]})\sim P_{(q; V_1, V_4)}$. 
\end{lemma}
\begin{proof}
It is clear that the restriction measures also satisfy all the required three properties.
\end{proof}

It turns out that the existence of local $\SLE$ with two extra points is related to positive functions which satisfy a certain PDE system and  conformal covariance: $h=(6-\kappa)/(2\kappa)$ and $b$ is a constant parameter,
\begin{itemize}
\item PDE system (PDE):
\begin{align}\label{eqn::hsle_partition_pde}
\begin{split}
&\frac{\kappa}{2}\partial_{x_1}^2\PartF+\sum_{2\le i\le 4}\frac{2\partial_{x_i}\PartF}{x_i-x_1}+\left(\frac{-2h}{(x_4-x_1)^2}+\frac{-2b}{(x_2-x_1)^2}+\frac{-2b}{(x_3-x_1)^2}\right)\PartF=0,\\
&\frac{\kappa}{2}\partial_{x_4}^2\PartF+\sum_{1\le i\le 3}\frac{2\partial_{x_i}\PartF}{x_i-x_4}+\left(\frac{-2h}{(x_1-x_4)^2}+\frac{-2b}{(x_2-x_4)^2}+\frac{-2b}{(x_3-x_4)^2}\right)\PartF=0,
\end{split}
\end{align}
\item Conformal covariance (COV): for all M\"{o}bius maps $\varphi$ of $\HH$ such that $\varphi(x_1)<\varphi(x_2)<\varphi(x_3)<\varphi(x_4)$,  
\begin{equation}\label{eqn::hsle_partition_cov}
\PartF(x_1, x_2, x_3, x_4)=\varphi'(x_1)^h\varphi'(x_2)^{b}\varphi'(x_3)^{b}\varphi'(x_4)^h\times \PartF(\varphi(x_1), \varphi(x_2), \varphi(x_3), \varphi(x_4)). 
\end{equation}
\end{itemize}

\begin{proposition}\label{prop::localsle_existence}
We have the following correspondence between local $\SLE$ with two extra marked points and positive solutions to PDE~\eqref{eqn::hsle_partition_pde} and COV~\eqref{eqn::hsle_partition_cov}.
\begin{enumerate}
\item [(a)]
Suppose $\PartF:\chamber_4\to\R$ is a positive solution to PDE~\eqref{eqn::hsle_partition_pde} and COV~\eqref{eqn::hsle_partition_cov}. Then there exists a local $\SLE_{\kappa}$ with two extra marked points such that the drift terms in~\eqref{eqn::localsle_sde} are given by $F_1=\kappa\partial_{x_1}\log\PartF$ and $F_4=\kappa\partial_{x_4}\log\PartF$. 
\item [(b)] Suppose there exists a local $\SLE_{\kappa}$ with two extra marked points. Then there exists a positive solution $\PartF:\chamber_4\to\R$ to PDE~\eqref{eqn::hsle_partition_pde} and COV~\eqref{eqn::hsle_partition_cov} such that the drift terms in~\eqref{eqn::localsle_sde} are given by $F_1=\kappa\partial_{x_1}\log\PartF$ and $F_4=\kappa\partial_{x_4}\log\PartF$.
\end{enumerate}
\end{proposition}
\begin{proof}[Proof of Proposition~\ref{prop::localsle_existence}---Part (a)]
There are two ways to sample $\gamma_1$ and $\gamma_4$: Method 1---sample $\gamma_1$ first, and  Method 2---sample $\gamma_4$ first. 

\textit{Method 1.} 
Since $\PartF$ satisfies PDE~\eqref{eqn::hsle_partition_pde}, the following process is a local martingale with respect to the law of $\SLE_{\kappa}$ in $\HH$ from $x_1$ to $\infty$: 
\[M_t^{(1)}=g_t'(x_2)^{b}g_t'(x_3)^{b}g_t'(x_4)^h\PartF(W_t, g_t(x_2), g_t(x_3), g_t(x_4)).\]
We sample $\gamma_1$ according to the law of $\SLE_{\kappa}$ in $\HH$ from $x_1$ to $\infty$ weighted by the local martingale $M_t^{(1)}$, up to the first time $\sigma_1$ that the process exits $U_1$. Let $G=g_{\sigma_1}$ and denote by \[\tilde{x}_1=G(\gamma_1(\sigma_1)),\quad \tilde{x}_2=G(x_2),\quad \tilde{x}_3=G(x_3),\quad \tilde{x}_4=G(x_4).\]
Since $\PartF$ satisfies PDE~\eqref{eqn::hsle_partition_pde}, the following process is a local martingale with respect to the law of $\SLE_{\kappa}$ in $\HH$ from $x_4$ to $\infty$: 
\[\tilde{M}_s^{(4)}=\tilde{g}_s'(\tilde{x}_1)^h\tilde{g}_s'(\tilde{x}_2)^b\tilde{g}_s'(\tilde{x}_3)^b\PartF(\tilde{g}_s(\tilde{x}_1), \tilde{g}_s(\tilde{x}_2), \tilde{g}_s(\tilde{x}_3), \tilde{W}_s).\]
We sample $\tilde{\gamma}_4$ according to the the law of $\SLE_{\kappa}$ in $\HH$ from $\tilde{x}_4$ to $\infty$ weighted by the local martingale $\tilde{M}_s^{(4)}$, up to the first time $\tilde{\sigma}_4$ that the process exits $G(U_4)$. Finally, set $\gamma_4=G^{-1}(\tilde{\gamma}_4)$.  

\textit{Method 2.} This is defined in the same way as in Method 1 except we switch the roles of $\gamma_1$  and $\gamma_4$. 

According to the local commutation relation in \cite[Theorem 7.1]{DubedatCommutationSLE}, these two methods give the same law on pairs $(\gamma_1, \gamma_4)$. The probability measure defined by the sampling procedure clearly satisfies the domain Markov property and the absolute continuity of the marginals. By COV~\eqref{eqn::hsle_partition_cov}, we could define the law on $(\gamma_1, \gamma_4)$ in any simple connected domain via conformal image. This implies the conformal invariance. 
\end{proof}

\begin{proof}[Proposition~\ref{prop::localsle_existence}---Part (b)]
Since the local $\SLE$ with extra two points is conformal invariant, we could assume $x_2=\infty, x_3=0, x_4=x, x_1=y$ for $0<x<y$. By \cite[Theorem 7.1]{DubedatCommutationSLE}, the existence of local $\SLE_{\kappa}$ in neighborhoods of $x$ and $y$ with two extra marked points $0$ and $\infty$ implies that there exists a positive function $\psi: \chamber_2\to\R$ that solves the following PDE system: 
\begin{align*}
&\frac{\kappa}{2}\partial_{x}^2\psi+\frac{2}{x}\partial_x\psi+\left(\frac{2}{x}+\frac{2}{y-x}\right)\partial_y\psi+\left(\frac{-2h}{(y-x)^2}+\frac{-\mu}{x^2}\right)\psi=0,\\
&\frac{\kappa}{2}\partial_{y}^2\psi+\frac{2}{y}\partial_{y}\psi+\left(\frac{2}{y}+\frac{2}{x-y}\right)\partial_{x}\psi+\left(\frac{-2h}{(x-y)^2}+\frac{-\mu}{y^2}\right)\psi=0,
\end{align*}
where $\mu$ is a constant parameter. Moreover, the marginal laws of $\gamma_1, \gamma_4$ is the Loewner chain driven by the solution to the following SDEs:
\begin{align*}
\text{for }\gamma_1: dW_t&=\sqrt{\kappa}dB_t+\kappa(\partial_y\log\psi)(V_t^4-V_t^3, W_t-V_t^3)dt,\quad dV_t^i=\frac{2dt}{V_t^i-W_t}, \quad i=3,4;\\
\text{for }\gamma_4: d\tilde{W}_t&=\sqrt{\kappa}dB_t+\kappa(\partial_x\log\psi)(\tilde{W}_t-\tilde{V}^3_t, \tilde{V}^1_t-\tilde{V}^3_t)dt,\quad d\tilde{V}_t^i=\frac{2dt}{\tilde{V}_t^i-\tilde{W}_t},\quad i=1,3.
\end{align*}

By the argument in \cite[Sections 4 and 5]{GrahamSLE}, we see that $\psi$ needs to satisfy the following conformal covariance: for any M\"{o}bius maps $\varphi$ of $\HH$ with $\varphi(x)<\varphi(y)$, 
\begin{equation*}
\psi(x,y)=\varphi'(x)^h\varphi'(y)^h \psi(\varphi(x), \varphi(y)). 
\end{equation*}

We write $\phi(x,y)=(y-x)^{2h}\psi(x,y)$. Then $\phi$ satisfies the following PDE system:
\begin{align*}
\frac{\kappa}{2}\partial_x^2\phi+\left(\frac{2}{x}+\frac{6-\kappa}{y-x}\right)\partial_x\phi+\left(\frac{2}{x}+\frac{2}{y-x}\right)\partial_y\phi+\frac{-\mu}{x^2}\phi=0,\\
\frac{\kappa}{2}\partial_y^2\phi+\left(\frac{2}{y}+\frac{6-\kappa}{x-y}\right)\partial_y\phi+\left(\frac{2}{y}+\frac{2}{x-y}\right)\partial_x\phi+\frac{-\mu}{y^2}\phi=0.
\end{align*}
Moreover, $\phi$ is conformal invariant: for any M\"{o}bius maps $\varphi$ of $\HH$ with $\varphi(x)<\varphi(y)$, we have $\phi(x,y)=\phi(\varphi(x), \varphi(y))$. Since $\phi$ is conformal invariant, it only depends on the ratio $x/y\in (0,1)$. Thus we may write $\phi(x,y)=f(x/y)$. Then $f$ satisfies the following ODE:
\begin{equation}\label{eqn::localsle_conformal_invariant_ode}
\frac{\kappa}{2}z^2f''(z)+\frac{z(2+(2-\kappa)z)}{1-z}f'(z)-\mu f(z)=0.
\end{equation}
Define, for $x_1<x_2<x_3<x_4$, 
\[\PartF(x_1, x_2, x_3, x_4):=(x_4-x_1)^{-2h}(x_3-x_2)^{-\mu}f(z),\quad \text{where }z=\frac{(x_2-x_1)(x_4-x_3)}{(x_3-x_1)(x_4-x_2)}.\]
Since $f$ satisfies ODE~\eqref{eqn::localsle_conformal_invariant_ode} and is conformal invariant, we could check that $\PartF$ satisfies PDE~\eqref{eqn::hsle_partition_pde} and COV~\eqref{eqn::hsle_partition_cov} with $b=\mu/2$. 
\end{proof}

\begin{corollary}\label{cor::hsle_commutation}
For any $\kappa\in (0,8)$ and $\nu\in\R$, 
there exists a local $\SLE_{\kappa}$ with two extra marked points such that the drift term in~\eqref{eqn::localsle_sde} are give by 
$F_1=\kappa\partial_{x_1}\log\PartF_{\kappa,\nu}$ and $F_4=\kappa\partial_{x_4}\log\PartF_{\kappa,\nu}$ where $\PartF_{\kappa,\nu}$ is defined in~\eqref{eqn::hSLE_partition}. In particular, the marginal law of $\gamma_1$ is $\hSLE_{\kappa}(\nu)$ in $\HH$ from $x_1$ to $x_4$ with marked points $(x_2, x_3)$ stopped at the first exiting time of $U_1$, and the marginal law of $\gamma_4$ is $\hSLE_{\kappa}(\nu)$ in $\HH$ from $x_4$ to $x_1$ with marked points $(x_3, x_2)$ stopped at the first exiting time of $U_4$. 
\end{corollary}
\begin{proof}
The function $\PartF_{\kappa,\nu}$ defined in~\eqref{eqn::hSLE_partition} satisfies PDE~\eqref{eqn::hsle_partition_pde} and COV~\eqref{eqn::hsle_partition_cov} for 
\[b=(\nu+2)(\nu+6-\kappa)/(4\kappa).\] Combining with Proposition~\ref{prop::localsle_existence}, Part (a), we obtain the conclusion. \end{proof}

%% file: tex/slepair_proof_sym.tex
\begin{proof}[Proof of Proposition~\ref{prop::slepair_cmp_sym}, Existence]
Note that, given $\eta^L[0,T^L]$, the conditional law of the remaining part of $\eta^L$ is $\SLE_{\kappa}(\rho)$ from $w^L$ to $y^L$ with force point $w^L$; given $\eta^R[0,T^R]$, the conditional law of the remaining part of $\eta^R$ is $\SLE_{\kappa}(\rho)$ from $w^R$ to $y^R$ with force point $w^R$. Thus, to show the existence of the pair $(\eta^L;\eta^R)$ in Proposition~\ref{prop::slepair_cmp_sym}, it is sufficient to show the existence of the pair $(\eta^L|_{[0,T^L]}; \eta^R|_{[0,T^R]})$. 

Set
$\nu:=\kappa-8-\rho<\kappa/2-4$. Let $\eta^R$ be $\hSLE_{\kappa}(\nu)$ in $\Omega$ from $x^R$ to $x^L$ with marked points $(y^R, y^L)$ conditioned to hit $(y^Ry^L)$ (since $\nu<\kappa/2-4$, this event has positive chance). For $\eps>0$, let $T^R_{\eps}$ be the first time that $\eta^R$ hits the $\eps$-neighborhood of $(y^Ry^L)$. Given $\eta^R[0,T^R_{\eps}]$, let $\eta^L$ be $\hSLE_{\kappa}(\nu)$ in $\Omega\setminus \eta^R[0,T^R_{\eps}]$ from $x^L$ to $\eta^R(T^R_{\eps})$ with marked points $(y^L, y^R)$ conditioned to hit $(y^Ry^L)$. Let $T^L_{\eps}$ be the first time that $\eta^L$ hits the $\eps$-neighborhood of $(y^Ry^L)$. Here we obtain a pair of continuous simple curves $(\eta^L|_{[0,T^L_{\eps}]}; \eta^R|_{[0,T^R_{\eps}]})$. We could also sample the pair by first sampling $\eta^R$ and then sampling $\eta^L$ conditional on $\eta^R$. Corollary~\ref{cor::hsle_commutation} guarantees that the law on the pair $(\eta^L|_{[0,T^L_{\eps}]}; \eta^R|_{[0,T^R_{\eps}]})$ does not depend on the sampling order. In order to apply Corollary~\ref{cor::hsle_commutation}, it is important that $\kappa\le 4$ and the curves do not hit each other almost surely. In this case, we could exhaust all dyadic localization neighborhoods, combining with Lemma~\ref{lem::localsle_compatible}, we could conclude that the law on the pair $(\eta^L|_{[0,T^L_{\eps}]}; \eta^R|_{[0,T^R_{\eps}]})$ does not depend on the sampling order. This is true for all $\eps>0$. 

Let $\eps\to 0$, combining with the continuity of $\hSLE_{\kappa}(\nu)$ up to and including the first hitting time of $[y^Ry^L]$ in Propositions~\ref{prop::hyperSLE} and \ref{prop::hsle_continuity_lownu}, we could conclude that the law on the pair $(\eta^L|_{[0,T^L]}; \eta^R|_{[0,T^R]})$ does not depend on the sampling order. Consider the pair $(\eta^L|_{[0,T^L]}; \eta^R|_{[0,T^R]})$, we see that the conditional law of $\eta^L$ given $\eta^R[0,T^R]$ is $\SLE_{\kappa}(\nu+2)$ in $\Omega^R$ from $x^L$ to $w^R$ with force point $y^L$ up to the first hitting time of $(y^Ry^L)$. By Lemma~\ref{lem::sle_kapparho_targetchanging}, we know that the conditional law of $\eta^L$ given $\eta^R[0,T^R]$ is $\SLE_{\kappa}(\rho)$ in $\Omega^R$ from $x^L$ to $y^L$ with force point $w^R$ up to the first hitting time of $(y^Ry^L)$. Similarly, the conditional law of $\eta^R$ given $\eta^L[0,T^L]$ is $\SLE_{\kappa}(\rho)$ in $\Omega^L$ from $x^R$ to $y^R$ with force point $w^L$ up to the first hitting time of $(y^Ry^L)$. This implies the existence part of Proposition~\ref{prop::slepair_cmp_sym}.
\end{proof}

\begin{proof}[Proof of Proposition~\ref{prop::slepair_cmp_sym}, Uniqueness]
The uniqueness part could be proved similarly as the proof of \cite[Theorem 4.1]{MillerSheffieldIG2}. The setting there is slightly different from ours and the same proof works. We will briefly summarize the proof and point out the different places. 
We construct a Markov chain on configurations in $X_0(\Omega; x^R, y^R, y^L, x^L)$: one transitions from one configuration $(\eta^L;\eta^R)$ by picking $i\in\{L, R\}$ uniformly and then resampling $\eta^i$ according to the conditional law given the other one. 
The uniqueness of $\QQ_q(\kappa, \rho)$ will follow from the uniqueness of the stationary measure of this Markov chain. 
The $\eps$-Markov chain is defined similarly except in each step we resample  the paths conditioned on  them staying in $X^{\eps}_0(\Omega; x^R, y^R, y^L, x^L)$. Denote by $P_{\eps}$ the transition kernel for the $\eps$-Markov chain. It suffices to show that there is a unique stationary distribution for the $\eps$-Markov chain. Sending $\eps\to 0$ implies that the original chain has a unique stationary distribution. 

It is proved in \cite{MillerSheffieldIG2} that the transition kernel for $\eps$-Markov chain is continuous. In this part, the requirements are that the conditional law---$\SLE_{\kappa}(\rho)$--- can be sampled as flow lines of GFF, and that the two curves do not hit each other almost surely. The conditional law is $\SLE_{\kappa}(\rho)$ with force point $w^L\in (y^Ry^L)$ or $w^R\in (y^Ry^L)$, thus the two curves do not hit for all $\rho>-2$ as long as $\kappa\le 4$. So our setting satisfies the two requirements. Let $\mu$ be any stationary distribution of the Markov chain, and let $\mu_{\eps}$ be $\mu$ conditioned on $X_0^{\eps}(\Omega; x^R, y^R, y^L, x^L)$, then $\mu_{\eps}$ is stationary for the $\eps$-Markov chain. Let $\LS_{\eps}$ be the set of all such stationary probability measures. Then $\LS_{\eps}$ is convex and compact by the continuity of the transition kernel of the $\eps$-Markov chain. By Choquet's Theorem, the measure $\mu_{\eps}$ can be uniquely expressed as a superposition of extremal elements of $\LS_{\eps}$. To show that $\LS_{\eps}$ consists of a single element, it suffices to show that there is only one extremal in $\LS_{\eps}$. Suppose that $\nu, \tilde{\nu}$ are two extremal elements in $\LS_{\eps}$. By Lebesgue decomposition theorem, one can uniquely write $\nu=\nu_0+\nu_1$ such that $\nu_0$ is absolutely continuous and $\nu_1$ is singular with respect to $\tilde{\nu}$. If $\nu_0$ and $\nu_1$ are both nonzero, since $\nu=\nu_0 P_{\eps}+\nu_1 P_{\eps}$, by the uniqueness of the Lebesgue decomposition, we see that $\nu_0$ and $\nu_1$ are both stationary and thus can be normalized as stationary distributions for the $\eps$-Markov chain. This contradicts that $\nu$ is an extremal measure. This implies that either $\nu$ is absolutely continuous with respect to $\tilde{\nu}$ or singular. 

Next, it is proved in \cite{MillerSheffieldIG2} that it is impossible for $\nu$ to be absolutely continuous with respect to $\tilde{\nu}$. The same proof for this part also works here. The last part is showing that $\nu$ can not be singular with respect to $\tilde{\nu}$. Suppose $(\eta^L_0;\eta^R_0)\sim \nu$ and $(\tilde{\eta}^L;\tilde{\eta}^R)\sim\tilde{\nu}$ are the initial state for the $\eps$-Markov chain. Then they argued that it is possible to couple $(\eta^L_2;\eta^R_2)$ and $(\tilde{\eta}^L_2;\tilde{\eta}^R_2)$ such that the event $(\eta^L_2;\eta^R_2)=(\tilde{\eta}^L_2;\tilde{\eta}^R_2)$ has positive chance. This implies that $\nu$ and $\tilde{\nu}$ can not be singular. The key ingredient in this part is \cite[Lemma 4.2]{MillerSheffieldIG2} which needs to be replaced by Lemma~\ref{lem::positive_chance} in our setting. 
\end{proof}

\begin{lemma}\label{lem::positive_chance}
Fix $\kappa\in (0,8)$ and $\rho>(-2)\vee(\kappa/2-4)$. 
Suppose $(\Omega; x^R, y^R, y^L, x^L)$ is a quad, $w^R\in (y^Ly^R)$ is a boundary point, and $\tilde{\Omega}\subset\Omega$ is such that $\tilde{\Omega}$ agrees with $\Omega$ in a neighborhood of $(w^Rx^L)$. Let $\eta$ be an $\SLE_{\kappa}(\rho)$ in $\Omega$ from $x^L$ to $y^L$ with force point $y^R$ and let $\tilde{\eta}$ be an $\SLE_{\kappa}(\rho)$ in $\tilde{\Omega}$ from $x^L$ to $y^L$ with force point $w^R$. Then there exists a coupling between $\eta$ and $\tilde{\eta}$ such that the event $\{\eta=\tilde{\eta}\}$ has positive chance.
\end{lemma}
\begin{proof}
Although the setting is different from it is in the proof of \cite[Lemma 4.2]{MillerSheffieldIG2}, the same proof works here. We can view $\eta$ (resp. $\tilde{\eta}$) as the flow line of a GFF $h$ in $\Omega$ (resp. the flow line of a GFF $\tilde{h}$ in $\tilde{\Omega}$). The key point is that the boundary value of $h$ and $\tilde{h}$ agree in a neighborhood $U$ of $(w^Rx^L)$. Therefore $h|_U$ and $\tilde{h}|_{U}$ are mutually absolutely continuous. Since the flow lines are  deterministic functions of the $\GFF$, this implies that the laws of $\eta$ and $\tilde{\eta}$ stopped upon first exiting $U$ are mutually absolutely continuous. Since there is a positive chance for $\eta$ to stay in $U$, the absolute continuity implies the conclusion.  
\end{proof}

%% file: tex/slepair_cmp.tex
In Definition~\ref{def::CMP_slepair}, there is an ambiguity in the definition of domain Markov property: we need to specify what happens when 
$\eta^R[0,\tau^R]$ disconnects $y^R$ from $y^L$ (or $\eta^L[0,\tau^L]$ disconnects $y^L$ from $y^R$). In this case, we think the CMP in Definition~\ref{def::CMP_slepair} becomes the CMP for $\eta^L|_{t\ge\tau^L}$ in Definition~\ref{def::CMP_threepoints}. In fact, this is the key point in showing the characterization part in Theorem~\ref{thm::SLEpair_CMP_SYM}. 

\begin{proof}[Proof of Theorem~\ref{thm::SLEpair_CMP_SYM}]
Suppose $(\eta^L;\eta^R)\sim\PP_q$ with $q=(\Omega; x^R, y^R, y^L, x^L)$, and assume the same notations as in Figure~\ref{fig::def_forcepoint}. Since $(\eta^L;\eta^R)$ satisfies CMP, we know that the conditional law of $\eta^L$ given $\eta^R$ satisfies CMP in Definition~\ref{def::CMP_threepoints}. By Theorem~\ref{thm::CMP_threepoints}, we know that the conditional law of $\eta^L$ given $\eta^R$ is $\SLE_{\kappa}(\rho)$ with force point $w^R$ for some $\kappa$ and $\rho$. Since we require the curves to be simple, we know that $\kappa\in (0,4]$ and $\rho>-2$. Similarly, the conditional law of $\eta^R$ given $\eta^L$ is $\SLE_{\tilde{\kappa}}(\tilde{\rho})$ for some $\tilde{\kappa}\in (0,4]$ and $\tilde{\rho}>-2$. By Symmetry in Definition~\ref{def::sym}, we know that $\tilde{\kappa}=\kappa$ and $\tilde{\rho}=\rho$. 
This implies $\PP_q$ equals the probability measures $\QQ_q(\kappa, \rho)$ in Proposition~\ref{prop::slepair_cmp_sym}. 

Next, we show that the pair $(\eta^L;\eta^R)\sim \QQ_q(\kappa, \rho)$ satisfies CMP. For every $\eta^L$-stopping time $\tau^L$ and every $\eta^R$-stopping time $\tau^R$, consider the conditional law of $(\eta^L|_{t\ge \tau^L}; \eta^R|_{t\ge\tau^R})$ given $\eta^L[0,\tau^L]$ and $\eta^R[0,\tau^R]$, it is clear that the conditional law of $\eta^R|_{t\ge\tau^R}$ given $\eta^L|_{t\ge \tau^L}$ is $\SLE_{\kappa}(\rho)$ and the conditional law of $\eta^L|_{t\ge \tau^L}$ given $\eta^R|_{t\ge\tau^R}$ is $\SLE_{\kappa}(\rho)$. Therefore, the pair $(\eta^L;\eta^R)$ satisfies CMP. 

Finally, we show that the pair $(\eta^L;\eta^R)\sim\QQ_q(\kappa,\rho)$ satisfies Condition C1. We only need to show that $\eta^L$ satisfies Condition C1. Suppose $(Q; a, b, c, d)$ is an avoidable quad for $\eta^L$. By the comparison principle of extremal distance (see \cite[Section 4-3]{AhlforsConformalInvariants}), we know that 
\[d_{Q\setminus\eta^R}((ab), (cd))\ge d_Q((ab), (cd)).\]
Note that the conditional law of $\eta^L$ given $\eta^R$ is $\SLE_{\kappa}(\rho)$ and $\SLE_{\kappa}(\rho)$ satisfies Condition C1 by Lemma~\ref{lem::sle_c1}, combining with the above inequality, we see that $\eta^L$ satisfies Condition C1. 
\end{proof}

Next, we will show Corollary~\ref{cor::SLEpair_CMP_REV}. To this end, we first discuss the reversibility of $\SLE_{\kappa}(\rho)$ processes. 
Suppose $x\le w\le y$, and let $\eta$ be an $\SLE_{\kappa}(\rho)$ in $\HH$ from $x$ to $y$ with force point $w$. The process $\eta$ does not have reversibility when $x<w<y$, see Lemma~\ref{lem::sle_rev}; but it enjoys reversibility when $w=x_+$, see \cite[Theorem 1.1]{MillerSheffieldIG2} and \cite[Theorem 1.2]{MillerSheffieldIG3}.
The reversibility when $w=x_+$ is a deep result, and we do not need it in our paper.   

\begin{lemma}\label{lem::sle_rev}
Fix $\kappa\in (0,8)$, $\rho>-2$ and $x<w, \tilde{w}<y$. Suppose $\eta$ is an $\SLE_{\kappa}(\rho)$ in $\HH$ from $x$ to $y$ with force point $w$. Then the time-reversal of $\eta$ is an $\SLE_{\tilde{\kappa}}(\tilde{\rho})$ from $y$ to $x$ with force point $\tilde{w}$ if and only if $\tilde{\kappa}=\kappa$ and $\rho=\tilde{\rho}=0$. 
\end{lemma}
\begin{proof}
Let $\hat{\eta}$ be the time-reversal of $\eta$. If $\hat{\eta}$ has the law of $\SLE_{\tilde{\kappa}}(\tilde{\rho})$, since the dimension of $\SLE_{\kappa}(\rho)$ process is $1+\kappa/8$ \cite{BeffaraDimension}, we have $\tilde{\kappa}=\kappa$. It remains to show $\tilde{\rho}=\rho=0$. Let $\tilde{\eta}$ be an $\SLE_{\kappa}(\tilde{\rho})$ in $\HH$ from $y$ to $x$ with force point $\tilde{w}$. 

When $\kappa\in (4,8)$ and $\rho\ge\kappa/2-2$, we have
$\eta\cap(w,y)=\emptyset$ and $\tilde{\eta}\cap(w,y)\neq\emptyset$ almost surely. Thus $\hat{\eta}$ can not have the same law as $\tilde{\eta}$. 
When $\kappa\in (4,8)$ and $\rho\in (\kappa/2-4, \kappa/2-2)$, we have almost surely (see \cite[Theorem 1.6]{MillerWuSLEIntersection})
\[\dim\eta\cap(w\vee\tilde{w},y)=1-(\rho+2)(\rho+4-\kappa/2)/\kappa,\quad \dim \tilde{\eta}\cap(w\vee\tilde{w},y)=1-(8-\kappa)/\kappa.\]
If $\hat{\eta}$ has the same law as $\tilde{\eta}$, then these two dimensions have to coincide, and hence $\rho=0$ and therefore $\tilde{\rho}=0$. When $\kappa\in (4,8)$ and $\rho\in (-2, \kappa/2-4]$, we know that $\eta$ fills interval $(w\vee\tilde{w},y)$ whereas $\tilde{\eta}\cap (w\vee\tilde{w},y)$ has no interior point. Thus the $\hat{\eta}$ can not have the same law as $\tilde{\eta}$.

When $\kappa\in (0,4]$ and $\rho<\kappa/2-2$, we have $\eta\cap (w\vee\tilde{w}, y)\neq\emptyset$ and $\tilde{\eta}\cap (w\vee\tilde{w}, y)=\emptyset$ almost surely. Thus $\hat{\eta}$ can not have the same law as $\tilde{\eta}$. When $\kappa\in (0,4]$ and $\rho\ge\kappa/2-2$, it is proved in Theorem~\ref{thm::hyperSLE_reversibility} that $\hat{\eta}$ is $\hSLE_{\kappa}(\rho-2)$ in $\HH$ from $y$ to $x$ with marked points $(y_{-}, w)$ which equals $\SLE_{\kappa}(\tilde{\rho})$ process from $y$ to $x$ with force point $\tilde{w}$ if and only if $\rho=\tilde{\rho}=0$.  
\end{proof}

Now, we are ready to show Corollary~\ref{cor::SLEpair_CMP_REV}

\begin{proof}[Proof of Corollary~\ref{cor::SLEpair_CMP_REV}]
Suppose $(\eta^L;\eta^R)\sim \PP_q$ for $q=(\Omega; x^R, y^R, y^L, x^L)$. 
By the proof of Theorem~\ref{thm::SLEpair_CMP_SYM}, there exists $\kappa\in (0,4]$ and $\rho>-2$ such that the conditional law of $\eta^L$ given $\eta^R$ is $\SLE_{\kappa}(\rho)$ in $\Omega^R$ from $x^L$ to $y^L$ with force point $w^R$. By the reversibility in Definition~\ref{def::rev}, the conditional law of $\eta^L$ given $\eta^R$ should be reversible. Combining with Lemma~\ref{lem::sle_rev}, we see that $\rho=0$. Therefore, the conditional law of $\eta^L$ given $\eta^R$ is $\SLE_{\kappa}$. Similarly, the conditional law of $\eta^R$ given $\eta^L$ is $\SLE_{\kappa}$. By Proposition~\ref{prop::slepair_rev}, there exists a unique such probability measure and the marginal of $\eta^R$ is $\hSLE_{\kappa}$.
\end{proof}

To end this section, we will discuss the relation between Propositions~\ref{prop::slepair_rev} and \ref{prop::slepair_cmp_sym}. 
First, when $\rho^L\neq 0$ or $\rho^R\neq 0$, the pair $(\eta^L;\eta^R)$ in Proposition~\ref{prop::slepair_rev} satisfies neither CMP in Definition~\ref{def::CMP_slepair} nor the symmetry in Definition~\ref{def::sym}, whereas it satisfies the reversibility in Definition~\ref{def::rev}.  

Next, we compare Proposition~\ref{prop::slepair_cmp_sym} with $\rho=0$ and Proposition~\ref{prop::slepair_rev} with $\rho^L=\rho^R=0$. In this case, the two propositions describe the same law on the pair $(\eta^L;\eta^R)$. 
From Proposition~\ref{prop::slepair_cmp_sym}, we see that the marginal law of $\eta^R$ is $\hSLE_{\kappa}(\kappa-8)$ from $x^R$ to $x^L$; whereas, from Proposition~\ref{prop::slepair_rev}, the marginal law of $\eta^R$ is $\hSLE_{\kappa}$ from $x^R$ to $y^R$. This implies that $\hSLE_{\kappa}$ from $x^R$ to $y^R$ has the same law as $\hSLE_{\kappa}(\kappa-8)$ from $x^R$ to $x^L$. This is consistent with the target-independence proved in Proposition~\ref{prop::hsle_change_target}.

%% file: tex/ising.tex
\subsection{Ising Model}

\textbf{Notation and terminology.}
We focus on the square lattice $\Z^2$. Two vertices $x=(x_1, x_2)$ and $y=(y_1, y_2)$ are neighbors if $|x_1-y_1|+|x_2-y_2|=1$, and we write $x\sim y$. 
The \textit{dual square lattice} $(\Z^2)^*$ is the dual graph of $\Z^2$. The vertex set is $(1/2, 1/2)+\Z^2$ and the edges are given by nearest neighbors.  The vertices and edges of $(\Z^2)^*$ are called dual-vertices and dual-edges. In particular, for each edge $e$ of $\Z^2$, it is associated to a dual edge, denoted by $e^*$, that it crosses $e$ in the middle. 
For a finite subgraph $G$, we define $G^*$ to be the subgraph of $(\Z^2)^*$ with edge-set $E(G^*)=\{e^*: e\in E(G)\}$ and vertex set given by the end-points of these dual-edges. The \textit{medial lattice} $(\Z^2)^{\diamond}$ is the graph with the centers of edges of $\Z^2$ as vertex set, and edges connecting nearest vertices. This lattice is a rotated and rescaled version of $\Z^2$. The vertices and edges of $(\Z^2)^{\diamond}$ are called medial-vertices and medial-edges. We identify the faces of $(\Z^2)^{\diamond}$ with the vertices of $\Z^2$ and $(\Z^2)^*$. A face of $(\Z^2)^{\diamond}$ is said to be black if it corresponds to a vertex of $\Z^2$ and white if it corresponds to a vertex of $(\Z^2)^*$. 


Let $\Omega$ be a finite subset of $\Z^2$. The Ising model with free boundary conditions is a random assignment $\sigma\in \{\ominus, \oplus\}^{\Omega}$ of spins $\sigma_x\in \{\ominus, \oplus\}$, where $\sigma_x$ denotes the spin at the vertex $x$. The Hamiltonian of the Ising model is defined by 
\[H^{\free}_{\Omega}(\sigma)=-\sum_{x\sim y}\sigma_x\sigma_y.\] 
The Ising measure is the Boltzmann measure with Hamiltonian $H^{\free}_{\Omega}$ and inverse-temperature $\beta>0$: 
\[\mu^{\free}_{\beta,\Omega}[\sigma]=\frac{\exp(-\beta H^{\free}_{\Omega}(\sigma))}{Z^{\free}_{\beta, \Omega}},\quad\text{where }Z^{\free}_{\beta, \Omega}=\sum_{\sigma}\exp(-\beta H^{\free}_{\Omega}(\sigma)).\]

For a graph $\Omega$ and $\tau\in \{\ominus, \oplus\}^{\Z^2}$, one may also define the Ising model with boundary conditions $\tau$ by the Hamiltonian
\[H^{\tau}_{\Omega}(\sigma)=-\sum_{x\sim y, \{x,y\}\cap \Omega\neq\emptyset}\sigma_x\sigma_y,\quad \text{if }\sigma_x=\tau_x, \forall x\not\in \Omega.\]
Suppose that $(\Omega; a, b)$ is a Dobrushin domain. The \textit{Dobrushin boundary conditions} is the following: $\oplus$ along $(ab)$, and $\ominus$ along $(ba)$. 

The set $\{\ominus, \oplus\}^{\Omega}$ is equipped with a partial order: $\sigma\le \sigma'$ if $\sigma_x\le \sigma_x'$ for all $x\in \Omega$. A random variable $X$ is increasing if $\sigma\le \sigma'$ implies $X(\sigma)\le X(\sigma')$. An event $\LA$ is increasing if $\one_{\LA}$ is increasing. The Ising model satisfies FKG inequality:
Let $\Omega$ be a finite subset and $\tau$ be boundary conditions, and $\beta>0$. For any two increasing events $\LA$ and $\LB$, we have 
$\mu^{\tau}_{\beta, \Omega}[\LA\cap\LB]\ge \mu^{\tau}_{\beta, \Omega}[\LA]\mu^{\tau}_{\beta, \Omega}[\LB]$.
As a consequence of FKG inequality, we have the comparison between boundary conditions: For boundary conditions $\tau_1\le \tau_2$ and an increasing event $\LA$, we have
\begin{equation}\label{eqn::ising_boundary_comparison}
\mu^{\tau_1}_{\beta, \Omega}[\LA]\le \mu^{\tau_2}_{\beta, \Omega}[\LA]. 
\end{equation}

The critical Ising model ($\beta=\beta_c$) is conformal invariant in the scaling limit, see \cite{DCParafermionic} for general background. We only collect several properties of the critical Ising model that will be useful later: strong RSW and the convergence of the interface.

Given a quad $(Q; a, b, c, d)$ on the square lattice, we denote by $d_{Q}((ab), (cd))$ the discrete extermal distance between $(ab)$ and $(cd)$ in $Q$, see \cite[Section 6]{ChelkakRobustComplexAnalysis}. The discrete extremal distance is uniformly comparable to and converges to its continuous counterpart--- the classical extremal distance. The quad $(Q; a, b, c, d)$ is crossed by $\oplus$ in an Ising configuration $\sigma$ if there exists a path of $\oplus$ going from $(ab)$ to $(cd)$ in $Q$. We denote this event by $(ab)\overset{\oplus}{\longleftrightarrow}(cd)$. 
\begin{proposition}
\label{prop::ising_rsw}\cite[Corollary 1.7]{ChelkakDuminilHonglerCrossingprobaFKIsing}.
For each $L>0$ there exists $c(L)>0$ such that the following holds: for any quad $(Q; a, b, c, d)$ with $d_{Q}((ab), (cd))\ge L$, 
\[\mu^{\text{mixed}}_{\beta_c, Q}\left[(ab)\overset{\oplus}{\longleftrightarrow}(cd)\right]\le 1-c(L),\]
where the boundary conditions are $\free$ on $(ab)\cup(cd)$ and $\ominus$ on $(bc)\cup(da)$. 
\end{proposition}

\begin{figure}[ht!]
\begin{center}
\includegraphics[width=0.7\textwidth]{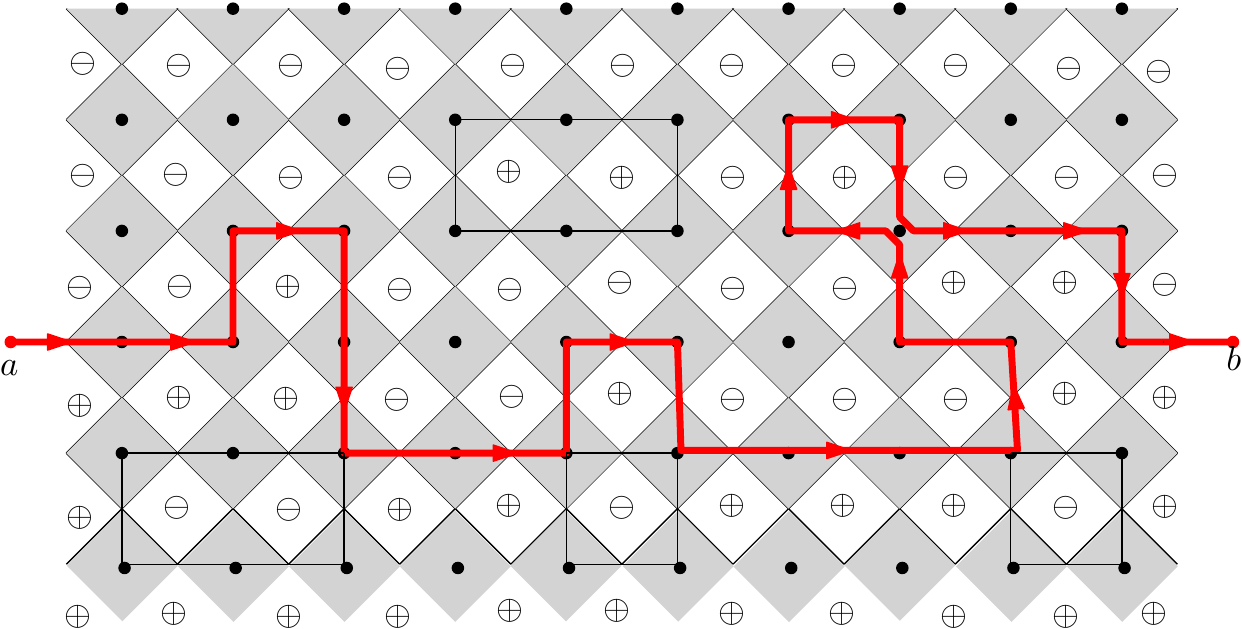}
\end{center}
\caption{\label{fig::ising_interface} The Ising interface with Dobrushin boundary conditions. }
\end{figure}

For $\delta>0$, we consider the rescaled square lattice $\delta\Z^2$. The definitions of dual lattice, medial lattice and  Dobrushin domains extend to this context, and they will be denoted by $(\Omega_{\delta}; a_{\delta}, b_{\delta})$, $(\Omega^*_{\delta}; a^*_{\delta}, b^*_{\delta})$, $(\Omega^{\diamond}_{\delta}; a^{\diamond}_{\delta}, b^{\diamond}_{\delta})$ respectively. 
Consider the critical Ising model on $(\Omega^*_{\delta}; a^*_{\delta}, b^*_{\delta})$. The boundary $\partial \Omega^*_{\delta}$ is divided into two parts $(a^*_{\delta}b^*_{\delta})$ and $(b^*_{\delta}a^*_{\delta})$. We fix the Dobrushin boundary conditions: $\ominus$ on $(b^*_{\delta}a^*_{\delta})$ and $\oplus$ on $(a^*_{\delta}b^*_{\delta})$. Define the \textit{interface} as follows. It starts from $a_{\delta}^{\diamond}$, lies on the primal lattice and turns at every vertex of $\Omega_{\delta}$ is such a way that it has always dual vertices with spin $\ominus$ on its left and $\oplus$ on its right. If there is an indetermination when arriving at a vertex (this may happen on the square lattice), turn left. See Figure~\ref{fig::ising_interface}. We have the convergence of the interface:  

\begin{theorem}\label{thm::ising_cvg_minusplus}\cite{CDCHKSConvergenceIsingSLE}.
Let $(\Omega^{\diamond}_{\delta}; a^{\diamond}_{\delta}, b^{\diamond}_{\delta})$ be a family of Dobrushin domains converging to a Dobrushin domain $(\Omega; a, b)$ in the Carath\'eodory sense. The interface of the critical Ising model in $(\Omega^*_{\delta}; a^*_{\delta}, b^*_{\delta})$ with Dobrushin boundary conditions converges weakly to $\SLE_{3}$ as $\delta\to 0$.
\end{theorem}

\begin{theorem}\label{thm::ising_cvg_minusfree}
Let $(\Omega^{\diamond}_{\delta}; a^{\diamond}_{\delta}, w_{\delta}^{\diamond},  b^{\diamond}_{\delta})$ be a family of triangles converging to a triangle $(\Omega; a, w, b)$ in the Carath\'eodory sense. The interface of the critical Ising model in $(\Omega^*_{\delta}; a^*_{\delta}, w_{\delta}^*, b^*_{\delta})$ with the boundary conditions $\ominus$ along $(b^*_{\delta}a^*_{\delta})$, $\oplus$ along $(a_{\delta}^*w_{\delta}^*)$ and $\free$ along $(w^*_{\delta}b^*_{\delta})$ converges weakly to $\SLE_{3}(-3/2)$ as $\delta\to 0$.
\end{theorem}
\begin{proof}
It is proved in \cite{HonglerKytolaIsingFree, BenoistDuminilHonglerIsingFree} that the interface with $(\free\free)$ boundary conditions converges weakly to $\SLE_3(-3/2;-3/2)$ as $\delta\to 0$. The same proof works here. 
\end{proof}

%% file: tex/ising_proof.tex
Let $(\Omega_{\delta}; x^R_{\delta}, y^R_{\delta}, y^L_{\delta}, x^L_{\delta})$ be a sequence of discrete quads on the square lattice $\delta\Z^2$ approximating some quad $q=(\Omega; x^R, y^R, y^L, x^L)$. Consider the critical Ising model in $\Omega_{\delta}^*$ with alternating boundary conditions: $\ominus$ along $(x_{\delta}^Lx^R_{\delta})$ and $(y_{\delta}^Ry_{\delta}^L)$, and $\xi^R\in\{\oplus, \free\}$ along $(x_{\delta}^Ry_{\delta}^R)$, and $\xi^L\in\{\oplus,\free\}$ along $(y_{\delta}^Lx_{\delta}^L)$. 
The quad is vertically crossed by $\ominus$ if there exists a path of $\ominus$ going from $(x^L_{\delta}x^R_{\delta})$ to $(y^R_{\delta}y^L_{\delta})$. 
The quad is horizontally crossed by $\oplus$ in an Ising configuration if there exists a path of $\oplus$ going from $(y^L_{\delta}x^L_{\delta})$ to $(x^R_{\delta}y^R_{\delta})$.
We denote these events by
\[ \LC_v^{\ominus}(q)=\{(x^L_{\delta}x^R_{\delta})\overset{\ominus}{\longleftrightarrow}(y^R_{\delta}y^L_{\delta})\},\quad \LC^{\oplus}_h(q)=\{(y^L_{\delta}x^L_{\delta})\overset{\oplus}{\longleftrightarrow}(x^R_{\delta}y^R_{\delta}\}. \]

Suppose there is a vertical crossing of $\ominus$.
Let $\eta^L_{\delta}$ be the interface starting from $x^L_{\delta}$ lying on the primal lattice. It turns at every vertex in the way that it has spin $\oplus$ on its left and $\ominus$ on its right, and that it turns left when there is ambiguity. 
Let $\eta^R_{\delta}$ be the interface starting from $x^R_{\delta}$ lying on the primal lattice. It turns at every vertex in the way that it has spin $\ominus$ to its left and $\oplus$ to its right, and turns right when there is ambiguity. Then $\eta^L_{\delta}$ will end at $y^L_{\delta}$ and $\eta^R_{\delta}$ will end at $y^R_{\delta}$. See Figure~\ref{fig::ising_pair}. Let $\Omega^L_{\delta}$ be the connected component of $\Omega_{\delta}\setminus\eta^L_{\delta}$ with $(x_{\delta}^Ry_{\delta}^R)$ on the boundary and denote by $\LD^L_{\delta}$ the discrete extremal distance between $\eta^L_{\delta}$ and $(x^R_{\delta}y^R_{\delta})$ in $\Omega^L_{\delta}$. Define $\Omega^R_{\delta}$ and $\LD^R_{\delta}$ similarly. 
\begin{lemma}\label{lem::ising_distance_tight}
The family of variables $\{(\LD^L_{\delta}; \LD^R_{\delta})\}_{\delta>0}$ is tight in the following sense: for any $u>0$, there exists $\eps>0$ such that
\[\PP\left[\LD^L_{\delta}\ge\eps, \LD^R_{\delta}\ge \eps\cond \LC_v^{\ominus}(q) \right]\ge 1-u,\quad \forall \delta>0.\]
\end{lemma}
\begin{proof}
Since $(\Omega_{\delta}; x^L_{\delta}, x^R_{\delta}, y^R_{\delta}, y^L_{\delta})$ approximates $(\Omega; x^L, x^R, y^R, y^L)$, 
by Proposition~\ref{prop::ising_rsw} and \eqref{eqn::ising_boundary_comparison}, we know that $\PP[\LC_v^{\ominus}(q)]$ can be bounded from below by some quantity that depends only on the extremal distance in $\Omega$ between $(x^Lx^R)$ and $(y^Ry^L)$ and that is uniform over $\delta$. 
Thus, it is sufficient to show $\PP\left[\{\LD^L_{\delta}\le\eps\}\cap\LC_v^{\ominus}(q)\right]$ is small for $\eps>0$ small. Given $\eta^L_{\delta}$ and on the event $\{\LD^L_{\delta}\le\eps\}$, combining Proposition~\ref{prop::ising_rsw} and \eqref{eqn::ising_boundary_comparison}, we know that the probability to have a vertical crossing of $\ominus$ in $\Omega^L_{\delta}$ is bounded by $c(\eps)$ which only depends on $\eps$ and goes to zero as $\eps\to 0$. Thus $\PP\left[\{\LD^L_{\delta}\le\eps\}\cap\LC_v^{\ominus}(q)\right]\le c(\eps)$. This implies the conclusion. 
\end{proof}

\begin{lemma}\label{lem::ising_cvg_slepair}
Let $(\Omega_{\delta}; x^R_{\delta}, y^R_{\delta}, y^L_{\delta}, x^L_{\delta})$ be a sequence of discrete quads on the square lattice $\delta\Z^2$ approximating some quad $q=(\Omega; x^R, y^R, y^L, x^L)$ as $\delta\to 0$. Consider the critical Ising model in $\Omega_{\delta}$ with the alternating boundary condition: 
\[\ominus \text{ along } (x^L_{\delta}x^R_{\delta})\cup (y^R_{\delta}y^L_{\delta}),\quad \xi^R\in\{\oplus, \free\}\text{ along }(x^R_{\delta}y^R_{\delta}), \quad \text{ and }\xi^L\in\{\oplus, \free\}\text{ along }(y^L_{\delta}x^L_{\delta}).\] 
Conditioned on the event $\LC_v^{\ominus}(q)$, then there exists a pair of interfaces $(\eta^L_{\delta};\eta^R_{\delta})$ where $\eta^L_{\delta}$ (resp. $\eta^R_{\delta}$) is the interface connecting $x^L_{\delta}$ to $y^L_{\delta}$ (resp. connecting $x^R_{\delta}$ to $y^R_{\delta}$).  The law of the pair $(\eta^L_{\delta}; \eta^R_{\delta})$ converges weakly to the pair of $\SLE$ curves in Proposition~\ref{prop::slepair_rev} as $\delta\to 0$ where $\kappa=3$ and $\xi^R, \xi^L,\rho^R,\rho^L$ are related in the following way: for $q\in\{L, R\}$,
\[\rho^q=0,\quad\text{if }\xi^q=\oplus;\quad \rho^q=-3/2,\quad\text{if }\xi^q=\free.\]
\end{lemma}
\begin{proof}
We only prove the conclusion for $\xi^R=\xi^L=\oplus$, and the other cases can be proved similarly (by replacing Theorem~\ref{thm::ising_cvg_minusplus} with \ref{thm::ising_cvg_minusfree} when necessary). 
Combining Proposition~\ref{prop::ising_rsw} with (\ref{eqn::ising_boundary_comparison}) and Lemma~\ref{lem::ising_distance_tight}, we see that the sequence $\{(\eta^L_{\delta};\eta^R_{\delta})\}_{\delta>0}$ satisfies the requirements in Theorem~\ref{thm::cvg_pairs}, thus the sequence is relatively compact. Suppose $(\eta^L;\eta^R)\in X_0(\Omega; x^L, x^R, y^R, y^L)$ is any sub-sequential limit and, for some $\delta_k\to 0$, 
\[(\eta^L_{\delta_k};\eta^R_{\delta_k})\overset{d}{\longrightarrow}(\eta^L;\eta^R)\quad \text{ in } X_0(\Omega; x^L, x^R, y^R, y^L).\]
Since $\eta^L_{\delta_k}\to \eta^L$, by Theorem~\ref{thm::cvg_curves_chordal}, we know that we have the convergence in all three topologies. In particular, this implies the convergence of $\Omega^L_{\delta_k}$ in Carath\'eodory sense. Note that the conditional law of $\eta^R_{\delta_k}$ in $\Omega^L_{\delta_k}$ given $\eta^L_{\delta_k}$ is the interface of the critical planar Ising model with Dobrushin boundary condition. Combining with Theorem~\ref{thm::ising_cvg_minusplus}, we know that, the conditional law of $\eta^R$ in $\Omega^L$ given $\eta^L$ is $\SLE_3$. By symmetry, the conditional law of $\eta^L$ in $\Omega^R$ given $\eta^R$ is $\SLE_3$. By Proposition~\ref{prop::slepair_rev}, there exists a unique such measure. Thus it has to be the unique sub-sequential limit. This proves the convergence of the whole sequence. 
\end{proof}

\begin{figure}[ht!]
\begin{subfigure}[b]{0.48\textwidth}
\begin{center}
\includegraphics[width=0.75\textwidth]{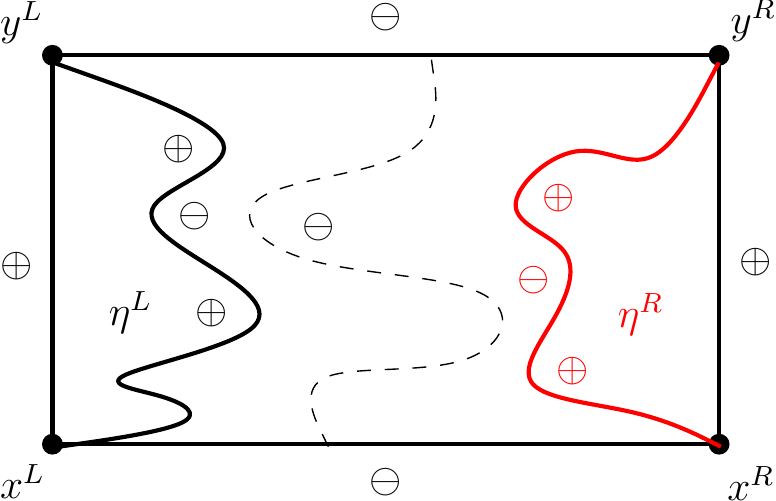}
\caption{When there is a vertical crossing of $\ominus$, there exists a pair of interfaces $(\eta^L;\eta^R)$: $\eta^R$ is the interface from $x^R$ to $y^R$ and $\eta^L$ is the interface from $x^L$ to $y^L$.}
\end{center}
\end{subfigure}
$\quad$
\begin{subfigure}[b]{0.48\textwidth}
\begin{center}\includegraphics[width=0.75\textwidth]{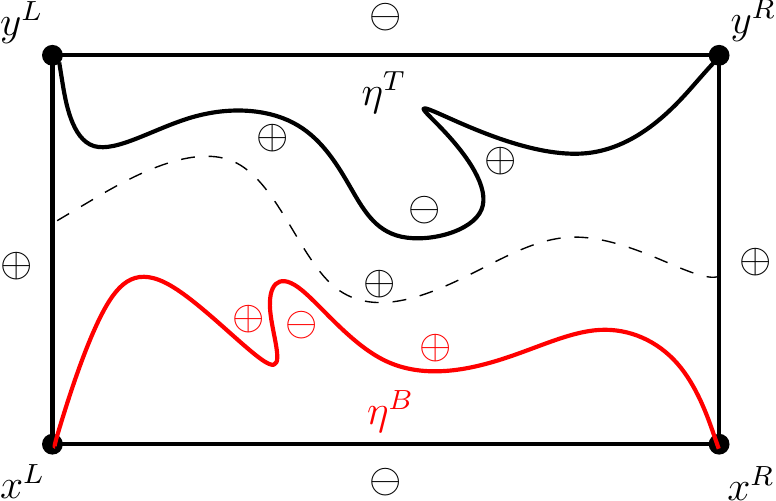}
\caption{When there is a horizontal crossing of $\oplus$, there exists a pair of interfaces $(\eta^B;\eta^T)$: $\eta^B$ is the interface from $x^R$ to $x^L$ and $\eta^T$ is the interface from $y^R$ to $y^L$.}
\end{center}
\end{subfigure}
\caption{\label{fig::ising_oplusominus} Consider the critical Ising in $\Omega$ with the boundary condition: $\ominus$ along $(x^Lx^R)\cup(y^Ry^L)$ and $\oplus$ along $(x^Ry^R)\cup (y^Lx^L)$. }
\end{figure}

\begin{lemma}\label{lem::ising_oplusominus}
Let $(\Omega_{\delta}; x^R_{\delta}, y^R_{\delta}, y^L_{\delta}, x^L_{\delta})$ be a sequence of discrete quads on the square lattice $\delta\Z^2$ approximating some quad $q=(\Omega; x^R, y^R, y^L, x^L)$ as $\delta\to 0$. Consider the critical Ising in $\Omega_{\delta}$ with the following boundary condition: 
\[\ominus\text{ along }(x^L_{\delta}x^R_{\delta})\cup(y^R_{\delta}y^L_{\delta}), \quad \oplus\text{ along }(x^R_{\delta}y^R_{\delta})\cup (y^L_{\delta}x^L_{\delta}).\]
\begin{itemize}
\item On the event $\LC_v^{\ominus}(q)$, let $\eta_{\delta}$ be the interface connecting $x^R_{\delta}$ and $y^R_{\delta}$. Then the law of $\eta_{\delta}$ converges weakly to $\hSLE_3$ in $\Omega$ from $x^R$ to $y^R$ with marked points $(x^L, y^L)$.
\item On the event $\LC_h^{\oplus}(q)$, let $\eta_{\delta}$ be the interface connecting $x^R_{\delta}$ and $x^L_{\delta}$. Then the law of $\eta_{\delta}$ converges weakly to $\hSLE_3$ in $\Omega$ from $x^R$ to $x^L$ with marked points $(y^R, y^L)$.
\end{itemize}
\end{lemma}
\begin{proof}
On the event $\LC_v^{\ominus}(q)$, there is a pair of Ising interfaces $(\eta^L_{\delta};\eta^R_{\delta})$, as indicated in Figure~\ref{fig::ising_oplusominus}~(a). By Lemma~\ref{lem::ising_cvg_slepair}, we know that the sequence $(\eta^L_{\delta};\eta^R_{\delta})$ converges weakly to the pair of $\SLE$s in Proposition~\ref{prop::slepair_rev} with $\kappa=3$ and $\rho^L=\rho^R=0$. In particular, the law of $\eta^R_{\delta}$ conditioned on $\LC_v^{\ominus}(q)$ converges weakly to $\hSLE_3$ in $\Omega$ from $x^R$ to $y^R$. The other case can be proved similarly.
\end{proof}

\begin{figure}[ht!]
\begin{subfigure}[b]{0.48\textwidth}
\begin{center}
\includegraphics[width=0.75\textwidth]{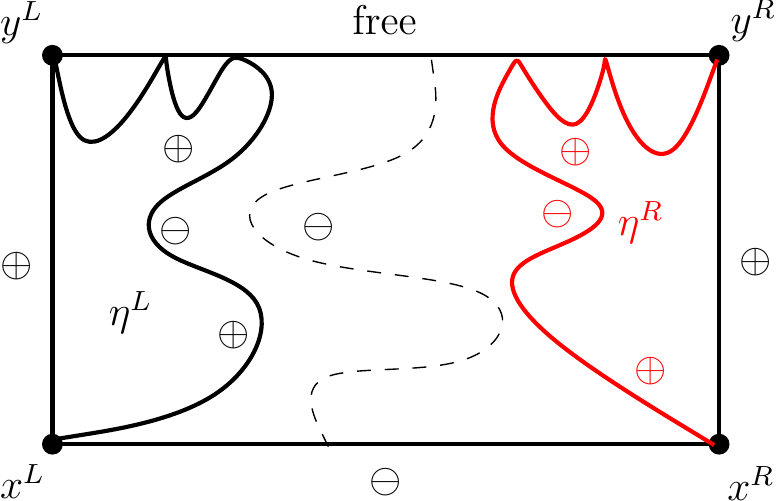}
\caption{When there is a vertical crossing of $\ominus$, there exists a pair of interfaces $(\eta^L;\eta^R)$: $\eta^R$ is the interface from $x^R$ to $y^R$ and $\eta^L$ is the interface from $x^L$ to $y^L$.}
\end{center}
\end{subfigure}
$\quad$
\begin{subfigure}[b]{0.48\textwidth}
\begin{center}\includegraphics[width=0.75\textwidth]{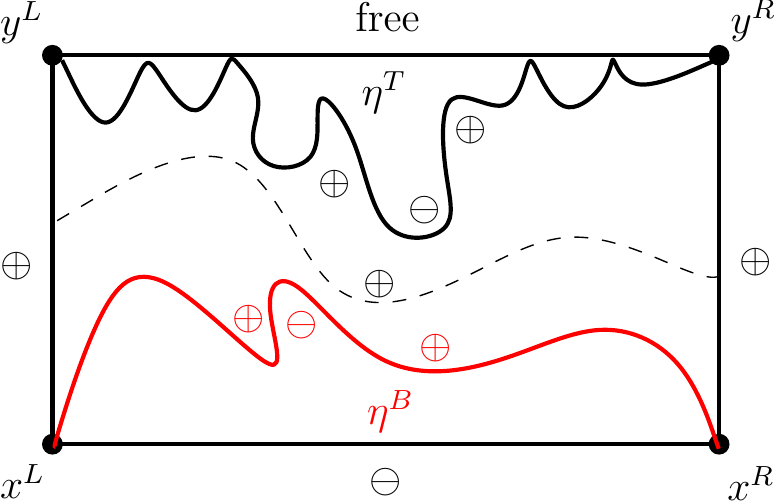}
\caption{When there is a horizontal crossing of $\oplus$, there exists a pair of interfaces $(\eta^B;\eta^T)$: $\eta^B$ is the interface from $x^R$ to $x^L$ and $\eta^T$ is the interface from $y^R$ to $y^L$.}
\end{center}
\end{subfigure}
\caption{\label{fig::ising_oplusominusfree} Consider the critical Ising in $\Omega$ with the boundary condition: $\ominus$ along $(x^Lx^R)$, $\oplus$ along $(x^Ry^R)\cup (y^Lx^L)$, and $\free$ along $(y^Ry^L)$. }
\end{figure}

\begin{lemma}\label{lem::ising_oplusominusfree}
Let $(\Omega_{\delta}; x^R_{\delta}, y^R_{\delta}, y^L_{\delta}, x^L_{\delta})$ be a sequence of discrete quads on the square lattice $\delta\Z^2$ approximating some quad $q=(\Omega; x^R, y^R, y^L, x^L)$ as $\delta\to 0$. Consider the critical Ising in $\Omega_{\delta}$ with the following boundary condition: 
\[\ominus\text{ along }(x^L_{\delta}x^R_{\delta}), \quad \oplus\text{ along }(x^R_{\delta}y^R_{\delta})\cup (y^L_{\delta}x^L_{\delta}),\quad \free\text{ along }(y^R_{\delta}y^L_{\delta}).\]
\begin{itemize}
\item On the event $\LC_v^{\ominus}(q)$, let $\eta_{\delta}$ be the interface connecting $x^R_{\delta}$ and $y^R_{\delta}$. Then the law of $\eta_{\delta}$ (up to the first hitting time of $[y^R_{\delta}y^L_{\delta}]$) converges weakly to $\hSLE_3(-7/2)$ from $x^R$ to $x^L$ conditioned to hit $[y^Ry^L]$ (up to the first hitting time of $[y^Ry^L]$).
\item On the event $\LC_h^{\oplus}(q)$, let $\eta_{\delta}$ be the interface connecting $x^R_{\delta}$ and $x^L_{\delta}$. Then the law of $\eta_{\delta}$ converges weakly to $\hSLE_3(-3/2)$ from $x^R$ to $x^L$.
\end{itemize}
\end{lemma}
\begin{proof}
On the event $\LC_v^{\ominus}(q)$, there is a pair of Ising interfaces $(\eta^L_{\delta};\eta^R_{\delta})$, as indicated in Figure~\ref{fig::ising_oplusominusfree}~(a). By a similar argument as in Lemma~\ref{lem::ising_cvg_slepair}, the sequence $(\eta^L_{\delta};\eta^R_{\delta})$ converges weakly to the pair of $\SLE$s in Proposition~\ref{prop::slepair_cmp_sym} with $\kappa=3$ and $\rho=-3/2$. In particular, the law of $\eta^R_{\delta}$ conditioned on $\LC_v^{\ominus}(q)$ converges weakly to $\hSLE_3(-7/2)$ in $\Omega$ from $x^R$ to $x^L$ conditioned to hit $[y^Ry^L]$ (here is $\hSLE_3(-7/2)$ from $x^R$ to $x^L$, this is not a typo). 

On the event $\LC^{\oplus}_h(q)$, there is a pair of Ising interfaces $(\eta^B_{\delta};\eta^T_{\delta})$ as indicted in Figure~\ref{fig::ising_oplusominusfree}~(b). By Lemma~\ref{lem::ising_cvg_slepair}, we see that the sequence $(\eta^B_{\delta};\eta^T_{\delta})$ converges weakly to the pair of $\SLE$s in Proposition~\ref{prop::slepair_rev} (rotated by 90 degree counterclockwise) with $\kappa=3$ and $\rho=-3/2$. In particular, the law of $\eta^B_{\delta}$ conditioned on $\LC^{\oplus}_h(q)$ converges weakly to $\hSLE_3(-3/2)$ in $\Omega$ from $x^R$ to $x^L$. 
\end{proof}

\begin{proof}[Proof of Proposition~\ref{prop::ising_hypersle}]
Proposition~\ref{prop::ising_hypersle} is a collection of Lemma~\ref{lem::ising_oplusominus} and Lemma~\ref{lem::ising_oplusominusfree}.
\end{proof}

%% file: tex/purepartitionfunctions_pre.tex
In this section, we will prove Theorem~\ref{thm::purepartition}. 
Recall that the multiple $\SLE$ pure partition functions is the collection $\{\PartF_{\alpha}: \alpha\in\LP\}$ of positive smooth functions $\PartF_{\alpha}: \chamber_{2N}\to\R_+$ for $\alpha\in\LP_N$, satisfying $\PartF_{\emptyset}=1$, PDE~\eqref{eqn::purepartition_PDE}, COV~\eqref{eqn::purepartition_COV}, ASY~\eqref{eqn::purepartition_ASY}, and the power law bound~\eqref{eqn::purepartition_PLB}. The uniqueness part was proved in \cite{FloresKlebanPDE1} and we only discuss the existence part in this section. To state our conclusion, we need to introduce some notations and properties first. 

Fix the constants in this section: 
\[\kappa\in (0,6],\quad h=\frac{6-\kappa}{2\kappa}.\]
As the pure partition functions satisfy the conformal covariance, we can define the pure partition function in polygon via conformal image. Suppose $(\Omega; x_1, \ldots, x_{2N})$ is a polygon such that $x_1, \ldots, x_{2N}$ lie on analytic segments of $\partial\Omega$. Define, for $\alpha\in\LP_N$,
\begin{equation}\label{eqn::purepartition_COV_polygon}
\PartF_{\alpha}(\Omega; x_1, \ldots, x_{2N})=\prod_{j=1}^{2N}\varphi'(x_j)^h\times\PartF_{\alpha}(\varphi(x_1),\ldots, \varphi(x_{2N})),
\end{equation} 
where $\varphi$ is any conformal map from $\Omega$ onto $\HH$ with $\varphi(x_1)<\cdots<\varphi(x_{2N})$. 

\smallbreak
Next, we introduce the cascade relation of the pure partition functions. Consider the polygon $(\Omega; x_1, \ldots, x_{2N})$.
Suppose $\alpha=\{\{a_1, b_1\}, \ldots, \{a_{N}, b_{N}\}\}\in\LP_N$, and we may assume $a_j<b_j$ for all $1\le j\le N$.  
For $1\le k\le N$, let $\eta_k$ be an $\SLE_{\kappa}$ in $c\Omega$ from $x_{a_k}$ to $x_{b_k}$. 
Define $\LE_k$ to be the event that $\eta_k\cap (x_{a_k+1}x_{b_k-1})=\emptyset$ and $\eta_k\cap (x_{b_k+1}x_{a_k-1})=\emptyset$. 
On $\LE_k$, consider the set $\Omega\setminus\eta_k$, we denote by $D_k^R$  the connected component having $(x_{a_k+1}x_{b_k-1})$ on the boundary, and denote by $D_k^L$ the connected component having $(x_{b_k+1}x_{a_k-1})$ on the boundary, see Figure~\ref{fig::purepartition_cascade_SYM}(a).
The link $\{a_k, b_k\}$ divides the link pattern $\alpha$ into two sub-link patterns, connecting $\{a_k+1, \ldots, b_k-1\}$ and $\{b_k+1, \ldots, a_{k}-1\}$ respectively. After relabelling the indices, we denote these two link patterns by $\alpha_k^R$ and $\alpha_k^L$. We expect the following \textit{cascade relation} of the pure partition functions:
\begin{align}\label{eqn::purepartition_CAS}
&\PartF_{\alpha}(\Omega; x_1, \ldots, x_{2N})\notag\\
&=H_{\Omega}( x_{a_k}, x_{b_k})^h\E\left[\PartF_{\alpha_k^R}(D_k^R; x_{a_k+1}, \ldots, x_{b_k-1})\times\PartF_{\alpha_k^L}(D_k^L; x_{b_k+1},\ldots, x_{a_k-1})\one_{\LE_k}\right].
\end{align}

We will prove the following proposition in this section.
\begin{proposition}\label{prop::purepartition_existence}
Let $\kappa \in (0,6]$. 
For each $N\ge 1$, there exists a collection $\{\PartF_{\alpha}: \alpha\in \LP_n, n\le N\}$ of smooth
functions $\PartF_\alpha : \chamber_{2n} \to \R_+$, for $\alpha\in\LP_n$, 
satisfying the normalization $\PartF_{\emptyset}=1$, 
PDE~\eqref{eqn::purepartition_PDE}, 
COV~\eqref{eqn::purepartition_COV}, 
ASY~\eqref{eqn::purepartition_ASY},
the power law bound~\eqref{eqn::purepartition_PLB}, 
and the cascade relation~\eqref{eqn::purepartition_CAS}.
\end{proposition}

This proposition is different from Theorem~\ref{thm::purepartition}. In fact, S. Flores and P. Kleban proved in \cite{FloresKlebanPDE1} that the collection of smooth functions $\{\PartF_{\alpha}: \alpha\in \LP\}$ satisfying the normalization
$\PartF_\emptyset = 1$ and PDE~\eqref{eqn::purepartition_PDE}, COV~\eqref{eqn::purepartition_COV}, ASY~\eqref{eqn::purepartition_ASY}
and the power law bound~\eqref{eqn::purepartition_PLB} is unique (if exists). 
In Proposition~\ref{prop::purepartition_existence}, we also require the cascade relation in the statement. In fact, the cascade relation plays an essential role during our proof. 

\smallbreak
Before we proceed, we collect some basic properties here. With the general definition of $\PartF_{\alpha}$ in~\eqref{eqn::purepartition_COV_polygon}, we can rewrite ASY~\eqref{eqn::purepartition_ASY} as follows: for all $\alpha \in \LP_N$ and for all $j \in \{1, \ldots, 2N\}$ and $\xi \in (x_{j-1}x_{j+2})$, 
\begin{align}\label{eqn::purepartition_ASY_polygon}
\lim_{x_j , x_{j+1} \to \xi} 
\frac{\PartF_\alpha(\Omega; x_1 , \ldots , x_{2N})}{H_{\Omega}(x_j, x_{j+1})^h} 
=\begin{cases}
0 \quad &
    \text{if } \{j,j+1\} \notin \alpha \\
\PartF_{\hat{\alpha}}(\Omega; x_{1},\ldots,x_{j-1},x_{j+2},\ldots,x_{2N}) &
    \text{if } \{j,j+1\} \in \alpha
\end{cases}
\end{align}
where
$\hat{\alpha} = \alpha \removal\{j,j+1\} \in \LP_{N-1}$.

For convenience, we define the power law bound function: for $\alpha=\{\{a_1, b_1\}, \ldots, \{a_N, b_N\}\}\in\LP_N$, 
\[\LB_{\alpha}(\Omega; x_1, \ldots, x_{2N})=\prod_{j=1}^N H_{\Omega}(x_{a_j}, x_{b_j})^h,\]
where $H_{\Omega}(x,y)$ is the boundary Poisson kernel, then the power law bound~\eqref{eqn::purepartition_PLB} can be written as follows:
\begin{equation}\label{eqn::purepartition_PLB_polygon}
0<\PartF_{\alpha}(\Omega; x_1, \ldots, x_{2N})\le \LB_{\alpha}(\Omega; x_1, \ldots, x_{2N}).
\end{equation}

The boundary Poisson kernel has monotonicity: suppose $(\Omega; x, y)$ is a Dobrushin domain such that $x,y$ lie on analytic segments of $\partial\Omega$, and suppose $U\subset\Omega$ is simply connected and agrees with $\Omega$ is neighborhoods of $x$ and $y$, then 
$H_{U}(x,y)\le H_{\Omega}(x,y)$.
As a consequence, we have the monotonicity of $\LB_{\alpha}$: suppose $(\Omega; x_1, \ldots, x_{2N})$ is a polygon such that $x_1, \ldots, x_{2N}$ lie on analytic segments of $\partial\Omega$, and suppose $U\subset\Omega$ is simply connected and agrees with $\Omega$ in neighborhoods of $\{x_1, \ldots, x_{2N}\}$, then, for any $\alpha\in\LP_N$,
\begin{equation}\label{eqn::balpha_mono}
\LB_{\alpha}(U; x_1, \ldots, x_{2N})\le \LB_{\alpha}(\Omega; x_1, \ldots, x_{2N}),
\end{equation}
because $h\ge 0$ when $\kappa\le 6$.

%% file: tex/purepartitionfunctions_proof.tex
We will prove the existence in Proposition~\ref{prop::purepartition_existence} by induction on $N$. It is immediate to check the existence for $N=1$ and $N=2$. 
When $N=1$, for $x<y$ and  
$\vcenter{\hbox{\includegraphics[scale=0.3]{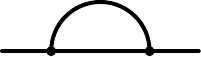}}} = \{\{1,2\}\}$,
\[ \PartF_{\vcenter{\hbox{\includegraphics[scale=0.2]{figures/link-0.pdf}}}}(x,y) 
= (y-x)^{-2h} . \]
When $N=2$, we obtain 
for $\vcenter{\hbox{\includegraphics[scale=0.3]{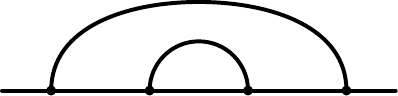}}} = \{\{1,4\}, \{2,3\} \}$ 
and $\vcenter{\hbox{\includegraphics[scale=0.3]{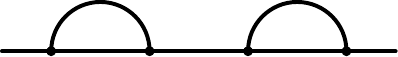}}} = \{\{1,2\}, \{3,4\}\}$, 
and for $x_1<x_2<x_3<x_4$,
\begin{align*}
\PartF_{\vcenter{\hbox{\includegraphics[scale=0.2]{figures/link-2.pdf}}}} (x_1,x_2,x_3,x_4)
= \; & (x_4-x_1)^{-2h}(x_3-x_2)^{-2h}z^{2/\kappa}F(z) , \\ 
\PartF_{\vcenter{\hbox{\includegraphics[scale=0.2]{figures/link-1.pdf}}}} (x_1,x_2,x_3,x_4)
= \; &  (x_2-x_1)^{-2h}(x_4-x_3)^{-2h}(1-z)^{2/\kappa}F(1-z),
\end{align*}
where $z$ is the cross-ratio and $F$ is the hypergeometric function as in~\eqref{eqn::hSLE_partition}:
\[z=\frac{(x_2-x_1)(x_4-x_3)}{(x_4-x_2)(x_3-x_1)},\quad F(z):=\hF\left(\frac{4}{\kappa}, 1-\frac{4}{\kappa}, \frac{8}{\kappa}; z\right).\]

Suppose the collection of pure partition functions exists up to $N$, and we consider $\LP_{N+1}$. Suppose $\alpha=\{\{a_1, b_1\}, \ldots, \{a_{N+1}, b_{N+1}\}\}$, and we may assume $a_j<b_j$ for all $1\le j\le N+1$. Consider the polygon $(\Omega; x_1, \ldots, x_{2N+2})$. 
For $1\le k\le N+1$, let $\eta_k$ be an $\SLE_{\kappa}$ in $\Omega$ from $x_{a_k}$ to $x_{b_k}$. We define $D_k^R, D_k^L, \alpha_k^R, \alpha_k^L$ in the same way as before: we denote by $D_k^R$  the connected component having $(x_{a_k+1}x_{b_k-1})$ on the boundary, and denote by $D_k^L$ the connected component having $(x_{b_k+1}x_{a_k-1})$ on the boundary.
The link $\{a_k, b_k\}$ divides the link pattern $\alpha$ into two sub-link patterns, connecting $\{a_k+1, \ldots, b_k-1\}$ and $\{b_k+1, \ldots, a_{k}-1\}$ respectively. After relabelling the indices, we denote these two link patterns by $\alpha_k^R$ and $\alpha_k^L$. 
As the pure partition functions exist up to $N$, the following two functions are well-defined: 
\[\PartF_{\alpha_k^R}(D_k^R; x_{a_k+1}, x_{a_k+2}, \ldots, x_{b_k-1}),\quad \PartF_{\alpha_k^L}(D_k^L; x_{b_k+1}, x_{b_k+2},\ldots, x_{a_k-1}). \] 
Then, we define 
\begin{align}\label{eqn::purepartition_cascade_def}
&\PartF_{\alpha}^{(k)}(\Omega; x_1, \ldots, x_{2N+2})\notag\\
&=H_{\Omega}( x_{a_k}, x_{b_k})^h\E\left[\PartF_{\alpha_k^R}(D_k^R; x_{a_k+1}, \ldots, x_{b_k-1})\times\PartF_{\alpha_k^L}(D_k^L; x_{b_k+1},\ldots, x_{a_k-1})\one_{\LE_k}\right].
\end{align}
When $\Omega=\HH$, we denote $\PartF_{\alpha}^{(k)}(\HH; x_1, \ldots, x_{2N+2})$ by $\PartF_{\alpha}^{(k)}(x_1, \ldots, x_{2N+2})$. 
At this point, the definition of $\PartF_{\alpha}^{(k)}$ depends on the choice of $k\in\{1, \ldots, N+1\}$, but we will show that, in fact, the function defined in~\eqref{eqn::purepartition_cascade_def} does not depend on the choice of $k$. 

\begin{lemma}\label{lem::purepartition_cascade_SYM}
Suppose Propsition~\ref{prop::purepartition_existence} holds up to $N$. The function $\PartF_{\alpha}^{(k)}(\Omega; x_1, \ldots, x_{2N+2})$ defined in~\eqref{eqn::purepartition_cascade_def} does not depend on the choice of $k$.
\end{lemma}
This lemma is the one that we need to use properties of hypergeometric SLE. We leave its proof to Section~\ref{subsec::proof_SYM}. 
Next, we show that the functions $\PartF_{\alpha}^{(k)}$ satisfy all the requirements in Proposition~\ref{prop::purepartition_existence} one by one in Lemmas~\ref{lem::purepartition_cascade_PDE_first} to~\ref{lem::purepartition_cascade_PLB}. 

\begin{lemma}\label{lem::purepartition_cascade_PDE_first}
Suppose Proposition~\ref{prop::purepartition_existence} holds up to $N$. The function $\PartF_{\alpha}^{(k)}(x_1, \ldots, x_{2N+2})$ defined in~\eqref{eqn::purepartition_cascade_def} is smooth and satisfies PDE~\eqref{eqn::purepartition_PDE} with $i=a_k$ and $i=b_k$. 
\end{lemma}
\begin{proof}
We only prove the conclusion for $i=a_k$ and the case when $i=b_k$ can be proved similarly as $\SLE$ is reversible. 

Recall that, in the definition of $\PartF_{\alpha}^{(k)}$, the curve $\eta_k$ is an $\SLE_{\kappa}$ in $\HH$ from $x_{a_k}(=x_i)$ to $x_{b_k}$. We parametrize $\eta_k$ by the half-plane capacity and denote by $(g_t, t\ge 0)$ the corresponding conformal maps in the Loewner chain. Let us calculate the conditional expectation $\E\left[\PartF_{\alpha_k^R}(D_k^R)\times\PartF_{\alpha_k^L}(D_k^L)\cond \eta_k[0,t]\right]$ for small $t>0$. By the conformal invariance of $\PartF_{\alpha_k^L}$ and $\PartF_{\alpha_k^R}$ in the hypothesis, we have
\begin{align*}
&\E\left[\PartF_{\alpha_k^R}(D_k^R; x_{i+1}, \ldots, x_{b_k-1})\times\PartF_{\alpha_k^L}(D_k^L; x_{b_k+1},\ldots, x_{i-1})\one_{\LE_k}\cond \eta_k[0,t]\right]\\
&=\prod_{j\neq i, b_k}g_t'(x_j)^h
\times\E\left[\PartF_{\alpha_k^R}\left(g_t(D_k^R); g_t(x_{i+1}), \ldots, g_t(x_{b_k-1})\right)\times\PartF_{\alpha_k^L}\left(g_t(D_k^L); g_t(x_{b_k+1}),\ldots, g_t(x_{i-1})\right)\one_{\LE_k}\cond \eta_k[0,t]\right]\\
&=\prod_{j\neq i, b_k}g_t'(x_j)^h\times\frac{\PartF_{\alpha}^{(k)}(g_t(x_1),\ldots, g_t(x_{i-1}), W_t, g_t(x_{i+1}), \ldots, g_t(x_{2N+2}))}{(g_t(x_{b_k})-W_t)^{-2h}}\\
&=\prod_{j\neq i}g_t'(x_j)^h\times\PartF_{\alpha}^{(k)}(g_t(x_1),\ldots, g_t(x_{i-1}), W_t, g_t(x_{i+1}), \ldots, g_t(x_{2N+2}))/N_t,
\end{align*}
where 
$N_t=g_t'(x_{b_k})^h(g_t(x_{b_k})-W_t)^{-2h}$.

Suppose $\gamma$ is an $\SLE_{\kappa}$ from $x_i$ to $\infty$. 
By Lemmas~\ref{lem::sle_kapparho_targetchanging} and~\ref{lem::sle_kapparho_mart}, the law of $\eta_k$ is the same as 
the law of $\gamma$ weighted by $N_t$. Therefore, the following functions is a local martingale for $\gamma$: 
\[\prod_{j\neq i}g_t'(x_j)^h\times\PartF_{\alpha}^{(k)}(g_t(x_1),\ldots, g_t(x_{i-1}), W_t, g_t(x_{i+1}), \ldots, g_t(x_{2N+2})).\]
By It\^{o}'s formula, the function $\PartF_{\alpha}^{(k)}$ satisfies the PDE~\eqref{eqn::purepartition_PDE} with $i=a_k$ in the distribution sense. By \cite[Lemma~5]{DubedatSLEVirasoroLocalization} (see also \cite[Proposition~2.5]{PeltolaWuGlobalMultipleSLEs}), the operator 
\[\frac{\kappa}{2}\partial^2_i + \sum_{j\neq i}\left(\frac{2}{x_{j}-x_{i}}\partial_j - 
\frac{2h}{(x_{j}-x_{i})^{2}}\right)\]
in PDE~\eqref{eqn::purepartition_PDE} is hypoelliptic. 
Therefore, the function $\PartF_{\alpha}^{(k)}$ is a smooth solution to the PDE~\eqref{eqn::purepartition_PDE} with $i=a_k$. 
\end{proof}

\begin{lemma}\label{lem::purepartition_cascade_PDE}
Suppose Propsition~\ref{prop::purepartition_existence} holds up to $N$. The function $\PartF_{\alpha}^{(k)}(x_1, \ldots, x_{2N+2})$ defined in~\eqref{eqn::purepartition_cascade_def} is smooth and satisfies the PDE system~\eqref{eqn::purepartition_PDE} of $2N+2$ partial differential equations. 
\end{lemma}
\begin{proof}
In Lemma~\ref{lem::purepartition_cascade_PDE_first}, we have shown that $\PartF_{\alpha}^{(k)}$ satisfies PDE~\eqref{eqn::purepartition_PDE} with $i=a_k, b_k$. By Lemma~\ref{lem::purepartition_cascade_SYM}, we know that $\PartF_{\alpha}^{(k)}=\PartF_{\alpha}^{(n)}$ for any $n\neq k$, combining with Lemma~\ref{lem::purepartition_cascade_PDE_first}, we know that $\PartF_{\alpha}^{(k)}=\PartF_{\alpha}^{(n)}$ also satisfies PDE~\eqref{eqn::purepartition_PDE} with $i=a_n, b_n$. This completes the proof. 
\end{proof}

\begin{lemma}\label{lem::purepartition_cascade_COV}
Suppose Propsition~\ref{prop::purepartition_existence} holds up to $N$. The function $\PartF_{\alpha}^{(k)}(x_1, \ldots, x_{2N+2})$ defined in~\eqref{eqn::purepartition_cascade_def} satisfies COV~\eqref{eqn::purepartition_COV}.
\end{lemma}
\begin{proof}
This is true because: (a) $\SLE_{\kappa}$ is conformally invariant; (b) the boundary Poisson kernel is conformally covariant; (c) the pure partition functions $\PartF_{\alpha_k^R}$ and $\PartF_{\alpha_k^L}$ are conformally covariant by the hypothesis. 
\end{proof}

\begin{lemma}\label{lem::purepartition_cascade_ASY}
Suppose Propsition~\ref{prop::purepartition_existence} holds up to $N$. The function $\PartF_{\alpha}^{(k)}(\Omega; x_1, \ldots, x_{2N+2})$ defined in~\eqref{eqn::purepartition_cascade_def} satisfies ASY~\eqref{eqn::purepartition_ASY_polygon}. 
\end{lemma}
\begin{proof}
In order to prove ASY~\eqref{eqn::purepartition_ASY_polygon}, we need to check the following cases: Case (a). $\{a_k, b_k\}=\{j, j+1\}$; Case (b). $a_k=j$ and $b_k\neq j+1$. The cases $\#\{a_k, b_k\}\cap\{j,j+1\}=1$ can be proved similarly; Case (c). $\{a_k, b_k\}\cap\{j, j+1\}=\emptyset$. 

\smallbreak
Case (a). Suppose $\{a_k, b_k\}=\{j, j+1\}$. Note that $\eta_k$ is the $\SLE_{\kappa}$ in $\HH$ from $x_j$ to $x_{j+1}$. In this case, $\alpha_k^R=\emptyset$ and $\alpha_k^L=\hat{\alpha}:=\alpha\removal\{j,j+1\}$. 
Then we have 
\begin{align*}
\PartF_{\alpha}^{(k)}(\Omega; x_1, \ldots, x_{2N+2})=H_{\Omega}(x_{j}, x_{j+1})^h\E\left[\PartF_{\hat{\alpha}}(D_k^L; x_{j+2}, \ldots, x_{j-1})\one_{\LE_k}\right].
\end{align*}
By the power law bound in the hypothesis and~\eqref{eqn::balpha_mono}, we have
\[\PartF_{\hat{\alpha}}(D_k^L; x_{j+2}, \ldots, x_{j-1})
\le \LB_{\hat{\alpha}}(D_k^L; x_{j+2}, \ldots, x_{j-1})
\le \LB_{\hat{\alpha}}(\Omega; x_{j+2}, \ldots, x_{j-1}).\]
Bounded convergence theorem gives 
\begin{align*}
\lim_{x_j, x_{j+1}\to \xi}\frac{\PartF_{\alpha}^{(k)}(\Omega; x_1, \ldots, x_{2N+2})}{H_{\Omega}(x_{j}, x_{j+1})^h}=\lim_{x_j, x_{j+1}\to \xi}\E\left[\PartF_{\hat{\alpha}}(D_k^L; x_{j+2}, \ldots, x_{j-1})\one_{\LE_k}\right]=\PartF_{\hat{\alpha}}(\Omega; x_{j+2}, \ldots, x_{j-1}).
\end{align*}

\smallbreak
Case (b). $a_k=j$ and $b_k\neq j+1$. In this case, we have 
\begin{align*}
&\frac{\PartF_{\alpha}^{(k)}(\Omega; x_1, \ldots, x_{2N+2})}{H_{\Omega}(x_j, x_{j+1})^h}\\
&=H_{\Omega}( x_{a_k}, x_{b_k})^h\E\left[\frac{\PartF_{\alpha_k^R}(D_k^R; x_{a_k+1}, \ldots, x_{b_k-1})}{H_{\Omega}(x_j, x_{j+1})^h}\times\PartF_{\alpha_k^L}(D_k^L; x_{b_k+1},\ldots, x_{a_k-1})\one_{\LE_k}\right].
\end{align*} 
By the power law bound in the hypothesis and~\eqref{eqn::balpha_mono}, we have 
\begin{align*}
\PartF_{\alpha_k^L}(D_k^L; x_{b_k+1},\ldots, x_{a_k-1})
&\le \LB_{\alpha_k^L}(\Omega; x_{b_k+1},\ldots, x_{a_k-1});\\
\frac{\PartF_{\alpha_k^R}(D_k^R; x_{a_k+1}, \ldots, x_{b_k-1})}{H_{\Omega}(x_j, x_{j+1})^h}
&\le \frac{\LB_{\alpha_k^R}(\Omega; x_{a_k+1}, \ldots, x_{b_k-1})}{H_{\Omega}(x_j, x_{j+1})^h}\to 0,\quad\text{as }x_j, x_{j+1}\to \xi. 
\end{align*}
Thus,
\[\lim_{x_j, x_{j+1}\to \xi}\frac{\PartF_{\alpha}^{(k)}(\Omega; x_1, \ldots, x_{2N+2})}{H_{\Omega}(x_j, x_{j+1})^h}=0.\]

\smallbreak
Case (c). $\{a_k, b_k\}\cap\{j, j+1\}=\emptyset$. We may assume $a_k<j<j+1<b_k$. In this case, we have 
\begin{align*}
&\frac{\PartF_{\alpha}^{(k)}(\Omega; x_1, \ldots, x_{2N+2})}{H_{\Omega}(x_j, x_{j+1})^h}\\
&=H_{\Omega}( x_{a_k}, x_{b_k})^h\E\left[\frac{\PartF_{\alpha_k^R}(D_k^R; x_{a_k+1}, \ldots, x_{b_k-1})}{H_{\Omega}(x_j, x_{j+1})^h}\times\PartF_{\alpha_k^L}(D_k^L; x_{b_k+1},\ldots, x_{a_k-1})\one_{\LE_k}\right].
\end{align*} 
By the power law bound in the hypothesis and~\eqref{eqn::balpha_mono}, we have
\begin{align*}
\PartF_{\alpha_k^L}(D_k^L; x_{b_k+1},\ldots, x_{a_k-1})
&\le \LB_{\alpha_k^L}(\Omega; x_{b_k+1},\ldots, x_{a_k-1});\\
\frac{\PartF_{\alpha_k^R}(D_k^R; x_{a_k+1}, \ldots, x_{b_k-1})}{H_{\Omega}(x_j, x_{j+1})^h}
&\le \frac{\LB_{\alpha_k^R}(\Omega; x_{a_k+1}, \ldots, x_{b_k-1})}{H_{\Omega}(x_j, x_{j+1})^h}.
\end{align*}
If $\{j,j+1\}\not\in\alpha$, then we have 
\[\frac{\LB_{\alpha_k^R}(\Omega; x_{a_k+1}, \ldots, x_{b_k-1})}{H_{\Omega}(x_j, x_{j+1})^h}\to 0,\quad \text{as }x_j, x_{j+1}\to \xi.\]
Thus 
\[\lim_{x_j,x_{j+1}\to \xi}\frac{\PartF_{\alpha}^{(k)}(\Omega; x_1, \ldots, x_{2N+2})}{H_{\Omega}(x_j, x_{j+1})^h}=0.\]
If $\{j, j+1\}\in\alpha$, then we have $\{j, j+1\}\in\alpha_k^{R}$, denote by $\hat{\alpha}=\alpha\removal\{j, j+1\}$ and $\hat{\alpha}_k^R=\alpha_k^R\removal\{j,j+1\}$. We have
\[\frac{\PartF_{\alpha_k^R}(D_k^R; x_{a_k+1}, \ldots, x_{b_k-1})}{H_{\Omega}(x_j, x_{j+1})^h}
\le \frac{\LB_{\alpha_k^R}(\Omega; x_{a_k+1}, \ldots, x_{b_k-1})}{H_{\Omega}(x_j, x_{j+1})^h}=\LB_{\hat{\alpha}_k^R}(\Omega; x_{a_k+1}, \ldots, x_{j-1}, x_{j+2}, \ldots, x_{b_k-1}).
\]
By the asymptotic in the hypothesis, we have almost surely on $\LE_k$
\[\lim_{x_j,x_{j+1}\to \xi}\frac{\PartF_{\alpha_k^R}(D_k^R; x_{a_k+1}, \ldots, x_{b_k-1})}{H_{\Omega}(x_j, x_{j+1})^h}=\PartF_{\hat{\alpha}_k^R}(D_k^R; x_{a_k+1}, \ldots, x_{j-1}, x_{j+2}, \ldots, x_{b_k-1} ).\]
Bounded convergence theorem and the cascade relation in the hypothesis give
\begin{align*}
&\lim_{x_j,x_{j+1}\to \xi}\frac{\PartF_{\alpha}^{(k)}(\Omega; x_1, \ldots, x_{2N+2})}{H_{\Omega}(x_j, x_{j+1})^h}\\
&=H_{\Omega}( x_{a_k}, x_{b_k})^h
\E\left[\PartF_{\hat{\alpha}_k^R}(D_k^R; x_{a_k+1}, \ldots, x_{j-1}, x_{j+2}, \ldots, x_{b_k-1} )\times\PartF_{\alpha_k^L}(D_k^L; x_{b_k+1},\ldots, x_{a_k-1})\one_{\LE_k}\right]\\
&=\PartF_{\hat{\alpha}}(\Omega; x_1, \ldots, x_{j-1}, x_{j+2}, \ldots, x_{2N}).
\end{align*} 
This completes the proof of Case (c) and hence completes the proof of this lemma. 
\end{proof}

\begin{lemma}\label{lem::purepartition_cascade_PLB}
Suppose Propsition~\ref{prop::purepartition_existence} holds up to $N$. The function $\PartF_{\alpha}^{(k)}(\Omega; x_1, \ldots, x_{2N+2})$ defined in~\eqref{eqn::purepartition_cascade_def} satisfies the power law bound~\eqref{eqn::purepartition_PLB_polygon}.
\end{lemma}
\begin{proof}
By the power law bound in the hypothesis and~\eqref{eqn::balpha_mono}, we have
\begin{align*}
&\PartF_{\alpha}^{(k)}(\Omega; x_1, \ldots, x_{2N+2})\\
&\le H_{\Omega}( x_{a_k}, x_{b_k})^h\times\LB_{\alpha_k^R}(\Omega; x_{a_k+1}, \ldots, x_{b_k-1})\times\LB_{\alpha_k^L}(\Omega; x_{b_k+1},\ldots, x_{a_k-1})\\
&=\LB_{\alpha}(\Omega; x_1, \ldots, x_{2N+2}). 
\end{align*}
\end{proof}

Now, we are ready to prove the conclusion.
\begin{proof}[Proof of Proposition~\ref{prop::purepartition_existence}]
We have seen that the conclusion holds for $N=1$. Suppose the conclusion holds up to $N$, for each $1\le k\le N+1$, we define $\PartF_{\alpha}^{(k)}$ as in~\eqref{eqn::purepartition_cascade_def}. By Lemma~\ref{lem::purepartition_cascade_SYM}, it does not depend on the choice of $k$, thus we denote it by $\PartF_{\alpha}$. Consider the collection $\{\PartF_{\alpha}, \alpha\in\LP_{N+1}\}$, 
it satisfies PDE~\eqref{eqn::purepartition_PDE} by 
Lemma~\ref{lem::purepartition_cascade_PDE}, 
it satisfies COV~\eqref{eqn::purepartition_COV} by 
Lemma~\ref{lem::purepartition_cascade_COV}, 
it satisfies ASY~\eqref{eqn::purepartition_ASY} by 
Lemma~\ref{lem::purepartition_cascade_ASY}, 
and it satisfies the power law bound~\eqref{eqn::purepartition_PLB} by Lemma~\ref{lem::purepartition_cascade_PLB}. 
Combining Lemma~\ref{lem::purepartition_cascade_SYM} and the expression in~\eqref{eqn::purepartition_cascade_def}, we obtain the cascade relation~\eqref{eqn::purepartition_CAS}. 
These complete the proof. 
\end{proof}

Combining the uniqueness result in \cite{FloresKlebanPDE1} and Proposition~\ref{prop::purepartition_existence}, we obtain Theorem~\ref{thm::purepartition}. Moreover, as a consequence of Proposition~\ref{prop::purepartition_existence}, we also obtain the cascade relation of the pure partition functions.
\begin{corollary}\label{cor::purepartition_CAS}
The collection of pure partition functions in Theorem~\ref{thm::purepartition} also satisfies the cascade relation~\eqref{eqn::purepartition_CAS}. 
\end{corollary}

In fact, the proof of Lemma~\ref{lem::purepartition_cascade_ASY}  implies the following refined asymptotic. We do not need this refined asymptotic in this paper, but it is very useful when one tries to prove the crossing probabilities in related models, see \cite[Section~5]{PeltolaWuGlobalMultipleSLEs}. So we record this result here. 
\begin{corollary}
The collection of pure partition functions in Theorem~\ref{thm::purepartition} also satisfies the following refined asymptotic:
for all $\alpha \in \LP_N$ and for all $j \in \{1, \ldots, 2N-1 \}$ and 
 $x_1 < x_2 < \cdots < x_{j-1} < \xi < x_{j+2} < \cdots < x_{2N}$,
\begin{align*}
\lim_{\substack{\tilde{x}_j , \tilde{x}_{j+1} \to \xi, \\ \tilde{x}_i\to x_i \text{ for } i \neq j, j+1}} 
\frac{\PartF_\alpha(\tilde{x}_1 , \ldots , \tilde{x}_{2N})}{(\tilde{x}_{j+1} - \tilde{x}_j)^{-2h}} 
=\begin{cases}
0 \quad &
    \text{if } \{j,j+1\} \notin \alpha \\
\PartF_{\hat{\alpha}}(x_{1},\ldots,x_{j-1},x_{j+2},\ldots,x_{2N}) &
    \text{if } \{j, j+1\} \in \alpha
\end{cases},
\end{align*}
where $\hat{\alpha} = \alpha \removal \{j, j+1\}$.
\end{corollary}

Finally, let us discuss the range of $\kappa$ in Theorem~\ref{thm::purepartition}. In \cite{PeltolaWuGlobalMultipleSLEs}, the existence was proved for $\kappa\le 4$. The construction there uses a Brownian loop soup construction of global multiple SLEs. Such construction only works for $\kappa\le 4$. In our construction, we use the cascade relation~\eqref{eqn::purepartition_cascade_def}. Such cascade relation is believed to be true for all $\kappa\in (0,8)$, but we need the monotonicity~\eqref{eqn::balpha_mono} in various places during the proof. Such monotonicity only holds for $\kappa\le 6$, and this explains the range $\kappa\in (0,6]$ in Theorem~\ref{thm::purepartition}. 

%% file: tex/purepartitionfunctions_SYM.tex
To show Lemma~\ref{lem::purepartition_cascade_SYM}, we need the following property of hypergeometric SLE and Proposition~\ref{prop::hypersle_mart}. 
\begin{proposition}\label{prop::hsle_symmetry}
Fix $\kappa\in (0,8)$ and a quad $q=(\Omega; x^R, y^R, y^L, x^L)$
\begin{itemize}
\item (Existence and Uniqueness)
There exists a unique probability measure on pairs of continuous curves $(\eta^L;\eta^R)\in X_0(\Omega; x^R, y^R, y^L, x^L)$ such that the conditional law of $\eta^R$ given $\eta^L$ is $\SLE_{\kappa}$ in $\Omega^L$ from $x^R$ to $y^R$; and that the conditional law of $\eta^L$ given $\eta^R$ is $\SLE_{\kappa}$ in $\Omega^R$ from $x^L$ to $y^L$. 
\item (Identification)
Under this probability measure, the marginal law of $\eta^L$ is $\hSLE_{\kappa}$ in $\Omega$ from $x^L$ to $y^L$ with marked points $(x^R, y^R)$.  
\end{itemize}
\end{proposition}
\begin{proof}
When $\kappa\le 4$, this proposition is a special case of Proposition~\ref{prop::slepair_rev} when $\rho^L=\rho^R=0$. When $\kappa\in (4,8)$, the existence and the uniqueness were argued in \cite{MillerSheffieldIG3}, see also the discussion in \cite[Section~4]{MillerWernerConnectionProbabilitiesCLE}. In \cite{BeffaraPeltolaWuUniqueness}, the authors provided another perspective for the existence and the uniqueness with $\kappa\in (4,6]$. We define \textit{global $2$-$\SLE_{\kappa}$} to be this unique probability measure. 
It remains to derive the marginal law of $\eta^L$ in global $2$-$\SLE_{\kappa}$. Such question is included in some form in previous papers:  
\cite[Section~8]{BauerBernardKytolaMultipleSLE}, 
\cite[Section~4]{DubedatEulerIntegralsCommutingSLEs}, 
and \cite[Section~4]{MillerWernerConnectionProbabilitiesCLE}. Let us briefly summarize how they derived the marginal law.

Suppose $(\eta^L;\eta^R)\in X_0(\Omega; x^R, y^R, y^L, x^L)$ is the global $2$-$\SLE_{\kappa}$. Suppose $U_1, \ldots, U_4$ are neighborhoods of the points $x^L, y^L, y^R, x^R$ respectively such that $\Omega\setminus U_j$ are simply connected and $U_j\cap U_k=\emptyset$ for $j\neq k$. 
Let $\gamma_1$ be the part of $\eta^L$ that starts from $x^L$ and ends at exiting $U_1$; 
let $\gamma_2$ be the part of $\eta^R$ that starts from $x^R$ and ends at exiting $U_2$;
let $\gamma_3$ be the part of the time-reversal of $\eta^R$ that starts from $y^R$ and ends at exiting $U_3$;
let $\gamma_4$ be the part of the time-reversal of $\eta^L$ that starts from $y^L$ and ends at exiting $U_4$.  
By the conformal invariance of the global $2$-$\SLE_{\kappa}$ and the reversibility of $\SLE_{\kappa}$, we could argue that $(\gamma_1, \ldots, \gamma_4)$ is a local $2$-$\SLE_{\kappa}$. 
Therefore, by the commutation relation in \cite{DubedatCommutationSLE} and a complete classification summarized in \cite[Theorem~A.4]{KytolaPeltolaPurePartitionFunctions}, we know that $\gamma_1$ has the law of $\hSLE_{\kappa}$. In other words, the law of $\eta^L$ restricted to $U_1$ has the law of $\hSLE_{\kappa}$. This is true for any localization neighborhoods $(U_1, \ldots, U_4)$. This implies that $\eta^L$ is an $\hSLE_{\kappa}$ up to any stopping time $\tau$ as long as $\eta^L[0,\tau]$ has positive distance from the points $\{x^R, y^R, y^L\}$.  
By Proposition~\ref{prop::hyperSLE}, $\hSLE_{\kappa}$ in $\Omega$ from $x^L$ to $y^L$ with marked points $(x^R, y^R)$ has  positive distance from the points $\{x^R, y^R\}$ and it is a continuous transient curve, thus $\eta^L$ has the law of $\hSLE_{\kappa}$ as desired. 
\end{proof}

\begin{figure}[ht!]
\begin{subfigure}[b]{0.48\textwidth}
\begin{center}
\includegraphics[width=0.7\textwidth]{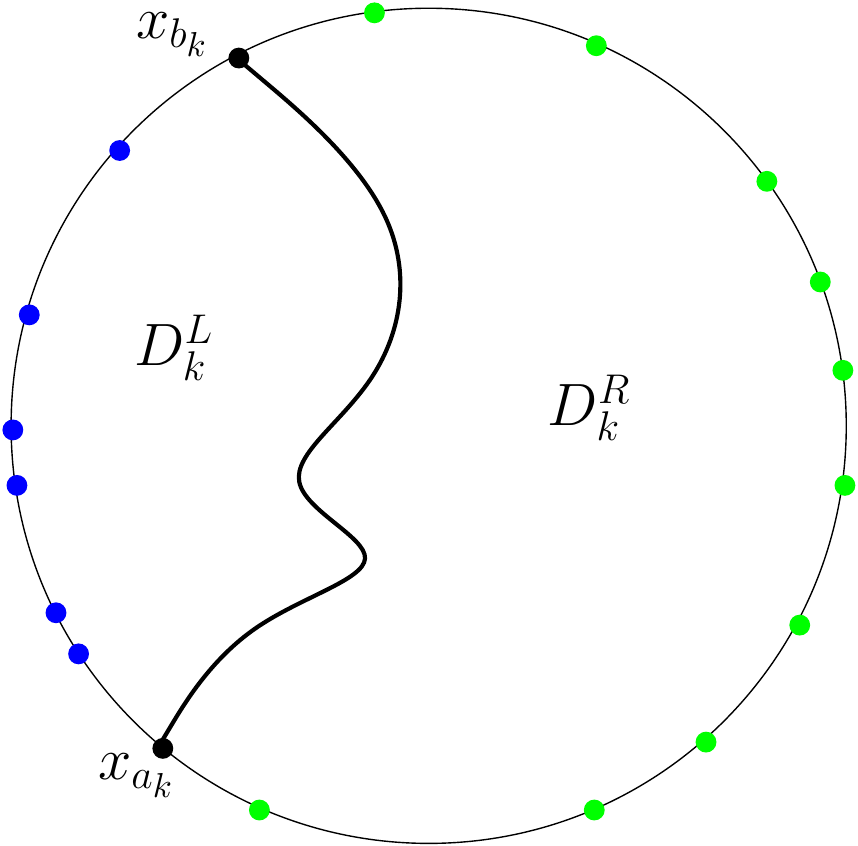}
\end{center}
\caption{The blue marked points correspond to $\alpha_k^L$ and the green marked points correspond to $\alpha_k^R$.}
\end{subfigure}
$\quad$
\begin{subfigure}[b]{0.48\textwidth}
\begin{center}\includegraphics[width=0.7\textwidth]{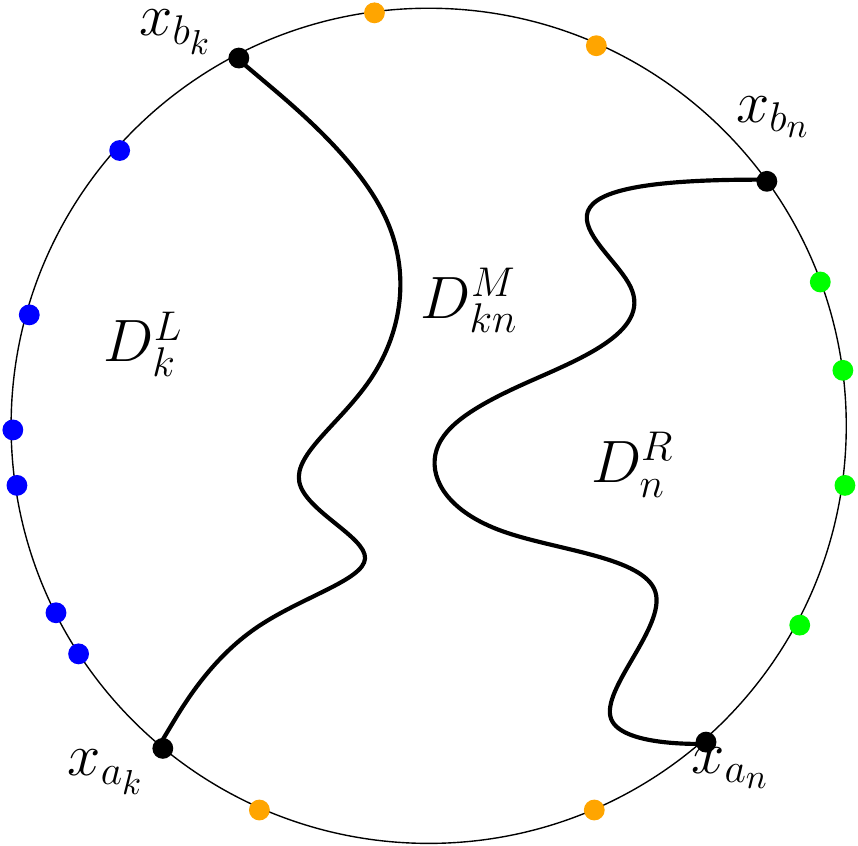}
\end{center}
\caption{The orange marked points correspond to $\alpha_{kn}^M$ and the green marked points correspond to $\alpha_{n}^R$.}
\end{subfigure}
\caption{\label{fig::purepartition_cascade_SYM} }
\end{figure}

\begin{proof}[Proof of Lemma~\ref{lem::purepartition_cascade_SYM}]
Pick $n\neq k$, we will show that $\PartF_{\alpha}^{(k)}=\PartF_{\alpha}^{(n)}$. Assume $a_k<a_n<b_n<b_k$. 
Recall that $\eta_k$ is an $\SLE_{\kappa}$ in $\Omega$ from $x_{a_k}$ to $x_{b_k}$, and $D_k^R$ is the connected component of $\Omega\setminus\eta_k$ with $(x_{a_k+1}x_{b_k-1})$ on the boundary, and $D_k^L$ is the connected component of $\Omega\setminus\eta_k$ with $(x_{b_k+1}x_{a_k-1})$ on the boundary. Recall that
\begin{align*}
&\PartF_{\alpha}^{(k)}(\Omega; x_1, \ldots, x_{2N+2})\notag\\
&=H_{\Omega}( x_{a_k}, x_{b_k})^h
\E\left[\PartF_{\alpha_k^R}(D_k^R; x_{a_k+1}, \ldots, x_{b_k-1})\times\PartF_{\alpha_k^L}(D_k^L; x_{b_k+1},\ldots, x_{a_k-1})\one_{\LE_k}\right].
\end{align*}

Let $\eta_n$ be an $\SLE_{\kappa}$ in $D_k^R$ from $x_{a_n}$ to $x_{b_n}$. 
Define $\LE_{kn}$ to be the event that $\eta_n\cap (x_{a_n+1}x_{b_n-1})=\emptyset$ and $\eta_n$ does not hit the boundary arc $(x_{b_n+1}x_{a_n-1})$ inside $D_k^R$. Note that, if $a_n=a_k+1$ or $b_n=b_k-1$, then $\eta_k\cap\eta_n\neq\emptyset$ is allowed in $\LE_{kn}$ when $\kappa\in (4,6]$. 
On $\LE_{kn}$, let $D_{n}^R$ be the connected component of $D_k^R\setminus\eta_n$ with $(x_{a_n+1}x_{b_n-1})$ on the boundary, and $D_{kn}^M$ be the connected component of $D_k^R\setminus\eta_n$ with $(x_{a_k+1}x_{a_n-1})\cup (x_{b_n+1}x_{b_k-1})$ on the boundary, see Figure~\ref{fig::purepartition_cascade_SYM}(b). The links $\{a_k, b_k\}$ and $\{a_n, b_n\}$ divide the link pattern $\alpha$ into three sub-link patterns, connecting $\{b_k+1, \ldots, a_k-1\}$, $\{a_k+1,\ldots, a_n-1, b_n+1,\ldots, b_k-1\}$, and $\{a_n+1, \ldots, b_n-1\}$ respectively. After relabelling the remaining indices, we denote these link patterns by $\alpha_k^L, \alpha_{kn}^M, \alpha_{n}^R$. 
The marked points of the domains $D_k^L, D_{kn}^M, D_{n}^R$ are clear, so we omit them from the notation. 


By the cascade relation in the hypothesis, we have
\begin{align*}
\PartF_{\alpha_k^R}(D_k^R; x_{a_k+1}, \ldots, x_{b_k-1})=H_{D_k^R}(x_{a_n}, x_{b_n})^h\E\left[\PartF_{\alpha_{n}^R}(D_{n}^R;\ldots)\times\PartF_{\alpha_{kn}^M}(D_{kn}^M; \ldots)\one_{\LE_{kn}}\right].
\end{align*}
Plugging into the definition of $\PartF_{\alpha}^{(k)}$, we have 
\begin{align*}
&\PartF_{\alpha}^{(k)}(\Omega; x_1, \ldots, x_{2N+2})\notag\\
&=H_{\Omega}( x_{a_k}, x_{b_k})^h
\E\left[H_{D_k^R}(x_{a_n}, x_{b_n})^h\PartF_{\alpha_{n}^R}(D_{n}^R;\ldots)\times\PartF_{\alpha_{kn}^M}(D_{kn}^M; \ldots)\times\PartF_{\alpha_k^L}(D_k^L; \ldots)\one_{\LE_k\cap\LE_{kn}}\right].
\end{align*}
Here $\E$ corresponds to the following probability measure: 
sample $\eta_k$ as $\SLE_{\kappa}$ in $\Omega$ from $x_{a_k}$ to $x_{b_k}$; given $\eta_k$, sample $\eta_n$ as $\SLE_{\kappa}$ in $D_k^R$ from $x_{a_n}$ to $x_{b_n}$. Note that $\LE_k\cap\LE_{kn}$ can be written as $\LE_k\cap\LE_n\cap\LF_{kn}$ where the event $\LF_{kn}$ is defined as follows: if $a_n>a_k+1$ and $b_n<b_k-1$, then $\LF_{kn}=\{\eta_k\cap\eta_n=\emptyset\}$; if not, then $\LF_{kn}$ is the event that $\eta_k$ stays to the left of $\eta_n$.

From Proposition~\ref{prop::hypersle_mart}, we know that the law of $\eta_k$ weighted by $H_{D_k^R}(x_{a_n}, x_{b_n})^h$ becomes $\hSLE_{\kappa}$ in $\Omega$ from $x_{a_k}$ to $x_{b_k}$ with marked points $(x_{a_n}, x_{b_n})$. Denote by $q=(\Omega; x_{a_n}, x_{b_n}, x_{b_k}, x_{a_k})$ and denote by $\QQ_q$ the following probability measure: 
sample $\eta_k$ as $\hSLE_{\kappa}$ in $\Omega$ from $x_{a_k}$ to $x_{b_k}$ with marked points $(x_{a_n}, x_{b_n})$; given $\eta_k$, sample $\eta_n$ as $\SLE_{\kappa}$ in $D_k^R$ from $x_{a_n}$ to $x_{b_n}$. Then we have
\begin{align}\label{eqn::pureparition_cascade_sym}
&\PartF_{\alpha}^{(k)}(\Omega; x_1, \ldots, x_{2N+2})\notag\\
&=\PartF_{\vcenter{\hbox{\includegraphics[scale=0.2]{figures/link-2.pdf}}}}(\Omega; x_{a_k}, x_{a_n}, x_{b_n}, x_{b_k})
\QQ_q\left[\PartF_{\alpha_{n}^R}(D_{n}^R;\ldots)\times\PartF_{\alpha_{kn}^M}(D_{kn}^M; \ldots)\times\PartF_{\alpha_k^L}(D_k^L; \ldots)\one_{\LE_k\cap\LE_n\cap\LF_{kn}}\right].
\end{align}

By Proposition~\ref{prop::hsle_symmetry}, we know that $\QQ_q$ is the same as the unique probability measure there. In particular, it is symmetric in $\eta_k$ and $\eta_n$. Therefore,
the function $\PartF_{\alpha}^{(n)}(\Omega; x_1, \ldots, x_{2N+2})$ can be expanded in the same way as the right hand side of~\eqref{eqn::pureparition_cascade_sym}. As a consequence,  
\[\PartF_{\alpha}^{(k)}(\Omega; x_1, \ldots, x_{2N+2})=\PartF_{\alpha}^{(n)}(\Omega; x_1, \ldots, x_{2N+2}),\]
as desired. 
\end{proof}

%% file: tex/appendix_hypergeometric_functions.tex
For $A, B, C\in \R$, the hypergeometric function is defined for $|z|<1$ by the power series:
\[F(z)=\hF(A, B, C; z)=\sum_{n=0}^{\infty}\frac{(A)_n(B)_n}{(C)_n}\frac{z^n}{n!},\]
where $(x)_n$ is the Pochhammer symbol $(x)_n:=x(x+1)\cdots(x+n-1)$ for $n\ge 1$ and $(x)_n=1$ for $n=0$.  The power series is well-defined when $C\not\in\{0, -1, -2, -3, \ldots\}$, and it is absolutely convergent on $z\in [0,1]$ when $C>A+B$. 
When $C>A+B$ and $C\not\in\{0, -1, -2, -3, \ldots\}$, we have 
\begin{equation}\label{eqn::hyperF_one}
F(1)=\frac{\Gamma(C)\Gamma(C-A-B)}{\Gamma(C-A)\Gamma(C-B)},
\end{equation}
where $\Gamma$ is Gamma Function. 
The hypergeometric function is a solution of Euler's hypergeometric differential equation
\begin{equation}\label{eqn::hyper_pde}
z(1-z)F''(z)+(C-(A+B+1)z)F'(z)-ABF(z)=0.
\end{equation}

If $C$ is not an integer, around the point $z=0$, two independent solutions are 
\[\hF(A, B, C; z),\quad z^{1-c}\hF(1+A-C, 1+B-C, 2-C; z).\]

If $C-A-B$ is not an integer, around the point $z=1$, two independent solutions are 
\[\hF(A, B, 1+A+B-C; 1-z),\quad (1-z)^{C-A-B}\hF(C-A, C-B, 1+C-A-B; 1-z).\]

\begin{lemma}\label{lem::hyperF_increasing}
When $C>0$ and $AB> 0$, the function $F(z)$ is increasing for all $z\in [0,1)$. 
\end{lemma}
\begin{proof}
We have $F(0)=1$ and $F'(0)=AB/C> 0$. If the conclusion is false, then there exists $z_0\in (0,1)$ such that $F$ is increasing for $z\in (0, z_0)$ and is decreasing for $z\in (z_0, z_0+\eps)$ for some $\eps>0$. This implies that $z_0$ is a local maximum and thus $F(z_0)\ge 1$, $F'(z_0)=0$ and $F''(z_0)< 0$ (note that $F''(z_0)\neq 0$ since the zero points of $F'$ are simple). However, by \eqref{eqn::hyper_pde}, we have $z_0(1-z_0)F''(z_0)=ABF(z_0)$, contradiction. 
\end{proof}